\documentclass{tac}
\usepackage{amsmath}
\usepackage[normalem]{ulem}
\numberwithin{equation}{section}
\usepackage{array,longtable}
\usepackage{caption}
\usepackage{subcaption}
\usepackage{setspace}
\newcolumntype{C}{>{$}c<{$}} 
\newcolumntype{L}{{$}l{$}} 

\usepackage{enumitem}
\setlist[enumerate, 1]{
  leftmargin = \parindent, 
  align = left,
  labelwidth=\parindent,
  labelsep = 0pt
}

\usepackage[all, cmtip]{xy}

\input diagxy

\usepackage[colorlinks=true]{hyperref}
\hypersetup{allcolors=[rgb]{0.1,0.1,0.4}}

\author{Victoria Noquez, Lawrence S.~Moss}

\address{Department of Mathematics \& Computer Science, St. Mary's College of California\\
1928 St. Mary's Rd., Moraga, CA 94575\\
\\
Department of Mathematics, Indiana University\\
Rawles Hall, 831 East 3rd St., Bloomington, IN 47405-7106\\
}

\title{The Sierpinski carpet as a final coalgebra}

\copyrightyear{2021}


\eaddress{vln1@stmarys-ca.edu \CR lmoss@indiana.edu}

\usepackage{tikz}
\usetikzlibrary{cd,arrows,shapes,automata,backgrounds,petri}
\usetikzlibrary{snakes}
\usetikzlibrary{decorations.shapes}
\usepackage{tikz-cd}

\newcommand{\norm}[1]{|#1|}

\newcommand{\onehalf}{\tfrac{1}{2}}
\newcommand{\onethird}{\tfrac{1}{3}}
\newcommand{\twothirds}{\tfrac{2}{3}}

\newcommand{\rthirds}{\tfrac{r}{3}}
\newcommand{\oneoverK}{\tfrac{1}{K}}

\renewcommand{\o}{\circ}

\thirdleveltheorems 
\newcommand{\reals}{\mathbbm{R}}

\newcommand{\mbar}{\overline{m}}

\newtheorem{theorem}{Theorem}
\newtheorem{proposition}[theorem]{Proposition}
\newtheorem{lemma}[theorem]{Lemma}
\newtheorem{corollary}[theorem]{Corollary}

\newtheorem{claim}{Claim}
\theoremstyle{definition}
\newtheoremrm{definition}{Definition}
\newtheoremrm{remark}{Remark} 
\newtheoremrm{remarks}{Remarks} 

\newtheoremrm{example}{Example}

\newcommand{\C}{\mathcal{C}}
\newcommand{\CC}{\mathrm{Com}}
\newcommand{\A}{\mathcal{A}}

\renewcommand{\AA}{\mathcal{A}}
\newcommand{\UU}{\mathcal{U}}

\newcommand{\shrink}{\mbox{\sf shrink}}

\newcommand{\SC}{\mathbbm{S}}
\newcommand{\set}[1]{\{ #1 \}}

\newcommand{\rem}[1]{\relax}
\newcommand{\Set}{\mbox{\sf Set}}
\newcommand{\Met}{\mbox{\sf Met}}
\newcommand{\CMS}{\mbox{\sf CMS}}
\newcommand{\SquaSet}{\mbox{\sf SquaSet}}
\newcommand{\SquaMS}{\mbox{\sf SquaMS}}
\newcommand{\Taxicab}{\mbox{\scriptsize\sf Taxi}}
\newcommand{\Euc}{\mbox{\scriptsize\sf Euc}}

\newcommand{\MS}{\mbox{\sf MS}}
\newcommand{\PS}{\mbox{\sf Pseu}}
\newcommand{\eps}{\varepsilon}
\newcommand{\AlgF}{\mbox{\sf Alg}\, F}

\newcommand{\id}{\mbox{\sf id}}

\newcommand{\sqone}{\mbox{\sc sq}_1}
\newcommand{\sqtwo}{\mbox{\sc sq}_2}

\newcommand{\sconeBprime}{
\begin{tikzpicture}[scale=1.333]
\foreach \x in {0,  1}
{
\draw (\x,0) -- (\x,1);
\draw (0,\x) --(1,\x);
};
\end{tikzpicture}
}

\newcommand{\sctwoBprime}[1]{
\begin{tikzpicture}[scale=4]
\draw ({1/6},{1/6}) node {#1};
\draw ({1/6},{1/2}) node  {#1};
\draw ({1/6},{5/6}) node {#1};
\draw ({1/2},{1/6}) node  {#1};
\draw ({1/2},{5/6}) node  {#1};
\draw ({5/6},{1/6}) node  {#1};
\draw ({5/6},{1/2}) node {#1};
\draw ({5/6},{5/6}) node  {#1};
\end{tikzpicture}
}

\usepackage{bbm}
\usepackage{stackrel}

\newcommand{\secfoura}{
  \begin{tikzpicture}[>=stealth',shorten >=1pt,auto,node distance=2cm,semithick,scale=1.5]
   \draw [help lines] (0,0) grid (3,3);   
      \draw (1.9,1)  node (m4)   {$\bullet$}; 
      \draw (.2,0.4)  node  (m1)   {$\bullet$};         \draw (.2,.2) node{\protect{\scriptsize $x$}};
      \draw (.7,1.75) node{\protect{\scriptsize $A$}};    
\draw (.2,1) node (m2) {$\bullet$};
   \draw (1,1.7) node (m3) {$\bullet$}; 
      \draw (1,0.7) node (m5) {$\bullet$}; 
            \draw (.5,0.6) node (m6) {$\bullet$};    \draw (.5,.4) node{\protect{\scriptsize $y$}};
\draw[line width=2pt] (.2,0.4) -- (.2,1);
\draw[line width=2pt] (.2,1) -- (1,1.7);
\draw[line width=2pt] (1,1.7) -- (1.9,1);
\draw[line width=2pt] (1.9,1) -- (1,0.7);
\draw[line width=2pt] (1,0.7) -- (.5,0.6);
\end{tikzpicture}
}

\newcommand{\secfourb}{
  \begin{tikzpicture}[>=stealth',shorten >=1pt,auto,node distance=2cm,semithick,scale=1.5]
   \draw [help lines] (0,0) grid (3,3);   
      \draw (1.9,1)  node (m4)   {$\bullet$}; 
      \draw (.2,0.4)  node  (m1)   {$\bullet$};      
      \draw (.2,.2) node{\protect{\scriptsize $x$}};
      \draw (1,1.25) node{\protect{\scriptsize $A$}};\ 
\draw (.2,1) node (m2) {$\bullet$};
   \draw (1,1) node (m3) {$\bullet$}; 
      \draw (1,0.7) node (m5) {$\bullet$}; 
            \draw (.5,0.6) node (m6) {$\bullet$};    \draw (.5,.4) node{\protect{\scriptsize $y$}};
\draw[line width=2pt] (.2,0.4) -- (.2,1);
\draw[line width=2pt] (.2,1) -- (1,1);
\draw[line width=2pt] (1,1) -- (1.9,1);
\draw[line width=2pt] (1.9,1) -- (1,0.7);
\draw[line width=2pt] (1,0.7) -- (.5,0.6);
\end{tikzpicture}
}

\newcommand{\secfourc}{
  \begin{tikzpicture}[>=stealth',shorten >=1pt,auto,node distance=2cm,semithick,scale=1.5]
   \draw [help lines] (0,0) grid (3,3);   
      \draw (2.9,1)  node (m4)   {$\bullet$}; 
      \draw (1.2,0.4)  node  (m1)   {$\bullet$};         \draw (1.2,.2) node{\protect{\scriptsize $x$}};
      \draw (1.7,1.75) node{\protect{\scriptsize $A$}}; 
\draw (1.2,1) node (m2) {$\bullet$};
   \draw (2,1.7) node (m3) {$\bullet$}; 
      \draw (2,0.7) node (m5) {$\bullet$}; 
            \draw (1.5,0.6) node (m6) {$\bullet$};    \draw (1.5,.4) node{\protect{\scriptsize $y$}};
\draw[line width=2pt] (1.2,0.4) -- (1.2,1);
\draw[line width=2pt] (1.2,1) -- (2,1.7);
\draw[line width=2pt] (2,1.7) -- (2.9,1);
\draw[line width=2pt] (2.9,1) -- (2,0.7);
\draw[line width=2pt] (2,0.7) -- (1.5,0.6);
\end{tikzpicture}
}

\newcommand{\secfourd}{
  \begin{tikzpicture}[>=stealth',shorten >=1pt,auto,node distance=2cm,semithick,scale=1.5]
   \draw [help lines] (0,0) grid (3,3);   
      \draw (1.2,0.4)  node  (m1)   {$\bullet$};         \draw (1.2,.2) node{\protect{\scriptsize $x$}};
\draw (.4,1) node (m2) {$\bullet$};
   \draw (1,1.5) node (m3) {$\bullet$};       
      \draw (.8,1.75) node{\protect{\scriptsize $A$}}; 
   \draw (2,1.5) node (m4) {$\bullet$};      \draw (2.2,1.75) node{\protect{\scriptsize $B$}}; 
      \draw (2.6,1) node (m5) {$\bullet$}; 
      \draw (2, .3) node (m6) {$\bullet$};       
            \draw (1.5,0.8) node (m7) {$\bullet$};    \draw (1.5,.6) node{\protect{\scriptsize $y$}};
\draw[line width=2pt] (1.2,0.4) -- (.4,1);
\draw[line width=2pt] (.4,1) -- (1,1.5);
\draw[line width=2pt] (1,1.5) -- (2,1.5);
\draw[line width=2pt] (2,1.5) -- (2.6,1);
\draw[line width=2pt] (2.6,1) -- (2, .3);
\draw[line width=2pt] (2, .3) -- (1.5,0.8);
\end{tikzpicture}
}

\newcommand{\secfoure}{
  \begin{tikzpicture}[>=stealth',shorten >=1pt,auto,node distance=2cm,semithick,scale=1.5]
   \draw [help lines] (0,0) grid (3,3);   
      \draw (1.2,0.4)  node  (m1)   {$\bullet$};         \draw (1.2,.2) node{\protect{\scriptsize $x$}};
\draw (.4,1) node (m2) {$\bullet$};
   \draw (1,1) node (m3) {$\bullet$};       
      \draw (.8,1.25) node{\protect{\scriptsize $A$}};
   \draw (2,1.) node (m4) {$\bullet$};      \draw (2.2,1.25) node{\protect{\scriptsize $B$}}; 
      \draw (2.6,1) node (m5) {$\bullet$}; 
      \draw (2, .3) node (m6) {$\bullet$};       
            \draw (1.5,0.8) node (m7) {$\bullet$};    \draw (1.5,.6) node{\protect{\scriptsize $y$}};
\draw[line width=2pt] (1.2,0.4) -- (.4,1);
\draw[line width=2pt] (.4,1) -- (1,1);
\draw[line width=2pt] (1,1) -- (2,1.);
\draw[line width=2pt] (2,1.) -- (2.6,1);
\draw[line width=2pt] (2.6,1) -- (2, .3);
\draw[line width=2pt] (2, .3) -- (1.5,0.8);
\end{tikzpicture}
}

\newcommand{\secfourf}{
  \begin{tikzpicture}[>=stealth',shorten >=1pt,auto,node distance=2cm,semithick,scale=1.5]
   \draw [help lines] (0,0) grid (3,3);   
      \draw (1.3,1.5)  node  (m1)   {$\bullet$};         \draw (1.3,1.3) node{\protect{\scriptsize $x$}};
\draw (1.2,2) node (m2) {$\bullet$};
   \draw (1,2.7) node (m3) {$\bullet$}; 
         \draw (.5,2)  node (m4)   {$\bullet$}; 
      \draw (.5,1) node (m5) {$\bullet$}; 
            \draw (1,.3) node (m6) {$\bullet$};   
               \draw (2,.3) node (m7) {$\bullet$};            
         \draw (2.5,1) node (m8) {$\bullet$};          
                  \draw (2,1.4) node (m9) {$\bullet$};        
                    \draw (1.6,1.3) node (m11) {$\bullet$};    
             \draw (1.6,1.1) node{\protect{\scriptsize $y$}};
\draw[line width=2pt] (1.3,1.5)--  (1.2,2);
\draw[line width=2pt]  (1.2,2) -- (1,2.7);
\draw[line width=2pt] (1,2.7) -- (.5,2);
\draw[line width=2pt] (.5,2) -- (.5,1);
\draw[line width=2pt] (.5,1) -- (1,.3);
\draw[line width=2pt] (1,.3) -- (2,.3);
\draw[line width=2pt] (2,.3) -- (2.5,1);
\draw[line width=2pt] (2.5,1) -- (2,1.4);
\draw[line width=2pt] (2,1.4) -- (1.6,1.3);
\end{tikzpicture}
}

\newcommand{\secfourg}{
  \begin{tikzpicture}[>=stealth',shorten >=1pt,auto,node distance=2cm,semithick,scale=1.5]
   \draw [help lines] (0,0) grid (3,3);   
      \draw (2.9,2)  node (m4)   {$\bullet$}; 
      \draw (1.2,1.4)  node  (m1)   {$\bullet$};         \draw (1.2,1.2) node{\protect{\scriptsize $x$}};
      \draw (1.7,2.75) node{\protect{\scriptsize $A$}}; 
\draw (1.2,2) node (m2) {$\bullet$};
   \draw (2,2.7) node (m3) {$\bullet$}; 
      \draw (2,01.7) node (m5) {$\bullet$}; 
            \draw (1.5,1.6) node (m6) {$\bullet$};    
            \draw (1.5,1.4) node{\protect{\scriptsize $y$}};
\draw[line width=2pt] (1.2,1.4) -- (1.2,2);
\draw[line width=2pt] (1.2,2)  -- (2,2.7);
\draw[line width=2pt] (2,2.7) -- (2.9,2);
\draw[line width=2pt] (2.9,2)  -- (2,01.7);
\draw[line width=2pt] (2,01.7) -- (1.5,1.6);
\end{tikzpicture}
}

\newcommand{\secfourh}{
  \begin{tikzpicture}[>=stealth',shorten >=1pt,auto,node distance=2cm,semithick,scale=1.5]
   \draw [help lines] (0,0) grid (3,3);   
      \draw (1.2,1.4)  node  (m1)   {$\bullet$};         \draw (1.2,1.2) node{\protect{\scriptsize $x$}};
\draw (.4,2) node (m2) {$\bullet$};
   \draw (1,2.5) node (m3) {$\bullet$};       
      \draw (.8,2.75) node{\protect{\scriptsize $A$}}; 
   \draw (2,2.5) node (m4) {$\bullet$};      \draw (2.2,2.75) node{\protect{\scriptsize $B$}}; 
      \draw (2.6,2) node (m5) {$\bullet$}; 
      \draw (2, 1.3) node (m6) {$\bullet$};       
            \draw (1.5,1.8) node (m7) {$\bullet$};    \draw (1.5,1.6) node{\protect{\scriptsize $y$}};
\draw[line width=2pt] (1.2,1.4) -- (.4,2);
\draw[line width=2pt] (.4,2) --   (1,2.5);
\draw[line width=2pt]  (1,2.5) -- (2,2.5);
\draw[line width=2pt] (2,2.5) -- (2.6,2);
\draw[line width=2pt] (2.6,2) -- (2, 1.3);
\draw[line width=2pt] (2, 1.3) -- (1.5,1.8);
\end{tikzpicture}
}

\begin{document}

\maketitle
\begin{abstract}
We advance the program of connections between final coalgebras as sources of circularity in mathematics
and fractal sets of real numbers.  In particular, we are interested in the 
Sierpinski carpet, taking it as a fractal subset of the unit square.  We construct a 
category of \emph{square metric spaces} and an endofunctor on it which corresponds to the 
operation of gluing eight copies of a given square metric space along segments, as in the Sierpinski carpet.   
We show that the initial
algebra and final coalgebra exists for our functor, and that the final coalgebra is 
bilipschitz equivalent to the Sierpinski carpet.  Along the way, we make connections
to topics such as the iterative construction of initial algebras as $\omega$-colimits,
corecursive algebras, and the classic treatment of fractal sets
due to Hutchinson.
\end{abstract}

\section{Introduction}

This paper continues work on fractal sets modeled as final coalgebras.
It  builds on a line of work that
began with Freyd's result~\cite{freyd:real} that the unit interval $[0,1]$
is the final coalgebra of a certain endofunctor on the category of \emph{bi-pointed} sets.  
Leinster's paper~\cite{lein} is a far-reaching generalization of Freyd's result.
It 
 represents many of what would be intuitively 
called \emph{self-similar} spaces using (a) bimodules (also called profunctors or distributors);
(b) an examination of non-degeneracy conditions on functors of various sorts; (c) a construction of
final coalgebras for the types of functors of interest using a notion of resolution.    In addition to the 
characterization of fractal sets as sets, his seminal paper also characterizes them as topological spaces.

In a somewhat different direction, work related to Freyd's 
 Theorem continues with development of \emph{tri-pointed sets}~\cite{Bhat} and the proof  that the 
Sierpinski gasket $\mathbbm{SG}$ is related to the final coalgebra of a functor modeled on that of Freyd~\cite{freyd:real}.
(Please note that the \emph{gasket} is different from the \emph{carpet}.)
Although it might seem that this result is but a special case of the much better results in  Leinster~\cite{lein},
the work on tri-pointed sets was carried out in the setting of metric spaces rather than topological spaces
(and so it re-proved Freyd's result in that setting, too).   Work in the metric setting is unfortunately more
complicated.   It originates  in Hasuo, Jacobs, and Niqui~\cite{hjn}, a paper which
emphasized algebras in addition to coalgebras, and proposed
endofunctors defined using quotient metrics. 
Following this, Bhattacharya et al.~\cite{Bhat} show
 that for the unit interval,
 the initial algebra of Freyd's functor is also interesting, being the metric space of dyadic rationals,
and thus the unit interval itself is its Cauchy completion.  For the  Sierpinski gasket, the initial algebra
of the functor on tripointed sets is 
connected to the finite addresses used in building the gasket as a fractal; its completion again 
turns out to be the final coalgebra;
and while the gasket itself is \emph{not} the final coalgebra, the two metric spaces are bilipschitz equivalent.

In this paper, we take the next step in this area by considering the Sierpinski carpet $\SC$.
The difference between this and the gasket (or the unit interval) is that the gluing of spaces needed to define the functor
involves \emph{gluing along line segments}, not just at points.   
This turns out to complicate matters at every step.   The main results of the paper are analogs of what we saw for the 
gasket:  we have a category of metric spaces with additional structure
that we call \emph{square metric spaces}, an endofunctor $M\otimes -$ which 
takes a space to 8 scaled copies of itself glued along segments
(the notation recalls Leinster's paper, and again we are in the metric setting), a proof that the initial algebra and final coalgebra
exist, and that the latter is the completion of the former, and a verification that the actual Sierpinski carpet $\SC$
is bilipschitz equivalent to the final coalgebra.   Along the way, we need to consider a different functor
 $N\otimes -$ 
which is like $M\otimes-$ but involves 9 copies (no ``hole'').
The 
 final 
coalgebra of $N\otimes-$
turns out to be the unit square with the taxicab metric.   Moreover, in much of this work we have found it convenient to 
work with \emph{corecursive algebras} as a stepping stone to the final coalgebra; the unit square 
with the taxicab metric
turns out to be a corecursive
algebra for $N\otimes -$ on square metric spaces.
The Sierpinski carpet  $\SC$ turns out to be a 
 corecursive algebra
for the endofunctor $M\otimes -$, but it is not a final coalgebra for that endofunctor.

\subsection{Outline}

The paper begins with a discussion of the
Sierpinski carpet $\SC$ in classical terms, reviewing the
results from Hutchinson~\cite{Hut} that we need.
What we need most is that $\SC$ is the fixed point of certain
contractive map $\sigma$ on the space of non-empty compact subsets of the unit square.
The first leading idea in the paper is that the action of $\sigma$ can
be generalized to give an endofunctor
$F\colon\C\to\C$
on a category $\C$.  But it is not
immediate what that $\C$ and $F$ are.
 The category $\C$ is defined in
Section~\ref{section-SquaMS}; we call it the category
$\SquaMS$ of \emph{square metric spaces},
and the functor $F$ in
Section~\ref{section-functor-definition} is  written $X\mapsto M\otimes X$.
A square metric space is metric space $X$ together with a map
$S_X: M_0\to X$, where $M_0$ is the boundary
of the unit square.    In pictures, it would look like
the space on the left in Figure~\ref{figure-first-cloud}.
The mapping $S_X$ needs to be injective and satisfy some natural metric properties.

For technical reasons, $\Met$ in this paper is the category of metric spaces with distances
bounded by $2$ (not by $1$, since we need $M_0$ to be an object).
On the right in the figure, we indicate $M\otimes X$.  We go into detail on this functor 
$M\otimes -$ in Section~\ref{section-functor-definition}, and 
this will take a fair amount of preparation.

\begin{figure}[ht]  
\centering 
   \begin{subfigure}{0.4\linewidth}
    \centering
    \begin{tikzpicture}   
\node[cloud,
    fill = gray!10,
    minimum width = 2.22cm,
    minimum height = 1.68cm] (c) at (0,1.3) {};  
\node[rectangle,
    draw = lightgray,
    text = olive,
    minimum width = 1.3cm, 
    minimum height = 1.3cm] (r) at (0,1.3) {};   
    \end{tikzpicture}%
   \caption{$X$}
  \end{subfigure}
\begin{subfigure}{0.4\linewidth}
\centering
\begin{tikzpicture}
\node[cloud,
    fill = gray!10,
    minimum width = 2.2cm,
    minimum height = 1.68cm] (c) at (0,0) {};
\node[cloud,
    fill = gray!10,
    minimum width = 2.2cm,
    minimum height = 1.68cm] (c) at (1.3,0) {};
\node[cloud,
    fill = gray!10,
    minimum width = 2.2cm,
    minimum height = 1.68cm] (c) at (2.6,0) {};  
\node[cloud,
    fill = gray!10,
    minimum width = 2.2cm,
    minimum height = 1.68cm] (c) at (0,1.3) {};
\node[cloud,
    fill = gray!10,
    minimum width = 2.2cm,
    minimum height = 1.68cm] (c) at (2.6,1.3) {};  
\node[cloud,
    fill = gray!10,
    minimum width = 2.2cm,
    minimum height = 1.68cm] (c) at (0,2.6) {};
\node[cloud,
    fill = gray!10,
    minimum width = 2.2cm,
    minimum height = 1.68cm] (c) at (1.3,2.6) {};
\node[cloud,
    fill = gray!10,
    minimum width = 2.2cm,
    minimum height = 1.68cm] (c) at (2.6,2.6) {};      
 \node[rectangle,
    draw = lightgray,
    minimum width = 1.3cm, 
    minimum height = 1.3cm] (r) at (0,0) {};
\node[rectangle,
    draw = lightgray,
    minimum width = 1.3cm, 
    minimum height = 1.3cm] (r) at (1.3,0) {};  
 \node[rectangle,
    draw = lightgray,
    minimum width = 1.3cm, 
    minimum height = 1.3cm] (r) at (2.6,0) {}; 
 \node[rectangle,
    draw = lightgray,
    minimum width = 1.3cm, 
    minimum height = 1.3cm] (r) at (0,1.3) {};
 \node[rectangle,
    draw = lightgray,
    minimum width = 1.3cm, 
    minimum height = 1.3cm] (r) at (2.6,1.3) {};   
 \node[rectangle,
    draw = lightgray,
    minimum width = 1.3cm, 
    minimum height = 1.3cm] (r) at (0,2.6) {};
\node[rectangle,
    draw = lightgray,
    minimum width = 1.3cm, 
    minimum height = 1.3cm] (r) at (1.3,2.6) {};  
 \node[rectangle,
    draw = lightgray,
    minimum width = 1.3cm, 
    minimum height = 1.3cm] (r) at (2.6,2.6) {};     
\end{tikzpicture} 
\caption{$M\otimes X$}
\end{subfigure}
\caption{\label{figure-first-cloud}}
\end{figure}

The second leading idea is that $\SC$  should be related to the
\emph{final coalgebra} of $M\otimes-$.
Indeed, this explains the title of this paper.
We have the intuition that this should
be so from previous work on the unit interval~\cite{freyd:08}
and the Sierpinski gasket~\cite{Bhat}, and from the general
treatment of self-similar sets~\cite{lein}.  However, 
as we remarked above, this paper
involves a 
great deal more work than in those earlier works; we are not giving
a straightforward generalization of them.   
For example, Section~\ref{section-initial-algebra} constructs the initial
algebra of $M\otimes-$, and this already is more difficult than in 
previous work because the morphisms of the initial-algebra chain of $M\otimes-$ are not
isometric embeddings.  Still, $M\otimes-$ does have an initial algebra, and its
completion is the final coalgebra of this functor. 
This and other results are
proved in Section~\ref{section-final-coalgebras}.
We find it useful to bring in the concept of a \emph{corecursive algebra}, 
and so the results of that section should be of independent 
interest. 
The paper ends in Section~\ref{section-bilipschitz} with 
a proof that $\SC$ is bilipschitz equivalent to the final coalgebra of 
the functor $M\otimes -$.

The paper as a whole contains a mixture of geometric ideas that crop up 
in the study of square metric spaces and our functor $M\otimes-$,
and also very general facts about colimits of chains in various categories
and facts about corecursive algebras.  We hope that readers interested in
one or the other of these kinds of work will come away from our paper with
interest in the other kind, and that the mixture of ideas here will be useful
in the category-theoretic treatment of other fractal sets.

\paragraph{Acknowledgment} We would to thank to anonymous referee for their thorough reading of our paper and helpful comments.

\section{The Sierpinski carpet}

The main object of interest in this paper is the Sierpinski carpet.

\begin{figure}[h]
\[	\includegraphics[scale=0.12]{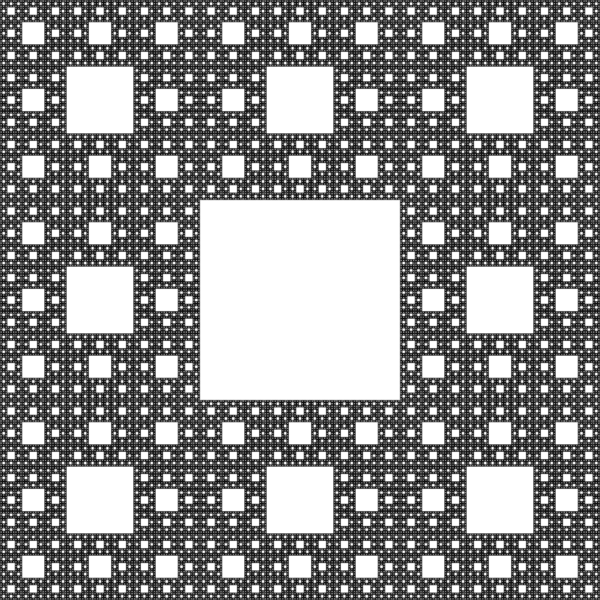}
\]
\end{figure}

We will begin by recalling the definition of the Sierpinski carpet $\SC$ (shown above) in terms of contractions of the unit square $U_0$, as in Hutchinson's work~\cite{Hut}.

\subsection{Review of  Hutchinson's theorem}

Let $(X,d)$ be a complete metric space, and let $\CC$
 be the set of non-empty compact subsets of $X$,
with the \emph{Hausdorff metric} $d_H$.   Here is how this is defined.
Given compact $A,B\subseteq X$, $d_H(A,B)$ 
    is the supremum of distances of points of one of the sets to the other
    one. This is defined by 
\begin{equation}
\label{eq:Hausdorff}
d_{H}(A,B) = \max(\displaystyle{\sup_{a\in A}}\  d(a,B),\displaystyle{\sup_{b\in B}}\  d(A,b)).
\end{equation}
In both cases, the distance from a point to a set is given by infima:
\[
d(a,B) = \inf_{b\in B} d(a,b).
\]
and similarly for $d(A,b)$.

 Let 
$M$
be a finite index set and suppose that for each $m\in M$, we have
a contracting map
$\sigma_m \colon X \to X$.  
We extend each $\sigma_m$ setwise to a function on (compact) sets 
 by taking images: for $A\subseteq X$, $\sigma_m(A) = \set{\sigma_m(x) : x\in A}$.  This
map  $\sigma_m$ is a contraction of $(\CC, d_{H})$.
Moreover, we define $\sigma \colon \CC\to\CC$ by 
\[ \sigma(A) = \displaystyle{\bigcup_{m\in M}} \sigma_m(A)\] 
Again $\sigma$ is a contracting map, and we let $K$ be its unique (non-empty) fixed point.
$K$ is called the \emph{invariant set} determined by the family $\set{\sigma_m: m\in M}$.

\begin{definition}
Fix $A\in\CC$ and contractions $\sigma_m$ for $m\in M$.
For each finite sequence $\vec{m} = m_1 m_2 \cdots m _k$ of elements of $M$, we define a set  $A_{\vec{m}}$ 
 by recursion on $k$, starting with $k = 0$ and the empty sequence $\eps$:

\[
\begin{array}{lcl}
A_{\eps} & = & A \\
A_{m_1 m_2 \cdots m _k m_{k+1}} & = &   \sigma_{m_{1}}( A_{m_2 m_3 \cdots m _{k+1} })\end{array}
\]

\end{definition}

\begin{proposition} [Hutchinson~\cite{Hut}] \label{prop-Hutchinson}
We have the following facts about the invariant set $K$:
\begin{enumerate}
\item If $A$ is a non-empty compact, then $diam(A_{m_1\ldots m_p})\rightarrow 0$ as $p\rightarrow\infty$, where 
 $diam(B) = 
\sup \set{d(x,y): x, y\in B}$.
\item For every infinite sequence $\overline{m} = m_1,m_2,\ldots, m_p, \ldots$ in $M$, 

\begin{equation}\label{decrease}
K_{\eps} \supseteq K_{m_1} \supseteq K_{m_1 m_2} \supseteq \cdots \supseteq K_{m_1 m_2\cdots m_p} \supseteq \cdots
\end{equation}

and $\displaystyle{\bigcap_{p=1}^\infty}K_{m_1\ldots m_p}$ is a singleton whose member is denoted $k_{\overline{m}}$.  $K$ is the union of these singletons.

\item \label{sequences} If $A$ is a non-empty compact set, then $d(A_{m_1\ldots m_p},k_{\overline{m}} )\rightarrow 0$ as $p\rightarrow\infty$.  In particular, $\sigma^p(A)\rightarrow K$ as $p\rightarrow\infty$ in the Hausdorff metric.

 \end{enumerate}
\end{proposition}

\subsection{The Sierpinski carpet}

Now we apply the general results in the last section to define the Sierpinski carpet $\mathbbm{S}$ as a subset of $U_0 = [0,1]^2$.  Throughout this paper, we will be working with $(U_0,d_{\Taxicab})$, where 

\begin{equation}
d_{\Taxicab}((x,y),(x_1,y_1)) = |x-x_1|+|y-y_1|
    \label{eq:taxicab}
\end{equation} is the taxicab metric.  

Most typically, we would view $\SC$ as a subset of $U_0$ with the Euclidean metric, $d_{\Euc}$.
However, we will see that we can use the taxicab metric in our characterization of $\SC$.

\begin{definition}
Two metric spaces $A$ and $B$ are \emph{bilipschitz equivalent} if there
is a bijection $f: A\to B$ and a number $K\geq 1$ such that 
\[  \oneoverK d_A(x,y) \leq d_B(f(x),f(y)) \leq K d_A(x,y) \]
for all $x,y\in A$.  
\end{definition}

\begin{proposition}\label{taxi-euc-bilip} $(U_0,d_{\Taxicab})$ is bilipschitz equivalent to $(U_0, d_{\Euc})$.  
\end{proposition}

\proof
Our bijection will be the identity map.  Let $K = 2$ and let $(x,y),(x_1,y_1)\in U_0$.  Then
\def\arraystretch{1.5}
\[\begin{array}{rclc}
     \onehalf d_{\Euc}((x,y),(x_1,y_1))& \leq & 
     \onehalf (d_{\Euc}((x,y),(x_1,y)) + d_{\Euc}((x_1,y),(x_1,y_1))) \\
     & = & \frac{1}{2}(|x-x_1|+|y-y_1|)  \\
     & \leq & d_{\Taxicab}((x,y),(x_1,y_1))  \\
     & = & |x-x_1| + |y-y_1|  \\
     & = & \sqrt{(x-x_1)^2} + \sqrt{(y-y_1)^2}  \\
     &\leq & \sqrt{(x-x_1)^2+(y-y_1)^2} + \sqrt{(x-x_1)^2+(y-y_1)^2}  \\
     & = & 2d_{\Euc}((x,y),(x_1,y_1))  \\
\end{array}
\]
\def\arraystretch{1}
\endproof

\begin{corollary} $C\subset U_0$ is a closed set with respect to $d_{\Taxicab}$ if and only if it is a closed set with respect to $d_{\Euc}$.  \end{corollary}

Let $\mathcal{C}$ denote the collection of non-empty closed subsets of $U_0$ (with respect to either metric). 
In order to apply Hutchinson's work to define $\SC$, we need to recall the general
definition of the  Hausdorff metric on compact sets from
(\ref{eq:Hausdorff}).
In our setting, let us introduce some notation:
\[d_{He}(A,B) = \max(\displaystyle{\sup_{a\in A}}\  d_{\Euc}(a,B),\displaystyle{\sup_{b\in B}}\  d_{\Euc}(A,b))\]
and 
\[d_{Ht}(A,B) = \max(\displaystyle{\sup_{a\in A}} \ d_{\Taxicab}(a,B),\displaystyle{\sup_{b\in B}} \ d_{\Taxicab}(A,b)).\]

\begin{proposition}\label{canconsidertaxi} 
$(\mathcal{C},d_{He})$ is bilipschitz equivalent to $(\mathcal{C}, d_{Ht})$.  
\end{proposition}

\proof
This follows from Proposition~\ref{taxi-euc-bilip}
\[
\begin{array}{rcl}
     \frac{1}{2} d_{He}(A,B) & = &   \onehalf \max(\displaystyle{\sup_{a\in A}}\  d_{\Euc}(a,B),\displaystyle{\sup_{b\in B}}\ d_{\Euc}(A,b)) \\
     & = & \max(\displaystyle{\sup_{a\in A}}\  \onehalf d_{\Euc}(a,B),\displaystyle{\sup_{b\in B}}\  \onehalf d_{\Euc}(A,b))\\
     & \leq & \max(\displaystyle{\sup_{a\in A}}\  d_{\Taxicab}(a,B),\displaystyle{\sup_{b\in B}}\  d_{\Taxicab}(A,b))\\
     & = & d_{Ht}(A,B)\\
\end{array}
\]
and similarly, $d_{Ht}(A,B) \leq 2d_{He}(A,B)$.  
\endproof

So from here on, we will consider $(U_0,d_{\Taxicab})$ and define $\mathbbm{S}$ as a subset of 
$U_0$ with respect to the taxicab metric.  

For the remainder of the section, we may write $d_{U_0}$ or simply $d$ to denote $d_{\Taxicab}$.

\begin{definition}
\label{def-sigmas}
\begin{enumerate}
\item $M$ is $\set{0,1,2}^2\setminus\set{(1,1)}$.

\item For each $m = (i,j)\in M$, let $\shrink(m)\in U_0$ be given by
 \[
\shrink(m) =( \onethird i, \onethird j).
\]
\item 
For a subset $A\subseteq U_0$, we define 
 $\sigma_m\colon \CC\to \CC$ by
\[
\sigma_m(A) =  \shrink(m) + \onethird(A) .
\]
Finally, 
let $\sigma\colon\CC \to \CC$ be $\sigma(A) = \displaystyle{\bigcup_{m\in M}} \sigma_m(A)$.
\end{enumerate}
\end{definition}

Since we are scaling by a factor of $\frac{1}{3}$, it is routine to verify that $\sigma$ is a contraction on $\CC$ with respect to $d_{Ht}$.  Indeed, it easy to verify that it is also a contracting map with respect to $d_{He}$.

\begin{definition}
The Sierpinski carpet $\SC$ is the unique fixed point of $\sigma\colon\CC \to \CC$.
That is, it is the unique non-empty compact (with respect to $d_{\Taxicab}$) subset of $U_0$  fixed by $\sigma$.

When we consider $\SC$ as a metric space, we primarily take the metric to be
the one inherited from $(U_0,d_{\Taxicab})$.  For example, the distance between
$(0,0)$ and $(1,1)$ is $2$.  
But because they are bilipschitz equivalent, if we had defined $\SC$ with respect to the Euclidean metric, we
would get the exact same fixed point.
\end{definition}

Indeed, $\SC$ is the unique non-empty compact (with respect to either metric) subset of $\mathbbm{R}^2$ fixed by $\sigma$.
But this is not relevant for us, and we prefer to work with subsets of 
the unit square $U_0$.

\section{The category of square metric spaces}
\label{section-SquaMS}

We start by defining $\SquaMS$, the category of 
\emph{square metric spaces}.  Though some of the arguments in the following sections will 
apply more generally, our work will primarily focus on this category.  
Our goal is to find an endofunctor $F$ on this category and an $F$-coalgebra which 
is bilipschitz equivalent to the Sierpinski carpet.  

\begin{definition}
Let \begin{equation}
    \label{eq-m0}
M_0=\{(r,s): r\in \{0,1\}, s\in [0,1]\}\cup\{(r,s):r\in [0,1], s\in \{0,1\}\}
\end{equation}
be the boundary of the unit square.  

A \emph{square set} is a set $X$ with with an injective map $S_X:M_0\rightarrow X$.  The idea is that $S_X$ designates the $4$ sides of the square.  Let $\SquaSet$ denote the category whose objects are square sets, and whose morphisms preserve $S_X$.  That is, for square sets $X$ and $Y$ and $f:X\rightarrow Y$, for $(r,s)\in M_0$, we must have $f(S_X((r,s))) = S_Y((r,s))$.
\end{definition}

\begin{example} Here are some
examples of square sets:  
\begin{itemize}
\item $M_0$ with $S_{M_0} = id$.
\item $X=[0,1]^2$, where $S_X$ is the inclusion map.
\item The Sierpinski carpet $\SC$, where $S_X$ is the inclusion map.
\end{itemize}
\end{example}

We are interested in square sets which are metric spaces.

\begin{definition}\label{definitionofsquams}
$(X,S_X)$ is a \emph{square metric space} if $X$ is a metric space bounded by $2$, and the boundary indicated by $S_X$ satisfies the following:
\begin{enumerate}
\item[($\sqone$)]\ For $i\in \{0,1\}$ and $r,s\in [0,1]$, $$d_X(S_X((i,r)),S_X((i,s))) = |s-r|$$ and $$d_X(S_X((r,i)),S_X((s,i))) = |s-r|.$$  That is, along each side of the square, distances coincide with distances on the unit interval.   
\item[($\sqtwo$)]\ For $(r,s),(t,u)\in M_0$, 
\[ d_X(S_X((r,s)),S_X((t,u)))\geq d_{\Taxicab}(S_{M_0}((r,s)),S_{M_0}((t,u)))=|r-t|+|s-u|.\]
This is a non-degeneracy requirement, which prevents our squares from ``collapsing''.  
For example, we want to avoid the case when opposite corners are less than distance $1$ from each other. 
 \end{enumerate}
\end{definition}

Note that we do not require the metric on the
boundary of the square to coincide with the Euclidean metric. 
Specifically, we are not requiring that opposite corners have distance $\sqrt{2}$.  In fact, we will be interested in a 
\emph{path metric} around the square. That is, we will determine the distance between points by the shortest path around the square (described in more detail below).

\begin{example}\label{pathmetric}
Here are examples of square metric spaces: 
\begin{itemize}
\item The unit square  $([0,1]^2, S)$ where $S$ is the inclusion map, with the taxicab metric.  

\item $(M_0, id)$ with the path metric: for $x,y\in M_0$, if they are on the same side, their distance coincides with the unit interval, if they are on adjacent sides which share a corner $C$, $d(x,y) = d(x,C)+d(C,y)$, and if they are on opposite sides, $d(x,y)$ is the minimum (between the two sides) of $d(x,C_1)+1+d(C_2,y)$ where $C_1,C_2$ are endpoints of a side not containing either $x$ or $y$, with $C_1$ on the side containing $x$ and $C_2$ on the side containing $y$.  Note that these distances are all bounded by $2$ (the distance between opposite corners is $2$).
Unless otherwise stated, when we use the notation $M_0$, it is for 
the boundary of the unit square with the path metric.

\item $(M_0, id)$ with the taxicab metric (the metric inherited from $([0,1]^2, S)$ above).
Note that the distance between points on opposite sides in this metric is almost always less than 
the distance in the path metric.
It will be important to distinguish the taxicab and path metrics on the set $M_0$.
\end{itemize}
\end{example}

\begin{definition}
 Let $X$ and $Y$ be metric spaces.
A map $f: X \to Y$ is \emph{short} if for all $x_1,x_2\in X$,
\[d_Y(f(x_1),f(x_2)) \leq d_X(x_1,x_2).\]
Other names for this notion are \emph{non-expanding} or \emph{non-distance-increasing} map.
When we consider metric spaces as a category $\MS$, we are using short maps as the morphisms.
\end{definition}

\begin{proposition}\label{Sxisashortmap} 
If $(X,S_X)$ is a square metric space, then $S_X:M_0\to X$ is a short map.
\end{proposition}

\begin{proof}
    Let $x,y\in M_0$.
  If $x$ and $y$ are on the same side of $M_0$, then $d_X(S_X(x),S_X(y)) = d_{M_0}(x,y)$,
  by ($\sqone$).
If $x$ and $y$ are on adjacent sides, let $C$ be the corner between them.  Then 
using  the triangle inequality in  $(X,S_X)$ and what we have just seen,
\[ \begin{array}{lcl} d_X(S_X(x),S_X(y))  
  & \leq & d_X(S_X(x),S_X(C))+d_X(S_X(C),S_X(y)) \\
 &  =  & d_{M_0}(x,C)+d_{M_0}(C,y) \\
 & =  & d_{M_0}(x,y).
 \end{array}\]
Finally, we have the case when $x$ and $y$ are on opposite sides
of the square.
Let $C_1,C_2$ be the endpoints of the side which provides the shortest path from $x$ to $y$ in $M_0$.  Then 
\[\begin{array}{cl}
&d_X(S_X(x),S_X(y))\\
  \leq & d_X(S_X(x),S_X(C_1))+ 
     d_X(S_X(C_1),S_X(C_2)) + d_X(S_X(C_2),S_X(y))\\
 = &  d_{M_0}(x,C_1) + 1 + d_{M_0}(C_2,y)  \\ 
= & d_{M_0}(x,y)
\end{array}
\]
\end{proof}

\begin{definition}
Let $\SquaMS$ be the category whose objects are square metric spaces (bounded by $2$) whose morphisms $f:(X,S_X)\to (Y,S_Y)$ are short maps which preserve $S$:
$S_Y = f\o S_X$.
\end{definition}

Proposition~\ref{sliceprop} provides a characterization of $\SquaMS$.\footnote{We
are grateful to an anonymous referee for this observation.}

\begin{proposition}\label{sliceprop}

$\SquaMS$ is the full subcategory of the slice category $M_0/\MS$ determined by the objects
$(X,S_X:M_0\to X)$ with the property that  $S_X$ is short and $(X,S_X)$ satisfies
($\sqone$) and ($\sqtwo$).  The initial object in $\SquaMS$  is  $(M_0,id)$
with the path metric.

\end{proposition}

\begin{proposition}\label{prop-mono}
Monomorphisms in $\SquaMS$ are
the morphisms which are one-to-one.
\end{proposition}

\proof
Let $f:X\rightarrow Y$ be a monomorphism.  Let $Z = M_0\cup\{z\}$ be the boundary of the unit square with the path metric and one extra point $z$ such that $d_Z(z,(r,s)) = 2$ for all $(r,s)\in M_0$.  This is an object in $\SquaMS$.  Let $x_0,x_1\in X$ and suppose $f(x_0) = f(x_1)$.  Define $g_i:Z\rightarrow X$ by $(r,s)\mapsto S_X((r,s))$ for $(r,s)\in M_0$ and $z\mapsto x_i$ for $i=0,1$.  These clearly preserve $S_Z$, and are short maps since
\[ d_X(g_i((r,s)),g_i((t,u)))\leq d_Z(S_Z((r,s)),S_Z((t,u)))\] for $(r,s),(t,u)\in M_0$ by the same argument as the previous proposition, and 
\[ d_X(g_i(z),g_i((r,s))) \leq 2 = d_Z(z,(r,s))\] for $(r,s)\in M_0$. 
Now 
$f\circ g_0 = f\circ g_1$, since $f(g_i((r,s))) = S_Y((r,s))$,
and $f(g_0(z)) = f(x_0)=f(x_1) = f(g_1(z))$.  So since $f$ is a monomorphism, $g_0 = g_1$, which means that $x_0 = g_0(z) = g_1(z)= x_1$.  

For the other direction, suppose $f$ is an injective morphism and $g_0,g_1:Z\rightarrow X$ are morphisms from an arbitrary object $Z$  such that $f\circ g_0 = f\circ g_1$.  Then for $z\in Z$, $f(g_0(z)) = f(g_1(z))$. 
Since $f$ is injective, $g_0(z) =g_1(z)$.  Hence, $g_0=g_1$.    
\endproof

\begin{proposition}$\SquaMS$ has no final object.
\end{proposition}

\proof 
As in the previous proposition, let $Z$ be the boundary of the unit square, $M_0$, with the path metric and a single point $z$ defined to be distance $2$ from every point in $M_0$.  This is an object in $\SquaMS$, via the inclusion $M_0 \rightarrow Z$.

Let $Y$ be an object in the category.
Then consider $f_0: Z\rightarrow Y$ defined by $(r,s)\mapsto S_Y((r,s))$ for $(r,s)\in M_0$ and $z\mapsto S_Y((0,0))$, and $f_1:Z\rightarrow Y$ defined by $(r,s)\mapsto S_Y((r,s))$ for $(r,s)\in M_0$ and $z\mapsto S_Y((1,1))$. 
As in the previous proposition, these maps are both morphisms.  
So since there are two distinct morphisms from $Z$ to $Y$, $Y$ cannot be a final object in $\SquaMS$. 
\endproof

\section{The functors $M\otimes -$ and $N\otimes -$}
\label{section-functor-definition}

In this section we will define a functor 
\[ M\otimes- : \SquaMS\to \SquaMS
\] which, when applied to the initial square metric space $M_0$ (using the path metric), 
will give us objects which correspond to iterations of the Sierpinski carpet.  The idea is that $M$ will be a set of indices indicating positions to place scaled copies of $X$,
and $M\otimes X$ also indicates identifications
that turn the metric space $M\times X$ into a
square metric space.  In detail,
$M\otimes X$ will contain $8$ copies of $X$
arranged in a $3\times 3$ grid, but without the central copy.  
We have mentioned this functor in the Introduction.
In Figure~\ref{figure-first-cloud} we showed caricatures of square spaces and
the action of $M$.
For a square space $X$, 
we want $M\otimes X$ to look like
eight copies scaled by a factor of $\onethird$ with appropriate gluings on edges of the squares,
and with a ``hole'' in the middle.

Later in the paper, we will iterate this functor in order to form a chain, beginning with $M_0$, the boundary of the unit square.
Then we take the colimit of this chain, and finally take the completion of the colimit.
As we shall see, we obtain a space bilipschitz equivalent to the Sierpinski carpet;
this is the main result in the paper.
We will also define a different functor
$N\otimes-$.
The difference between $M\otimes X$ and $N\otimes X$
is that $N\otimes X$ uses $9$ copies instead of $8$;
it has no central ``hole.''
This functor shares properties with $M\otimes-$. 
To obtain the desired results about $M\otimes-$ it is useful to also use results on $N\otimes -$.  

\subsection{A general discussion of quotient metrics on sets}
\label{section-L}

In this section, we work at a high 
level of generality so that we can obtain results which we then
apply to the main functors on $\SquaMS$ of interest in this paper.\footnote{
We will do this work for the category $\SquaMS$, though it should be noted that the 
results of this section can be adapted to apply to a broad collection of categories, such as the  bipointed or tripointed metric spaces in \cite{Bhat,freyd:real}.}
As mentioned above, those functors are called $M\otimes-$ and $N\otimes -$,
but they are not defined until Sections \ref{section-M-otimes} and \ref{section-N-otimes} respectively.

Let $L$ be a finite set. 
We call the elements of $L$ \emph{indices}.  The idea is that these will indicate positions in which we will place scaled copies of a given square metric space $X$. 
We shall endow the product set $L\times X$
with a metric space structure in (\ref{nottensor}).
We subsequently define $L\otimes X$ using the quotient metric (Definition \ref{quotientmetricdef}, via a certain equivalence relation $E$). 
Our work is rather general.   
 We will give requirements on $L$ and $E$ which will guarantee that $L\otimes X$ is a metric space.

 We are not, however, going to show that $L\otimes -$ is a functor on $\SquaMS$.  Indeed, our requirements on $L$ and $\sim$ will not guarantee that $L\otimes X$ is in $\SquaMS$, and they are not enough allow us to define $L\otimes f$ for morphisms $f$ in $\SquaMS$.  
 The intention here is to work at a level of generality such that we can use the metric space result towards showing that $M\otimes-$ and $N\otimes-$ are functors.

Let $E$ be an equivalence relation on $L\times M_0$. Later in the paper, given an 
object $X$, the pairs in $E$ will identify places where we ``glue copies of $X$'' by a procedure which we will specify
shortly.    Of course, the set $E$ is defined independently of $X$; it is simply an equivalence relation on $L\times M_0$.

For a fixed object $X$ in $\SquaMS$, define a relation $\approx$ on $L\times X$ as follows: For $m,n\in L$ and $(r,s), (t,u)\in M_0$, 
\begin{equation}\label{extendE}
\mbox{$(m,S_X((r,s))) \approx (n,S_X((t,u)))$ if and only if $((m,(r,s)),(n,(t,u)))\in E$.} 
\end{equation}

Let $\sim$ be the symmetric, reflexive, and transitive closure of 
$\approx$ on $L\times X$.
In more detail, if $E$ is symmetric, then so is $\approx$.  If $E$ is transitive, then again so is $\approx$.
But even if $E$ is reflexive, $\approx$ need not be reflexive, since $S_X$ is almost never surjective.
So this is why we must in general extend $\approx$ to get the relation $\sim$.

In the following definition, we wish to characterize equivalence relations which suit our needs later on, but are sufficiently general to apply to a broader class of similar constructions.

As we said, the big idea is that we will ``glue copies of $X$'' together, specifically along sides of the image of $M_0$ under $S_X$.  We need to do this in such a way that we set ourselves up to view the resulting object as a metric space.

\begin{definition}\label{quotientsuitabledefinition} 

Let $D = \{B,\ell,R,T\}$  be a set denoting the bottom, left, right, and top sides of $M_0$.   That is, 
\[ 
\begin{array}{lcl}
B & = &  \{(r,0):r\in[0,1]\},
\\ \ell  & = & \{(0,s):s\in [0,1]\}, \\
\end{array}
\qquad
\begin{array}{lcl}
R & = &\{(1,s):s\in [0,1]\},
\\ T  & = & \{(r,1):r\in[0,1]\}.
\end{array}
\]

An equivalence relation $E$ on $L\times M_0$ is \emph{quotient suitable} if the following data exist, and if $E$ is 
characterized in terms of them as mentioned below: 

First, an injective partial function $\kappa:L\times D\rightarrow L\times D$ such that 
\begin{itemize}
    \item For all $m\in L$ and $Y\in D$, there is no $Z\in D$ such that $\kappa(m,Y) = (m,Z)$.
    \item The domain and image of $\kappa$ are disjoint. 
    \item If $(m,Y)$ is in the domain of $\kappa$ and $\kappa(m,Y) = (n,Z)$, then for all $(m,Y')$ in the domain of $\kappa$ with $Y'\neq Y$, $\kappa(m,Y')\neq (n,Z')$ for any $Z'\in D$.  
\end{itemize}

Second, for each $(m,Y)$ in the domain of $\kappa$, an isometry $f_{m,Y}:Y\rightarrow Z$ (where $(n,Z) =\kappa(m,Y)$).

Observe that we may view each side of $M_0$ as an isometric copy of $[0,1]$, so the only possible isometries $f_{m,Y}$ are either the identity or the map $r\mapsto 1-r$.  

So if $\kappa(m,Y)=(n,Z)$, we also have an isometry $f_{n,Z}:Z\rightarrow Y$ where $f_{n,Z} = f_{m,Y}^{-1}$.

And our requirement about all of this is that $E$ is the symmetric, transitive, and reflexive closure of 
$$\displaystyle{\bigcup_{(m,Y)\in dom(\kappa)}} \{((m,y),(n,f_{m,Y}(y))): y\in Y, \mbox{for some $Z$, $\kappa(m,Y) = (n,Z)$}\}.$$

\end{definition}

The big idea is that $E$ comes from matching sides of $M_0$ to sides in different copies of it.  The first requirement on $\kappa$ tells us that in a single copy of $M_0$, none of the sides are equivalent to each other.  The second requirement along with the fact that $\kappa$ is an injective function tells us that if we fix one copy and one side, it is matched with at most one other side in one other copy.  The third requirement tells us that between two copies, we cannot have multiple sides which are equivalent. Geometrically, we may view the maps $f_{m,Y}$ as preserving a side, or reflecting it.

When $E$ is a quotient suitable relation, we have a few nice properties 
of the induced equivalence relation on $L\times X$ for an arbitrary $X\in \SquaMS$.  When we refer to sides in $S_X[M_0]$, we mean the image of the corresponding sides in $M_0$ under $S_X$.  Since $S_X$ is injective, the sides are disjoint except at their shared corners.

\begin{lemma}\label{technicalquotientsuitablecorollary}
Let $E$ be a quotient suitable relation on $L\times M_0$ and let $X\in \SquaMS$.  Let $\sim$ be the equivalence relation on $L\times X$ described below (\ref{extendE}).

\begin{enumerate}

\item If $(m,x)\neq (n,y)$ in $L\times X$, then $(m,x)\sim (n,y)$ 
implies that $x,y\in S_X[M_0]$.  

\item $\sim$ relates corners to corners.  That is, if $(r,s)\in M_0$ is such that $r,s\in \{0,1\}$, and $m,n\in L$ and $(t,u)\in M_0$ are such that $(m,S_X((r,s)))\sim (n,S_X((t,u)))$, then $(t,u)$ is a corner (that is, $t,u\in \{0,1\}$). 

\item Suppose $x$ is in $S_X[M_0]$ but is not a corner and $y$ is on the same side of $S_X[M_0]$.  If there are  $m,n\in L$ and $x'\in X$ are such that $m\neq n$ and $(m,x)\sim (n,x')$, then there is some $y'\in M_0$ on the same side as $x'$ such that $(m,y)\sim (n,y')$.  

Furthermore, $d_{X}(x,y) = d_X(x',y')$.

\item Suppose $x$ is not a corner in $S_X[M_0]$ and 
that there are $m,n\in L$ and $x'\in X$ such that $m\neq n$ and $(m,x)\sim(n,x')$.  Suppose further that $y$ is on the same side as $x$ in $S_X[M_0]$ and is also not a corner, and that for some $l\in L$ with $l\neq m$ and $y'\in X$, $(m,y)\sim (l,y')$.  Then we must have $l=n$ and $y'$ is on the same side as $x'$ in $S_X[M_0]$.  
\end{enumerate}

3. and 4. can be thought of as existence and uniqueness in some sense.  The idea is that if we have one point on a side related to another side in another copy, its entire side is related to that other side in that other copy as well, and furthermore, we cannot relate any other sides between these two copies of $X$.  

\end{lemma}

\begin{proof}
\begin{enumerate}
 \item Immediate from the definition of $\sim$. 
 \item  Follows from the fact that the only isometries between sides will map corners to corners, and taking the symmetric, reflexive, and transitive closure will still only relate corners to corners. 
 \item Start with $x\in S_X[M_0]$ which is not a corner, and let $y\in S_X[M_0]$ be on the same side as $x$.  Let $(r,s),(t,u)\in M_0$ be such that $S_X((r,s)) = x$ and $S_X((t,u))=y$, and let $Y\in D$ be the side containing $(r,s)$ and $(t,u)$. 

 Suppose there are $m,n\in L$ and $x'\in X$ such that $m\neq n$ and $(m,x)\sim (n,x')$.  Then by part 1., there is $(r',s')\in M_0$ such that $x'=S_X((r',s'))$, and by part 2., $(r',s')$ is not a corner.  So there is a single side $Z\in D$ containing $(r',s')$.  Then from the definition of quotient suitable, we must have $(r',s') = f_{m,Y}((r,s))$, where $f_{m,Y}: Y\rightarrow Z$ is the appropriate isometry.  

 Then by our definition of $E$, we know that $((m,(t,u)),(n,f_{m,Y}((t,u)))) \in E$, so let $(t',u')=f_{m,Y}((t,u))$, which is on the same side as $(r',s')$.  Thus, $y'=S_X((t',u'))$ is such that $(m,y)\sim (n,y')$, and is on the same side as $x'$.  

 Furthermore, since $S_X$ is an isometry between points on the same side, 
\[\begin{array}{rcl}
d_X(x,y) &=& d_X(S_X((r,s)),S_X((t,u))\\
&=& d_{M_0}((r,s),(t,u))\\
&=& d_{M_0}(f_{m,Y}((r,s)),f_{m,Y}((t,u)))\\
&=& d_{M_0}((r',s'),(t',u'))\\
&=& d_X(S_X((r',s')),S_X((t',u')))\\
&=& d_X(x',y')\\

\end{array} \]
 \item Suppose $x$ is not a corner in $S_X[M_0]$ and there are $m,n\in L$ and $x'\in X$ such that $m\neq n$ and $(m,x)\sim (n,x')$.  Then $x = S_X((r,s))$ for some $(r,s)\in M_0$, and by part 1., $x' = S_X((r',s'))$ for some $(r',s')\in M_0$.  

 Suppose further that $y\in S_X[M_0]$ is also not a corner and is on the same side as $x$, so $y = S_X((t,u))$ for some $(t,u)\in M_0$ on the same side as $(r,s)$.  Assume for some $l\in L$ with $l\neq m$ and $y'\in X$, $(m,y)\sim (l,y')$.  Then by part 1., $y'=S_X((t',u'))$ for some $(t',u')\in M_0$. 

 Let $Y\in D$ be the unique side containing $(r,s)$ and $(t,u)$, and let $Z\in D$ be the unique side containing $(r',s')$. (Since neither $(r,s)$ nor $(t,u)$ is a corner,
each is only on one side).  
 Then since $((m,(r,s)),(n,(r',s')))\in E$, by the definition of quotient suitable,  $(r',s') = f_{m,Y}((r,s))$.

Even after taking the symmetric, reflexive, and transitive closures, the only elements of the equivalence class of $(m,(t,u))$ under $E$ are itself and $(n, f_{m,Y}((t,u)))$.  Thus, since $((m,(t,u)),(l,(t',u')))\in E$ and $l\neq m$, we must have $(n,f_{m,Y}((t,u))) = (l,(t',u'))$, so $l=n$ and $(t',u') = f_{m,Y}((t,u))$ which is on the side $Z$.
Thus, $y'=S_X((t',u'))$ is on the same side as $x'$.  
\end{enumerate}
 
\end{proof}

\paragraph{The Quotient Space and Quotient Metric}  Recall that every object in $\SquaMS$ has distances bounded by $2$.
Ultimately we will define $L\otimes X$, in which we will consider a quotient of $L\times X$, and show that this is a metric space.  
As a stepping stone, we consider a metric $d$ on $L\times X$ defined by 
\begin{equation}\label{nottensor} 
d_{L\times X}((m,x),(n,y)) = 
\left\{ \begin{array}{ll}
\frac{1}{3}d_X(x,y) & \mbox{if $m=n$} \\
      2 & \mbox{otherwise}\\
\end{array}
\right\}
\end{equation}

So the distance is scaled by $\frac{1}{3}$ inside of 
each copy of $X$, and otherwise, it is $2$ (the maximum distance).  
The constant $\frac{1}{3}$ comes from the 
particular sets $M$ and $N$ to which we apply the construction in the next sections.

We see right away that $d_{L\times X}$ is bounded by $2$, since points in the same copy of $X$ will be at most $\frac{2}{3} < 2$ from each other, and points in different copies will be $2$ away from each other. 

Fix a quotient suitable equivalence relation $E$ on $L\times M_0$. 

\begin{definition}
Let $X$ be a square metric space.
The space $L\otimes X$ is the quotient of $L\times X$ by the equivalence relation $\sim$ (described below (\ref{extendE})):
\[
\begin{array}{lcl}
L\otimes X &  =  & (L\times X)/\!\!\sim
\\
m\otimes x & 
 \mbox{denotes} &  \mbox{the equivalence class of $(m,x)$ in $L\otimes X$.}\\
 \end{array}
 \]

 \end{definition}
(Note that our notations $L\otimes X$ 
and $m\otimes x$
do not include $\sim$, but this is to unburden the notation.
All our work uses $\sim$.)

In order to define a metric on $L\otimes X$, we will need the following notions:

\begin{definition} 
\label{quotientmetricdef}
\begin{enumerate}
\item  For $(m,x),(n,y)\in L\times X$, a \emph{path} from $(m,x)$ to $(n,y)$ is a finite list of elements of $L\times X$, $(m_0,x_0),\ldots, (m_k,x_k)$, such that $(m_0,x_0) =(m,x)$ and $(m_k,x_k)=(n,y)$.
\item The \emph{score} of the path $(m_0,x_0),\ldots,(m_k,x_k)$ is $$\displaystyle{\sum_{i=0}^{k-1}} \widehat{d}((m_i,x_i),(m_{i+1},x_{i+1}))$$ where 

\[\widehat{d}((m_i,x_i),(m_{i+1},x_{i+1}))= 
\left\{ \begin{array}{ll}
0 & \mbox{if $(m_i,x_i)\sim(m_{i+1},x_{i+1})$} \\
      d_{L\times X}((m_i,x_i),(m_{i+1},x_{i+1})) & \mbox{otherwise}\\
\end{array}
\right\}\]
\item For $m\otimes x$ and $n\otimes y$ let $d_{L\otimes X}(m\otimes x,n\otimes y)$ denote the infimum over all paths from $(m,x)$ to $(n,y)$ of the score. We will refer to this as the \emph{quotient metric}. 
\end{enumerate}
\end{definition}

Note that, as it is defined, $d_{L\otimes X}$ is a pseudo-metric:  clearly
$d_{L\otimes X}$ is symmetric, the distance between any point and itself is $0$, 
and it will satisfy the triangle inequality since the concatenation of two paths is a path.  We will show that 
$d_{L\otimes X}$ is in fact a metric: distinct points will have positive distance. To achieve this, we will show that the distance is actually witnessed by the score of 
some particular finite path; it is not just an infimum of the scores of 
an infinite set of paths.

\begin{definition}\label{simAlternatingPath}
For $(m,x),(n,y)\in L\times X$, an \emph{alternating path in $L\times X$} from $(m,x)$ to $(n,y)$ is  a path from $(m,x)$ to $(n,y)$ such that every other element is related by $\sim$, and those elements not related by $\sim$ are distinct and share the same first entry $m_i$.  The relation by $\sim$ can start with the first or second entry, and either the last pair is related by $\sim$
or the pair just before the last is related by $\sim$.
In other words, it is a sequence
 of the form
\begin{equation}\label{L9}
 (m,x)=(m_0,x_0),(m_0,x_0')\sim (m_1,x_1),(m_1,x_1')\sim \ldots \sim (m_p,x_p), (m_p,x_p') = (n, y)
 \end{equation}
where $(m_0,x_0)$ or $(m_p,x_p')$ (or both) might be omitted.
(If the first is omitted, then $x$  belongs to  $S_X[M_0]$, and similarly with the last and $y$.) 
\end{definition}

\begin{remark}\label{length0path} Note, if $(m,x)$ and $(n,y)$ have an alternating path between them in which $p=0$, either $(m,x)=(n,y)$ or $m=n$.  When we say that there is an alternating path from $m\otimes x$ to $n\otimes y$, we mean there is an alternating path from $(m,x)$ to $(n,y)$.  Note that the choice of representatives of the equivalences classes of $m\otimes x$ and $n\otimes y$ are not important, since if we have an alternating path between two representatives, by adding one more entry on each end with $\sim$ or replacing the first or last entry as appropriate, we have an alternating path between any two representatives of $m\otimes x$ and $n\otimes y$ respectively. \end{remark}

\begin{lemma}\label{distancelemma}  
For $(m,x)$ and $(n,y)$ in $L\times X$,  either every path from $(m,x)$ to $(n,y)$ has score $2$, or for any path from 
$(m,x)$ to $(n,y)$ there is an alternating path 
from $(m,x)$ to $(n,y)$ 
with smaller or equal score
\end{lemma}

\proof
If it exists, take a path from $(m,x)$ to $(n,y)$, say
\begin{equation}\label{distancelemmapath}
(m,x) = (m_0,x_0), (m_1, x_1), (m_2, x_2), \ldots, (m_p, x_p) = (n,y)
\end{equation}
with score strictly less than $2$. 

If any adjacent pair on the path, say   $(m_k, x_k)$ and $(m_{k+1},x_{k+1})$ has 
$m_k \neq m_{k+1}$ and the pair is not related by $\sim$, then this pair
contributes $2$ to the score, which is impossible since we assumed the score is $<2$.  We thus assume that this case does not arise in what follows.

We may take our path (\ref{distancelemmapath}) and shorten any chain of $\sim$ relations.
This is because $\sim$ is transitive.  Thus, we can assume that no three adjacent pairs are related by $\sim$.
In other words, we never have  $(m_k, x_k) \sim (m_{k+1},x_{k+1}) \sim (m_{k+2},x_{k+2})$.

At this point, we argue by induction on $p$ that for  every path (\ref{distancelemmapath})
which meets all of the
assumptions so far, there is an alternating path with smaller or equal score.
If $p=0$, then $(m,x)=(n,y)$ so $(m,x)=(m_0,x_0),(m_0,x_0')=(n,y)$ is an alternating path with score $0$ (so, equal to the score of the original).

Assume our result for paths of length $<p$, and consider a path as in (\ref{distancelemmapath})
of length $p$.  If this path is not alternating, then there must be  
 $(m_k, x_k) \not\sim (m_{k+1},x_{k+1}) \not\sim (m_{k+2},x_{k+2})$, since we have assumed we cannot have $3$ or more entries in a row related by $\sim$.
 In this case, what we said at the end of the first paragraph implies that $m_k = m_{k+1} = m_{k+2}$.
 So we may shorten our path by deleting $(m_{k+1},x_{k+1})$.
 By the triangle inequality (in $X$) and the definition of the metric on $L\times X$
 in (\ref{nottensor}), the score does not increase with this deletion.
 And then applying our induction hypothesis
 to the shortened path proves our result.
\endproof

We need a few more technical lemmas about shortening alternating paths. 

\begin{lemma}\label{technicalpathlemma}
Consider an alternating path  
\begin{equation}\label{firstbad}
(m,x) =(m_0,x_0), (m_0,x_0')\sim(m_1,x_1),(m_1,x_1')\sim\ldots \sim (m_p,x_p),(m_p,x_p')=(n,y)
\end{equation}
with  $p\geq 0$.  

Just in the context of this lemma, say a \emph{bad configuration}
in an alternating path (\ref{firstbad})
is 
a number $k$  such that one of the following holds: 
\begin{itemize}
    \item $k=0$, $x_0$ and $x_0'$ are on the same side of $S_X[M_0]$, and $x_0'$ is not a corner, 
    \item $0< k< p$, $x_k$ and $x_k'$ are on the same side of $S_X[M_0]$, and at least one of $x_k$ or $x_k'$ is not a corner, 
    \item $k=p$, $x_p$ and $x_p'$ are on the same side of $S_X[M_0]$, and $x_p$ is not a corner. 
    
\end{itemize}

Note that if $(m_0,x_0)$ (or $(m_p,x_p)$) is omitted because $x$ (or respectively $y$) is in $S_X[M_0]$, then the only possible bad configuration is the second of the three cases
above.

Then: from (\ref{firstbad})
we can find an alternating path with strictly fewer entries,  a smaller or equal score than the original path, and with no bad configurations.
\end{lemma}

\begin{proof}
    First suppose our alternating path
    (\ref{firstbad})
    has exactly one bad configuration. 
    Without loss of generality, we will suppose the bad configuration is 
     $k$ such that $0<k<p$ and $x_k'$ is not a corner.  The cases when $k=0$ or $p$, as well as the case when $0<k<p$ and $x_k$ is not a corner (but $x_k'$ might be) are all similar.


This assumption tells us that
 $x_k$ and $x_k'$ are on the same side of $S_X[M_0]$.  
Our hypothesis in this lemma implies that $(m_k,x_k')\sim (m_{k+1},x_{k+1})$. 

By Lemma~\ref{technicalquotientsuitablecorollary}(3), there exists $\widehat{x}_k$ on the same side as $x_{k+1}$ such that $(m_k,x_k)\sim (m_{k+1},\widehat{x}_k)$. (Note: when we say ``on the same side'', there is no ambiguity.  Since $x_k'$ is not a corner, $x_{k+1}$ is also not a corner by Lemma~\ref{technicalquotientsuitablecorollary}(1), which means that there is in fact only one side of $S_X[M_0]$ containing it.) 
   By the definition of the metric on $L\times X$, the triangle inequality in $X$, and Lemma~\ref{technicalquotientsuitablecorollary}(3),

 \setstretch{1.2}  
 \[
 \begin{array}{cl}
& d_{L\times X}((m_{k+1},\widehat{x}_k),(m_{k+1},x_{k+1}')) \\
 =& \frac{1}{3} d_X(\widehat{x_k},x_{k+1}')\\
 \leq & \frac{1}{3} d_X(\widehat{x_k},x_{k+1}) + \frac{1}{3} d_X(x_{k+1},x_{k+1}')\\
=& \frac{1}{3}d_{X}(x_k,x_k') + \frac{1}{3}d_{X}(x_{k+1},x_{k+1}')\\
=& d_{L\times X}((m_k,x_k),(m_k,x_k')) +d_{L\times X}((m_{k+1},x_{k+1}),(m_{k+1},x_{k+1}')). \\
\end{array}\]
\setstretch{1}

\noindent
Thus, we can replace this section of the path: 
$$(m_{k-1},x_{k-1}), (m_{k-1},x_{k-1}')\sim(m_k,x_k),(m_k,x_k')\sim (m_{k+1},x_{k+1}),(m_{k+1},x_{k+1}')$$ with the 
path just below, which has strictly fewer entries:
$$(m_{k-1},x_{k-1}), (m_{k-1},x_{k-1}')\sim (m_{k+1},\widehat{x}_k), (m_{k+1},x_{k+1}').$$
We are using that $(m_{k-1},x_{k-1}') \sim (m_k,x_{k})\sim (m_{k+1},\widehat{x}_k)$, and 
that $\sim$ is transitive.  
Hence, we get a path with fewer entries and a smaller or equal score. Furthermore, since we assumed that there was only one bad configuration, we only need to make sure that $\widehat{x_k}$ and $x_{k+1}'$ are not on the same side, but we know this holds since $\widehat{x_k}$ is on the same side as $x_{k+1}$ and is not a corner, so it cannot be on the same side as $x_{k+1}'$.

Then we proceed by induction on the number of bad configurations in the alternating path.  Use the process described to ``remove'' the first (leftmost in the indexing) bad configuration, then apply the induction hypothesis.  
\end{proof}

On a related note, if we have two entries in our path which are 
\emph{strictly on the same side of $S_X[M_0]$}
(that is, neither are corners) in the same copy of $X$, then we can replace one of those entries with a corner such that we do not have two entries which are strictly on the same side of $M_0$.

\begin{lemma}\label{enteringonsamesidelemma} For $(m,x)$ and $(n,y)$ in $L\times X$,
either every path from $(m,x)$ to $(n,y)$ has score $2$, or for any path from $(m,x)$ to $(n,y)$ in $L\times X$, there exists a path
(\ref{newpath})
with shorter or equal score, and at most as many entries as the original, 
\begin{equation}\label{newpath}
(m,x)=(m_0,x_0),(m_0,x_0')\sim (m_1,x_1)\ldots (m_{p-1},x_{p-1}')\sim (m_p,x_p), (m_p,x_p') = (n,y)
\end{equation}
such that 
\begin{itemize}
    \item The new path (\ref{newpath}) is an alternating path,
    \item For $0\leq i\leq p$, if $x_i$ and $x_i'$ are on the same side of $S_X[M_0]$, then they are both corners, 
    \item For $j$ and $k$ with $0\leq j<k\leq p$, if $m_j=m_k$,  and
     $x_j'$ and $x_k$ are  on the same side of $S_X[M_0]$,
    then at least one of these points $x_j'$ or  $x_k$ is a corner.
\end{itemize}    
    \end{lemma}

\begin{proof}
    By Lemma~\ref{distancelemma} and Lemma~\ref{technicalpathlemma}
    we may start with an alternating path 
\begin{equation}\label{temppath}(m,x)=(m_0,x_0),(m_0,x_0')\sim (m_1,x_1)\ldots (m_{p-1},x_{p-1}')\sim (m_p,x_p), (m_p,x_p') = (n,y)\end{equation} 
    satisfying our requirements such that if  $x_i$ and $x_i'$ are on the same side of $S_X[M_0]$, then at least one of them is a corner for $0\leq i\leq p$. 

    From here on, we will assume we have such a path.

 Say a pair of indices $j$ and $k$ with $0\leq j<k\leq p$ is a \emph{bad configuration}  (in this proof) if 
 \begin{itemize}
     \item $m_j = m_k$,
     \item $x_j'$ and $x_k$ are on the same side of $S_X[M_0]$,
     \item Neither of $x_j'$ and $x_k$ are corners.
     
 \end{itemize}
Note that if our alternating path in (\ref{temppath}) has a bad configuration, it will fail to satisfy the third condition in our lemma.

    We will prove by induction on the number of bad configurations that we may adapt our path (\ref{temppath}) such that it will satisfy all three requirements, and so that it has at most as many entries as (\ref{temppath}) and a score at most that of the original. 


Suppose that we have exactly one bad configuration, so there are $j$ and $k$ with $0\leq j<k\leq p$ such that $m_j=m_k$, and $x_j'$ and $x_k$ are both not corners and are on the same side of $S_X[M_0]$.

Suppose that $j$ and $k$ are only one index apart, that is, $k = j+1$.  Then our path looks like $(m_j,x_j),(m_j,x_j')\sim (m_k,x_k),(m_k,x_k')$.  But then since $m_j=m_k$, we may eliminate $(m_j,x_j')\sim (m_k,x_k)$, and still have an alternating path with less or equal score, and which no longer has a bad configuration. 
   
Otherwise, by Lemma~\ref{technicalquotientsuitablecorollary}(4), we must have $m_{j+1} = m_{k-1}$, and since $x_j'$ and $x_k$ are not corners, $x_{j+1}$ and $x_{k-1}'$ are on the same side.  

By assumption, there are no other bad configurations, which means that if 
 $x_{j+1}'$ and $x_{k-1}$ are both not corners, then they cannot be on the same side of $S_X[M_0]$.   In other words, they cannot both be non-corners and on the opposite side of $x_{j+1}$ and $x_{k-1}'$.  So at least one of them must be on an adjacent side.  

To better understand the situation, assume without loss of generality that $x_{j+1}$ and $x_{k-1}'$ are on the top of $S_X[M_0]$, and that $x_{j+1}'$ is on the left side. The other cases are similar.  
We have the following picture.   
\[
   \begin{tikzpicture}[>=stealth',shorten >=1pt,auto,node distance=2cm,semithick,scale=2.7]   
\draw [help lines] (0,-1) grid (1,1);
    \draw (.7,0) node (pointr) {$\bullet$};
       \draw (.7,.6) node (pointrr) {$\bullet$}; 
    \draw (.3,0) node (pointb) {$\bullet$};  
     \draw (.3,.5) node (pointbb) {$\bullet$};    
         \draw (2,0) node {};
          \draw (0,0) node (pointp) {$\bullet$};       
      \draw (0,-.4) node (pointq) {$\bullet$};        
             \draw (1.75,-.15)  node {};
    \draw (1.75,.25)  node {};
\draw (2.3,.8) node (top) {};       
\draw (.7,0) -- (.7,.6);
\draw (.3,0) -- (.3,.5); 
\draw [dotted] (.7,0) .. controls (.7,-1) and (.4,-2.2) .. (0,-.4);
\draw (0,-.4) -- (.3,0);
\draw (0,0) -- (.3,.5);
\draw (0,-.4) -- (0,0);
\draw (-.2,.15) node (x)  {$y'_j$};
\draw (-.2,-.1)  node  (xx) {$y_{j+1}$};
\draw (.4,.15) node (w) {$x'_j$};
\draw (.4,-.15) node  (ww)  {$x_{j+1}$};
\draw (-.2,-.4) node  (ww)  {$x'_{j+1}$};
\draw (.3,.65) node  (ww)  {$x_j$};
\draw (.8,.15) node (w) {$x_{k}$};
\draw (.85,-.15) node  (ww)  {$x'_{k-1}$};
\draw (.7,.75) node  (ww)  {$x'_{k}$};
  \end{tikzpicture}
\] 

In this picture, $x_j$ and $x_k'$ are actually somewhere on $S_X[M_0]$, but it is not important where. 
We can replace $x_j'$ with $y_j'$ and $x_{j+1}$ with $y_{j+1}$, the corner between $x_{j+1}$ and $x_{j+1}'$, since we assumed they are on adjacent sides. 
Note that $(m_j,y_j')\sim (m_{j+1},y_{j+1})$.

 Currently, the portion of the path from $(m_j,x_j)$ to $(m_{j+1},x_{j+1}')$ contributes $$d_{L\times X}((m_j,x_j),(m_j,x_j')) + d_{L\times X}((m_{j+1},x_{j+1}),(m_{j+1},x_{j+1}'))$$ to the score. So using the definition of the metric on $L\times X$, we get 
\setstretch{1.5} 
\[\begin{array}{rcll}
& & d_{L\times X}((m_j,x_j),(m_j,x_j')) + d_{L\times X}((m_{j+1},x_{j+1}),(m_{j+1},x_{j+1}'))&\\
&= & \onethird d_X(x_j,x_j') + \frac{1}{3} d_X(x_{j+1},x_{j+1}')& \\ 
& \geq & \frac{1}{3}d_{X}(x_j,x_j') + \frac{1}{3}d_{X}(x_{j+1},y_{j+1}) + \frac{1}{3}d_{ X}(y_{j+1},x_{j+1}') & (1)\\
&=& \frac{1}{3}d_{ X}(x_j,x_j') +  \frac{1}{3}d_{ X}(x_j',y_j')+\frac{1}{3}d_{ X}(y_{j+1},x_{j+1}')&(2)\\
&\geq &  \frac{1}{3}d_{X}(x_j,y_j') + \frac{1}{3}d_{ X}(y_{j+1},x_{j+1}')& (3)\\
& = & d_{L\times X}((m_j,x_j),(m_j,y_j')) + d_{L\times X}((m_{j+1},y_{j+1}),(m_{j+1},x_{j+1}')) & \\ 
\end{array}\]

\setstretch{1}
\noindent
where (1) is by  ($\sqtwo$) and Corollary~\ref{Sxisashortmap}, (2) is by Lemma~\ref{technicalquotientsuitablecorollary}(3), and (3) is by the triangle inequality in $X$.  
 Thus, we may replace $(m_j,x_j')\sim (m_{j+1},x_{j+1})$ in our path with $(m_j,y_j')\sim (m_{j+1},y_{j+1})$ to obtain a path with the same number of entries, 
a score which is shorter or equal to that of the original,  and such that $j$ and $k$ are no longer a bad configuration.  Note that our new path is still an alternating path, and since the replacement entry is a corner, performing this process cannot create a path which violates the second or third points in the statement of the lemma. 

Thus, we may proceed by induction on the number of bad configurations.  Perform the process described above to remove one instance of a bad configuration, then apply the induction hypothesis. 
\end{proof}

We have one more technical lemma involving elements of an alternating path sharing sides of $S_X[M_0]$. 

\begin{lemma}\label{theverylasttechnicalpathlemma}
Consider an alternating path from $(m,x)$ to $(n,y)$ in $L\times X$, 
\begin{equation}\label{anotheraltpath} (m,x) = (m_0,x_0),(m_0,x_0')\sim (m_1,x_1),\ldots, (m_{p-1},x_{p-1})\sim (m_p,x_p),(m_p,x_p') = (n,y).\end{equation}

Suppose that there are $i$ and $j$ with $0\leq i<j\leq p$ such that $m_i = m_j$.  Further suppose that one of the following holds:
\begin{itemize}
    \item $x_i$ and $x_j$ are on the same side of $S_X[M_0]$ as each other, and $x_i'$ and $x_j'$ are on the same side of $S_X[M_0]$ as each other, 
    \item $x_i$ and $x_j'$ are on the same side of $S_X[M_0]$ as each other, and $x_i'$ and $x_j$ are on the same side of $S_X[M_0]$ as each other.
    
\end{itemize}
Then we can delete the entries strictly between $(m_i,x_i)$ and $(m_j,x_j')$ in our alternating path to obtain another alternating path with a smaller or equal score to that of the original. 
\end{lemma}

\begin{proof}
First, suppose $x_i$ and $x_j$ are on the same side of $S_X[M_0]$.  Without loss of generality, suppose they are on the bottom and $x_i$ is to the left of $x_j$ (in case it is the opposite, we may just reverse the order of the path).

We will consider the cases when $x_i'$ and $x_j'$ are on the left side of $M_0$ (the right side is similar) and when they are on the top of $M_0$.

In the first case, $x_i' =S_X((0,s))$ and $x_j' = S_X((0,s'))$ for some $s,s'\in [0,1]$.  Either $s<s'$ or $s'<s$. 
First, suppose that $s<s'$.   We consider the diagram on the left below.
\[
   \begin{tikzpicture}[>=stealth',shorten >=1pt,auto,node distance=2cm,semithick,scale=2.5]   
\draw [help lines] (0,0) grid (1,1);
    \draw (.8,0) node (pointr) {$\bullet$};
        \draw (.8,-.15) node {$x_j$};
    \draw (.5,0) node (pointb) {$\bullet$};  
         \draw (.5,-.15) node {$x_i$};
         \draw (2,0) node {};
          \draw (0,.8) node (pointp) {$\bullet$};
             \draw (-.15,.8) node (pointp) {$x'_j$};       
      \draw (0,.4) node (pointq) {$\bullet$};   
       \draw (-.15,.4) node (pointp) {$x'_i$};       
             \draw (1.75,-.15)  node {};
    \draw (1.75,.25)  node {};
\draw (2.3,.8) node (top) {};       
\draw (.35,.35)  node {$e$};
\draw (.2,.1)  node {$a$};
\draw (.7,.2)  node {$c$};
\draw (.5,0) -- (0,.4);
\draw (.5,0) -- (0,.8);
\draw (.8,0) -- (0,.8);
\draw [dotted] (.8,0) .. controls (0,-1) and (-.4,-.5) .. (0,.4);
\draw (-.25,-.3) node {$b$};
  \end{tikzpicture}
\begin{tikzpicture}[>=stealth',shorten >=1pt,auto,node distance=2cm,semithick,scale=2.5]   
\draw [help lines] (0,0) grid (1,1);
    \draw (.8,0) node (pointr) {$\bullet$};
        \draw (.8,-.15) node {$x_j$};
    \draw (.5,0) node (pointb) {$\bullet$};  
         \draw (.5,-.15) node {$x_i$};
         \draw (2,0) node {};
          \draw (0,.8) node (pointp) {$\bullet$};
             \draw (-.15,.8) node (pointp) {$x'_i$};       
      \draw (0,.4) node (pointq) {$\bullet$};   
       \draw (-.15,.4) node (pointp) {$x'_j$};       
             \draw (1.75,-.15)  node {};
    \draw (1.75,.25)  node {};
\draw (2.3,.8) node (top) {};       
\draw (.25,.65)  node {$a$};
\draw (.2,.1)  node {$e$};
\draw (.6,.2)  node {$c$};
\draw (.5,0) -- (0,.4);
\draw (.5,0) -- (0,.8);
\draw (.8,0) -- (0,.4);
\draw [dotted] (.8,0) .. controls (0,-.8) and (-1.5,-.4) .. (0,.8);
\draw (-.75,-.3) node {$b$};
  \end{tikzpicture}
\]  
The entries in the path from $(m_i,x_i)$ to $(m_j,x_j')$ contribute $a+b+c$ to the score. 
Since the metric on $M_0$ is the path metric, and $S_X$ acts isometrically on adjacent sides by ($\sqtwo$) and Corollary~\ref{Sxisashortmap},  $e<c$, so $e< a+b+c$,  meaning we can delete the entries between $(m_i,x_i),(m_j,x_j')$  to obtain an alternating path (since $m_i=m_j$) with a smaller or equal score and with strictly fewer entries. 

The case when $s'<s$ is similar.  The picture is on the right above.
The entries from $(m_i,x_i)$ to $(m_j,x_j')$ contribute $a+b+c$ to the score of the path. Again, since the metric on $M_0$ is the path metric and $S_X$ acts isometrically on points on adjacent sides of $M_0$, $e<c$.  So $e<a+b+c$.  Thus, we can delete the entries between $(m_i,x_i),(m_j,x_j')$ to get an alternating path with strictly fewer entries whose score is less than or equal to that of the original. 

Finally, we consider the case when $x_i'$ and $x_j'$ are on the top of $M_0$.  It does not matter which is leftmost. 
\[   
\begin{tikzpicture}[>=stealth',shorten >=1pt,auto,node distance=2cm,semithick,scale=2.5]   
\draw [help lines] (0,0) grid (1,1);
    \draw (.8,0) node (pointr) {$\bullet$};
        \draw (.8,-.15) node {$x_j$};
    \draw (.5,0) node (pointb) {$\bullet$};  
         \draw (.5,-.15) node {$x_i$};
         \draw (2,0) node {};
          \draw (.8,1) node (pointp) {$\bullet$};
             \draw (.5,1.15) node (pointp) {$x'_i$};       
      \draw (.5,1) node (pointq) {$\bullet$};   
       \draw (.8,1.25) node (pointp) {$x'_j$};       
             \draw (1.75,-.15)  node {};
    \draw (1.75,.25)  node {};
\draw (2.3,.8) node (top) {};       
\draw (.4,.5)  node {$a$};
\draw (.9,.5)  node {$c$};
\draw (.7,.5)  node {$e$};
\draw (.68,.14)  node {$f$};
\draw (.5,0) -- (.5,1);
\draw (.5,0) -- (.8,1);
\draw (.8,0) -- (.8,1);
\draw [dotted] (.8,0) .. controls (1.5,-.5) and (1.5,1.5) .. (.5,1);
\draw (1.4,1) node {$b$};
  \end{tikzpicture}
\]  
The entries from $(m_i,x_i)$ to $(m_j,x_j')$ contribute $a+b+c$ to the score, and note that $a\geq 1$ and $c\geq 1$ by ($\sqtwo$). 
By the triangle inequality, $e\leq f+c$.
Since $x_i$ and $x_j$ are strictly on the bottom of the image of $M_0$ under $S_X$, by $(\sqone)$, $f<1$.  So $e<f+c<1+c\leq a+c \leq a+b+c$.  
Thus, we may delete the entries between $(m_i,x_i)$ to $(m_j,x_j')$ to get a path with strictly fewer entries and a smaller or equal score than that of the original. 

So this completes the first case. 

Now suppose $x_i$ and $x_j'$ are on the same side.  Note that we cannot simply apply Lemma \ref{enteringonsamesidelemma} because $x_i$ and $x_j'$ could be corners.   

As before, we will assume without loss of generality that $x_i$ and $x_j'$ are on the bottom and consider the cases when $x_i'$ and $x_j$ are on the left (the right is similar) and when they are on the top.

Again, $x_i' =S_X((0,s))$ and $x_j = S_X((0,s'))$ for some $s,s'\in [0,1]$.  Either $s<s'$ or $s'<s$. 
First, suppose that $s<s'$.   We consider the diagram on the left below.
\[
   \begin{tikzpicture}[>=stealth',shorten >=1pt,auto,node distance=2cm,semithick,scale=2.5]   
\draw [help lines] (0,0) grid (1,1);
    \draw (.8,0) node (pointr) {$\bullet$};
        \draw (.8,-.15) node {$x_j'$};
    \draw (.5,0) node (pointb) {$\bullet$};  
         \draw (.5,-.15) node {$x_i$};
         \draw (2,0) node {};
          \draw (0,.8) node (pointp) {$\bullet$};
             \draw (-.15,.85) node (pointp) {$x_j$};       
      \draw (0,.4) node (pointq) {$\bullet$};   
       \draw (-.15,.3) node (pointp) {$x'_i$};       
             \draw (1.75,-.15)  node {};
    \draw (1.75,.25)  node {};
\draw (2.3,.8) node (top) {};       
\draw (.65,.05)  node {$e$};
\draw (.2,.1)  node {$a$};
\draw (.5,.4)  node {$c$};
\draw (.5,0) -- (0,.4);
\draw (.8,0) -- (0,.8);
\draw (.5,0) -- (.8,0);
\draw [dotted] (0,.4) .. controls (-.5,.6) and (-.5,.7) .. (0,.8);
\draw (-.5,.65) node {$b$};
  \end{tikzpicture}
\quad  
   \begin{tikzpicture}[>=stealth',shorten >=1pt,auto,node distance=2cm,semithick,scale=2.5]   
\draw [help lines] (0,0) grid (1,1);
    \draw (.8,0) node (pointr) {$\bullet$};
        \draw (.8,-.15) node {$x'_j$};
    \draw (.3,0) node (pointb) {$\bullet$};  
         \draw (.3,-.15) node {$x_i$};
         \draw (2,0) node {};
          \draw (0,.8) node (pointp) {$\bullet$};
             \draw (-.15,.9) node (pointp) {$x'_i$};       
      \draw (0,.4) node (pointq) {$\bullet$};   
       \draw (-.15,.3) node (pointp) {$x_j$};       
             \draw (1.75,-.15)  node {};
    \draw (1.75,.25)  node {};
\draw (2.3,.8) node (top) {};       
\draw (.15,.65)  node {$a$};
\draw (.5,.05)  node {$e$};
\draw (.4,.3)  node {$c$};
\draw (.3,0) -- (.8,0);
\draw (.3,0) -- (0,.8);
\draw (.8,0) -- (0,.4);
\draw (-.85,.6) node {$b$};
\draw [dotted] (0,.4) .. controls (-1,.5) and (-1,.7) .. (0,.8);
  \end{tikzpicture}
\]  
The entries in the path from $(m_i,x_i)$ to $(m_j,x_j')$ contribute $a+b+c$ to the score. By ($\sqtwo$), $e \leq c$, so $e\leq a+b+c$, meaning we can delete the entries between $(m_i,x_i),(m_j,x_j')$  to obtain an alternating path (since $m_i=m_j$) with a smaller or equal score and with strictly fewer entries. 

The case when $s'<s$ is similar.   The diagram is on the right above.
The entries from $(m_i,x_i)$ to $(m_j,x_j')$ contribute $a+b+c$ to the score of the path. Again, $e\leq c$ by ($\sqtwo$).  Thus, we can delete the entries between $(m_i,x_i),(m_j,x_j')$ to get an alternating path with strictly fewer entries whose score is less than or equal to that of the original. 

Finally, we consider the case when $x_i'$ and $x_j$ are on the top of $M_0$.  It does not matter which is leftmost. 
Here is a picture:
\[   
\begin{tikzpicture}[>=stealth',shorten >=1pt,auto,node distance=2cm,semithick,scale=2.5]   
\draw [help lines] (0,0) grid (1,1);
    \draw (.8,0) node (pointr) {$\bullet$};
        \draw (.8,-.15) node {$x'_j$};
    \draw (.5,0) node (pointb) {$\bullet$};  
         \draw (.5,-.15) node {$x_i$};
         \draw (2,0) node {};
          \draw (.8,1) node (pointp) {$\bullet$};
             \draw (.5,1.15) node (pointp) {$x'_i$};       
      \draw (.5,1) node (pointq) {$\bullet$};   
       \draw (.8,1.15) node (pointp) {$x_j$};       
             \draw (1.75,-.15)  node {};
    \draw (1.75,.25)  node {};
\draw (2.3,.8) node (top) {};       
\draw (.4,.5)  node {$a$};
\draw (.9,.5)  node {$c$};

\draw (.68,.14)  node {$e$};
\draw (.5,0) -- (.5,1);
\draw (.8,0) -- (.8,1);
\draw [dotted] (.8,1) .. controls (1.5,-.5) and (1.5,2.3) .. (.5,1);
\draw (1.4,1) node {$b$};
  \end{tikzpicture}
\]  
The entries from $(m_i,x_i)$ to $(m_j,x_j')$ contribute $a+b+c$ to the score.  Note that $a\geq 1$ and $c\geq 1$ by ($\sqtwo$). 
Since $x_i$ and $x_j'$ are on the same side, by ($\sqone$), $e\leq 1$.  Thus, $e<a+b+c$, so we may delete the entries between $(m_i,x_i)$ to $(m_j,x_j')$ to get a path with strictly fewer entries and a smaller or equal score than that of the original. 
\end{proof}

\begin{lemma}\label{finitelymanypaths}
There exists a positive integer $K$
such that for every alternating path $$(m,x) =(m_0,x_0),(m_0,x')\sim(m_1,x_1),(m_1,x_1')\sim\ldots (m_p,x_p),(m_p,x_p')=(n,y)$$
with $p> K$, there exists an alternating path from $(m,x)$ to $(n,y)$ with smaller or equal score with strictly fewer entries. 
\end{lemma}

\begin{proof}

Fix $K = 36|L|+1$, which is finite since $L$ is finite.  We will see the justification for this choice of $K$ in the proof.

Suppose that we have an alternating path of the form above with $p>K$. 

First, we may assume that it does not have any repetitions, since if there are two entries which are equal (not just equivalent, but actually equal), then we can delete every entry between those and one of the two repeated entries to obtain a path with strictly fewer entries and a score which is less than or equal to that of the original.

By Lemma \ref{technicalquotientsuitablecorollary}(1),  all of the entries in an alternating path except for the first and last must be of the form $(m,S_X((r,s)))$ for some $(r,s)\in M_0$. So by ignoring the first and last entries (if necessary), we have a path $$(m_0,x_0')\sim(m_1,x_1),(m_1,x_1')\sim\ldots (m_p,x_p),$$
where each $x_i$ and $x_i'$ are in the image of $S_X[M_0]$.

Furthermore, by Lemma \ref{technicalpathlemma}, we may assume that if $x_i$ and $x_i'$ are not both corners, then they are not on the same side of $S_X[M_0]$. By Lemma \ref{enteringonsamesidelemma}, we may further assume that if $0\leq i<j\leq p$ and $m_i = m_j$, then $x_i'$ is not on the same side as $x_j$.

Since $p>36|L|+1$, the list $m_1,\ldots, m_{p-1}$ has at least $36|L|+1$ entries.  So by the pigeonhole principle, there is $m\in L$ such that for at least $37$ many $k$ (with $1\leq k\leq p-1$), $m_k = m$.  

Then consider the set $\{x_k\mid m_k =m,1\leq k\leq p-1\}$.  Since we assumed that there are no identical entries in our path, every element of this set is distinct, so this set has at least $37$ elements. 
We use the pigeonhole principle again.  Since $S_X[M_0]$ has four sides, there is a side of $S_X[M_0]$ containing at least $10$ of the $x_k$'s from this set. 
Fix this side.
Again, since the $x_k$'s under discussion
are all distinct, at most two of these are corners, so this side contains at least $8$ $x_k$'s which are not corners. Let $x_{k_1},\ldots, x_{k_8}$ be 8 of them.  
Consider the corresponding points $x_{k_1}',\ldots, x_{k_8}'$.  Since they are all distinct, there
are at least four which are not corners.  By our assumptions about our path, they cannot be on the same side as $x_{k_1},\ldots, x_{k_8}$.  Thus, we have $3$ sides where these four entries can be, so there must be one of the four sides such that some $x_{k_i}'$ and $x_{k_j}'$ are on the same side (and are both not corners). 
For ease of notation, we will just refer to these indices as $i$ and $j$. 
From here on out, we will assume without loss of generality that $x_i$ and $x_j$
are on the bottom side of the image of $M_0$ under $S_X$, that is, $\{(r,0): 0<r<1\}$.



Arrange $x_i$ and $x_j$ so that $i<j$.  Then, by Lemma \ref{theverylasttechnicalpathlemma}, we may delete some portion of our path to obtain an alternating path with strictly fewer entries whose score is at most that of the original path.  
\end{proof}

We are ready to prove that the quotient metric can be calculated as the score of some
particular finite path, not just an infimum over a set of paths. 
The assumption in Theorem~\ref{quotientmetric}  below is very mild; the idea is that the distance between points in the same copy of $X$ cannot be made shorter by going outside of $X$ on some other path.    It plays a key role in our connection of the Sierpinski carpet with iterations of a functor on square sets.\footnote{To understand the assumption,
it might help to look 
ahead to Sections~\ref{section-M-otimes} and~\ref{section-N-otimes} for
 the definitions of $L\otimes X$ for the sets $M$ and $N$
that we want most to take for $L$  and for the set $E$ underlying an equivalence relation $\sim$.}
As an example of why we need to make this assumption (in addition to our 
the requirement that $E$ be 
quotient suitable), consider the following example:  Let $L = \{a,b,c\}$ and let $X = M_0$.  Then define $E$ on $L\times M_0$ by 

\[\begin{array}{rcl}
(a,(r,1)) & E & (b,(r,0))\\
(b,(0,r)) & E & (c, (1,r))\\
(c,(0,r)) & E & (a,(1-r,0))\\
\end{array}\]
for $r\in [0,1]$.  
We can visualize this as in Figure~\ref{qquestion}.

\begin{figure}[ht]
\[
  \begin{tikzpicture}[>=stealth',shorten >=1pt,auto,node distance=2cm,semithick,scale=1.7]
  \draw [help lines] (0,0) grid (2,1);  
    \draw [help lines] (1,-1) grid (2,0); 
  \draw (.5,.5)  node {$c$};   
  \draw (1.5,.5)  node {$b$}; 
    \draw (1.5,-.5)  node {$a$};   
 \draw[<-,decorate,decoration=zigzag] (0,0) -- (0,1);  
 \draw[decorate,decoration=zigzag,->] (1,-1) -- (2,-1);   
\end{tikzpicture}
\]
\caption{Justification for the assumption in Theorem~\ref{quotientmetric}\label{qquestion}}
\end{figure}

\noindent
Then $(a,(1,0)),(a,((0,1))$ would be an alternating path with score $2$.  In addition
$$(a,(1,0))\sim (c,(0,0)),(c,(1,0))\sim (a,(0,1))$$
is also an alternating path with score $1$.  
The big idea is that with squares we do not need to consider situations where we allow gluing which requires ``twisting'' copies of $M_0$. In fact, gluing with a twist would create a situation where we could find shorter distances by going through different copies of the square. We resolve this via the hypotheses
in the following theorem.
(Incidentally, we have not defined the sets $M$ and $N$ yet to which we shall apply all of this general theory, but when we do define them, we will see 
that the hypotheses of the theorem just below are indeed satisfied by both $M$ and $N$.)
The reason that we do not address this at the level of defining quotient suitability is that there may exist other examples (such as triangles in the Sierpinski Gasket (see, e.g.~\cite{Bhat}))
where we would need to allow for this.  

\begin{theorem}\label{quotientmetric} 
Suppose that $L$ is a finite index set and $E$ is a quotient suitable equivalence relation on $L\times M_0$.  Further, 
suppose that for any $x,y\in X$ and $m\in L$, for any path $(m,x)=(m_0,x_0), \ldots, (m_p,x_p) = (m,y)$, 
\[ \onethird d_X(x,y) \leq \displaystyle{\sum_{k=0}^{p-1}}d_{L\times X}((m_k,x_k),(m_{k+1},x_{k+1})).\]  
That is, given two points
$(m,x)$ and $(m,y)$
in the same scaled copy
$\set{m}\times X$
of $X$, their distance in $L\otimes X$  
is at least $\frac{1}{3}d_X(x,y)$.  
(Equivalently, our assumption is that there is no   path 
in $L\times X$ from $(m, x)$ to $(m,y)$
with a score smaller than $\frac{1}{3}d_X(x,y)$.) 
Then for all $m\otimes x$ and $n\otimes y$, either
\begin{enumerate}
    \item  $d_{L\otimes X}(m\otimes x,n\otimes y) = 2$, or
\item For some alternating path from $(m,x)$ to $(n,y)$, 
\[d_{L\otimes X}(m\otimes x,n\otimes y) = \displaystyle{\sum_{k=0}^p} d_{L\times X}((m_k,x_k),(m_k,x_k')),\]

\end{enumerate}
  \end{theorem}

\proof
Let $m\otimes x$ and $n\otimes y$ in $L\otimes X$ be given.  If every path between them has score $2$, then $d_{L\otimes X}(m\otimes x,n\otimes y) = 2$.  Otherwise, consider an alternating path  from $(m,x)$ to $(n,y)$.
Lemma~\ref{distancelemma} shows that the distance from $m\otimes x$ to $n\otimes y$ is the 
infimum of the scores of alternating paths.
The point is that any path which is not alternating gives rise to a
alternating path with score that is at most the score of the original.

By Lemma~\ref{finitelymanypaths}, since $E$ is quotient suitable, there is a finite $K$ such that we only need to consider alternating paths $(m,x)=(m_0,x_0)\ldots (m_p,x_p')=(n,y)$ with $p\leq K$.  
Since there are only finitely many tuples from $L$ of length $\leq K+1$,
we need only show that for each $p\leq K$ and each fixed tuple
$m = m_0, m_1, \ldots, m_p = n$, the infimum of the scores of paths
involving this tuple (allowing the $x$'s to vary) is attained.

For $0\leq i \leq p-1$, let 
\[ C_i = (\set{m_i}\times S_{X}[M_0]) \times (\set{m_{i+1}} \times  S_{X}[M_0]).\]
Each $C_i$ is a compact set:
$M_0$ is compact, and $S_X$ is continuous (since it is a short map by Corollary \ref{Sxisashortmap}), so the image $S_{X}[M_0]$ is compact.
And thus, so is each set $\set{m_i}\times S_{X}[M_0]$.
So the following set $C^*$ is also compact:
\[
C^* =\{(m,x)\}\times C_0\times C_1\times\cdots\times C_{p-1}\times \{(n,y)\}
\]
Each element of $C^*$ is a tuple, and each gives us a path as in 
(\ref{L9}).  
In more detail, we can write an element of $C^*$ as
\begin{equation}\label{tuple}
((m_0,x_0),((m_0, x_0'),(m_1 , x_1)), \ldots, 
((m_{p-1} ,  x_{p-1}'),(m_p , x_p)),(m_p,x_p'))
\end{equation}
where $(m_0,x_0) = (m,x)$ and $(m_p,x_p') = (n,y)$
Again, the $m$'s are the ones
which we fixed above, and the $x$'s belong to $S_X[M_0]$.  

The path corresponding to this is the one with the same notation as in (\ref{L9}).
Moreover, every path as in (\ref{L9}) comes from 
an element of our set $C^*$.
Consider the function which takes an element of $C^*$ to the score of its corresponding path. 

This 
function is  continuous, so we have 
a continuous function $C^* \to \reals$.  Since $C^*$ is compact, this function
indeed attains its minimum value at some point, just as we want.
\endproof

\begin{definition}
Let $m\otimes x$ and $n\otimes y$ be points in $L\otimes X$.
A \emph{witness path} from  $m\otimes x$ to $n\otimes y$   is an alternating sequence of points
\[
 (m,x)=(m_0,x_0),(m_0,x_0')\sim (m_1,x_1),(m_1,x_1')\sim \ldots, (m_p,x_p), (m_p,x_p')
 \]
such that
\[d_{L\otimes X}(m\otimes x,n\otimes y) = \displaystyle{\sum_{k=0}^p} d_{M\times X}((m_k,x_k),(m_k,x_k')).\]
\label{definition-witness-path}
\end{definition}

 Our previous work shows us that the distances in $L\otimes X$  which are below the maximum distance $2$ are witnessed by a
single finite path, not just an infimum of an infinite set of paths.  This gives us the following: 

\begin{corollary}\label{quotientmetriccorollary}
Under the same assumption as in Theorem~\ref{quotientmetric},
\begin{enumerate}
    \item $L\otimes X$ is a metric space.
    \item For each $m\in L$, the function $x\mapsto m\otimes x:X\rightarrow L\otimes X$ is an injection.
    
   Moreover, for $x,y\in X$, $d_{L\otimes X}(m\otimes x,m\otimes y) = \frac{1}{3} d_X(x,y)$. 
\end{enumerate}
\end{corollary}

\proof
For the first assertion, assume that  $d_{L\otimes X}(m\otimes x,n\otimes y) =0$.  Then there exists a
     witness path  whose score is equal to $0$.
    The adjacent entries $(m_k,x_k), (m_{k+1}, x_{k+1})$  not related by 
    $\sim$ must then contribute $0$ to the score.  This only happens when $m_k = m_{k+1}$ and $x_k = x_{k+1}$.
    In this case, the entire path is a sequence completely related by $\sim$.
    So we have $(m,x)\sim(n,y)$.  Thus, 
 $m\otimes x=n\otimes y$. 
 
 The second point is immediate from the assumption in  
 Theorem~\ref{quotientmetric}.
\endproof

 As mentioned at the beginning of this section, we are not aiming to show that $L\otimes -$ is a functor (indeed, for $X\in \SquaMS$, $L\otimes X$ may not be a square set).  However, we will show here that for a morphism $f:X\rightarrow Y$ in $\SquaMS$, that we may define a function $L\otimes f:L\otimes X\rightarrow L\otimes Y$ by $f(m\otimes x) = m\otimes f(x)$ and that this is well-defined. We want to reiterate though that the function $L\otimes f$ generally will not have the properties required to be a morphism (e.g., it may not be a short map, and if there is a Square Set structure, might not preserve it).

\begin{lemma}\label{welldefinedmorphisms}  For $X,Y\in \SquaMS$ and $f:X\rightarrow Y$ a morphism in $\SquaMS$, the function $L\otimes f:L\otimes X\rightarrow L\otimes Y$ given by $L\otimes f(m\otimes x) = m\otimes f(x)$ is well defined. 
\end{lemma}

\begin{proof}
 Let $X,Y\in \SquaMS$ and $f:X\rightarrow Y$ be a morphism in $\SquaMS$.  Let $m\otimes x = n\otimes y$ in $L\otimes X$.  Since $m\otimes x=n\otimes y$, either $(m,x)=(n,y)$, in which case $f(m\otimes x) = f(n\otimes y)$, or $(m,x)\sim (n,y)$.  By Lemma~\ref{technicalquotientsuitablecorollary} (1), $x,y\in S_X[M_0]$, so $x = S_X((r,s))$ and $y= S_X((t,u))$ for some $(r,s),(t,u)\in M_0$.  Since $E$ does not depend on $X$, since $(m,S_X((r,s)))\sim (n,S_X((t,u)))$, we must have $(m,S_Y((r,s)))\sim (n,S_Y((t,u)))$ in $L\otimes Y$ as well. 

 Since $f$ is a morphism in $\SquaMS$, it preserves $S_X$, so $f(x) = S_Y((r,s))$ and $f(y) = S_Y((t,u))$.  Thus, $L\otimes f(m\otimes x) = m\otimes f(x) = m\otimes f(S_X((r,s))) = m\otimes S_Y((r,s)) \sim n\otimes S_Y((t,u)) =n\otimes f(S_X((t,u))) = n\otimes f(y)= L\otimes f(n\otimes y)$.  
    
\end{proof}

\subsection{Defining $M\otimes -$ for square metric spaces}
\label{section-M-otimes}

The last section dealt
with properties of the operation 
$X\mapsto L\otimes X$ which were presented in an abstract fashion.  Now it is time to be more concrete.   We 
take $L$ to be a particular set $M$
in this section, and also define a relation $E$
on $M\times M_0$ and show that it
is quotient suitable.  Then we will verify
the other hypotheses used in the results of the last section for this $M$ and $E$.
In a subsequent section, we do the same thing for a different set $N$ and a different relation $E$.

Let $M = \{0,1,2\}^2\setminus \{(1,1)\}$.  
Each $m=(i,j)\in M$ will indicate a (column, row) entry in the $3\times 3$ grid,
except that $(1,1)$ is missing.
 \[
   \begin{tikzpicture}[>=stealth',shorten >=1pt,auto,node distance=2cm,semithick,scale=1.2]
  \draw [help lines] (0,0) grid (3,3);
 \draw (.5,.5)  node {$(0,0)$};
  \draw (1.5,.5)  node {$(1,0)$};
   \draw (2.5,.5)  node {$(2,0)$};
    \draw (.5,1.5)  node {$(0,1)$};
   \draw (2.5,1.5)  node {$(2,1)$};
    \draw (.5,2.5)  node {$(0,2)$};
  \draw (1.5,2.5)  node {$(1,2)$};
   \draw (2.5,2.5)  node {$(2,2)$};
 \draw (1,1) -- (2,2);
 \draw (1,2) -- (2,1);
  \end{tikzpicture}
\]
The idea is that $m\in M$ will tell us where a scaled copy of an object $X$ in $\SquaMS$ will go.  
Our goal is to show that $X\mapsto M\otimes X$ is a functor on $\SquaMS$.  
We will use the results of the previous section to establish that $M\otimes X$ is a metric space. 

We will obtain $M\otimes X$ as a quotient space of $M\times X$.
Let $E$ be the equivalence relation generated by the following relation on $M\times M_0$ for  $r\in [0,1]$:
\begin{equation} \label{eq-E}
\begin{array}{lcl}
((0,0), (r,1))& E & ((0,1),(r,0))\\
((0,1), (r,1) )& E & ((0,2),(r,0))\\
((0,2),(1,r))& E & ( (1,2),(0,r))\\
((1,2),(1,r))& E &  ((2,2),(0,r))\\
\end{array}
\qquad
\begin{array}{lcl}
((2,2), (r,0))& E & ((2,1), (r,1))\\
((2,1), (r,0))& E & ( (2,0),(r,1))\\
((2,0), (0,r))& E & ((1,0), (1,r))\\
((1,0), (0,r))& E & ((0,0), (1,r))\\
\end{array}
\end{equation}
For any $X$ in $\SquaMS$ we 
 then define $\approx$ using (\ref{extendE}).  Finally, we take the 
equivalence relation generated by $\approx$
and call it $\sim$, just as in our more general work in the previous section.

\begin{lemma}\label{EisQuotientSuitable} $E$ on $M\times M_0$ is quotient suitable (Definition~\ref{quotientsuitabledefinition}).\end{lemma}

\begin{proof}
If we define $\kappa\subset (M\times D)^2$ by $\kappa((m,Y))=(n,Z)$ if and only if there are $y\in Y$ and $z\in Z$ such that $(m,y) E (n,z)$ and $m$ appears before $n$ in the lexicographic order on $M$, we see quickly that $\kappa$ satisfies the conditions in the definition of quotient suitable, and that the relation $E$ described in the definition coincides precisely with our relation $E$ on $M\times M_0$. 
\end{proof}
 
Next, we will see that $M\otimes X$ is a square set (and ultimately, a square metric space). 

Recall that square sets come with a function $S_X:M_0\to X$.  

\def\arraystretch{1.1}
Define $S_{M\otimes X}:M_0\to M\otimes X$ by
\begin{equation}\label{SMotimesXdef1} S_{M\otimes X}((0,r)) = 
\left\{ \begin{array}{lr}
      (0,0)\otimes S_X((0,3r)) & 0\leq r\leq \frac{1}{3} \\
      (0,1)\otimes S_X((0,3r-1)) & \frac{1}{3}\leq r\leq \frac{2}{3}\\
      (0,2)\otimes S_X((0,3r-2)) & \frac{2}{3}\leq r\leq 1\\
   \end{array}
   \right\}
\end{equation}

\begin{equation} 
\label{sisone}
S_{M\otimes X}((r,1)) = 
\left\{ \begin{array}{lr}
      (0,2)\otimes S_X((3r,1)) & 0\leq r\leq \frac{1}{3} \\
      (1,2)\otimes S_X((3r-1,1)) & \frac{1}{3}\leq r\leq \frac{2}{3}\\
      (2,2)\otimes S_X((3r-2,1)) & \frac{2}{3}\leq r\leq 1\\
   \end{array}
   \right\}
\end{equation}

\begin{equation} S_{M\otimes X}((1,r)) = 
\left\{ \begin{array}{lr}
      (2,0)\otimes S_X((1,3r)) & 0\leq r\leq \frac{1}{3} \\
      (2,1)\otimes S_X((1,3r-1)) & \frac{1}{3}\leq r\leq \frac{2}{3}\\
      (2,2)\otimes S_X((1,3r-2)) & \frac{2}{3}\leq r\leq 1\\
   \end{array}
   \right\}
\end{equation}

\begin{equation}\label{siszero} S_{M\otimes X}((r,0)) = 
\left\{ \begin{array}{lr}
      (0,0)\otimes S_X((3r,0)) & 0\leq r\leq \frac{1}{3} \\
      (1,0)\otimes S_X((3r-1,0)) & \frac{1}{3}\leq r\leq \frac{2}{3}\\
      (2,0)\otimes S_X((3r-2,0)) & \frac{2}{3}\leq r\leq 1\\
   \end{array}
   \right\}
\end{equation}
\def\arraystretch{1}

The idea is that each new side consists of $3$ copies of the corresponding side from $X$.  So far, $(M\otimes X, S_{M\otimes X})$ is a square \emph{set} (we know $S_{M\otimes X}$ is well-defined because of the identified segments in $\sim$ on $M\times X$).  
As before in (\ref{nottensor}), the metric $d_{M\times X}$ is
\begin{equation}\label{distancesInMotimesX}
     d_{M\times X}((m,x),(n,y)) = 
\left\{ \begin{array}{ll}
      \frac{1}{3}d_X(x,y) & \mbox{if $m = n$} \\
      2 & \mbox{otherwise}\\
   \end{array}
\right\}
\end{equation}

So the distance is scaled by $\frac{1}{3}$ inside of each copy of $X$, and otherwise, it is $2$ (the maximum distance).  Then we define the quotient metric on $M\otimes X$ as we did in the previous section, and so far, we know that it is a pseudo metric.   To apply Theorem~\ref{quotientmetric} and Corollary~\ref{quotientmetriccorollary} to show that $M\otimes X$ is a metric space, we need more details about the quotient metric.  In particular, we need to see that $d_{M\otimes X}(m\otimes x,m\otimes y) = \frac{1}{3} d_X(x,y)$ for $x,y\in X$ and $m\in M$. To achieve this, we will describe the paths in $M\otimes X$ in more detail.


\subsection{Classification of regular witness paths}

Let $X$ be a square metric space, and let $x, y\in M\otimes X$.
Recall the definition of a \emph{witness path}
in Definition~\ref{definition-witness-path}.
Such a path is 
an \emph{alternating path} from $x$ to $y$; it does not contain superfluous
visits to any entry, and its score is minimal over all paths from $x$ to $y$.

\begin{definition}\label{regularwitnesspath} We say that such a witness path is \emph{regular} if, in addition, its length (as a sequence of points)
is minimal over all witness paths from $x$ to $y$. \end{definition}  



It will be helpful to have a classification of regular witness paths.  But before that we will consider a few illustrative examples.

\begin{example}
\label{example-first-new}

Let $0\leq r\leq \frac{1}{3}$ and $\frac{1}{3}<s\leq \frac{2}{3}$,
and let $k\in[0,1]$.
Consider the following path in $M\times X$:
\begin{equation}\label{regularWP} 
\begin{array}{ll}
 & ((0,0),S_X((3r,0))), ((0,0),S_X((1,k))) \\
\sim & ((1,0),S_X((0,k))),((1,0),S_X((3s-1,0)))
\end{array}
\end{equation}

We check that as $k$ ranges over $[0,1]$,  
 the score of this path is minimized when $k = 0$, and the minimum score of such a path is $|r-s|$.

Here is the reasoning.
Let us draw a picture and introduce some notation.
In the figure below, $a$, $b$, and $k$ represent distances in $M\times X$ along
the evident line segments. 
\[
   \begin{tikzpicture}[>=stealth',shorten >=1pt,auto,node distance=2cm,semithick,scale=3]
  \node[cloud,
    fill = gray!10,
    minimum width = 5.2cm,
    minimum height = 3.92cm] (c) at (.5,.5) {};
   \node[cloud,
    fill = gray!10,
    minimum width = 5.2cm,
    minimum height = 3.92cm] (c) at (1.5,.5) {};   
\draw [help lines] (0,0) grid (2,1);
    \draw (.5,0) node (pointb) {$\bullet$};
            \draw (1.5,0) node (pointd) {$\bullet$};
         \draw (2,0) node {};
          \draw (1,.5) node (apex) {$\bullet$};
 \draw (.25,-.15)  node {$3 r$};
 \draw (.75,-.15)  node {$1-3r$};
  \draw (1.25,-.15)  node {$3s-1$};
 \draw (1.75,-.15)  node {};
  \draw (.7,.35)  node {$a$};
     \draw (1.05,.15)  node {$k$};
   \draw (1.3,.35)  node {$b$};
    \draw (1.75,.25)  node {};
   \draw[line width=2pt]   (.5,0) -- (1,.5);
    \draw [line width=2pt]  (1.5,0) -- (1,.5);
  \draw (0,-.5) node (pointa) {$((0,0), S_X((3r,0)))$};
  \path[->](pointa) [bend right] edge (pointb);
   \draw (2.5,-.5) node (pointc) {$((1,0), S_X((3s-1,0)))$};
     \path[->](pointc) edge [bend right] (pointd);
\draw (2.3,.8) node (top) {};     
 \draw (3.5,.8) node {$((0,0), S_X((1,k))) \sim ((1,0),S_X((0,k)))$};   
    \path[->](top) edge [bend right] (apex);
  \end{tikzpicture}
\]  
The path under discussion is shown.  It has score $a + b$.
Now the left endpoint of the path and the midpoint
have the same first component, $(0,0)$.
By (\ref{nottensor}), the distance between them is 
\[
\begin{array}{lcl}
a &  = &
d_{M\times X}(((0,0),S_X((3r,0))),((0,0),S_X((1,k)))) \\
& = & \onethird d_{X}(S_X((3r,0)),S_X((1,k)))\\
& \geq & \onethird - r + \onethird k.
\end{array}
\]
At the end we used the fact that
 $X$ is a square metric space: by ($\sqtwo$), 
 the distance above is at least the taxicab distance in the unit square between the 
 corresponding points,
 and this is $1-3r+k$.

Similar work shows us that $b \geq s-\onethird +\onethird k$.
Thus, the score of our path is $\geq s-r + \twothirds k$.   As a function of $k$, this is obviously minimized when $ k=0$.
When $k = 0$, the score is 
\[ \onethird |1-3r| + \onethird |3s-1| = \onethird ( |1-3r| + |3s-1|) = \onethird |1-3r + 3s -1|=|r-s|.\]
\end{example}

\begin{example}
    \label{example-second-new}

Let $0\leq r\leq \frac{1}{3}$ and $\frac{1}{3}<s\leq \frac{2}{3}$,
and note that
\[
\begin{array}{lcl}
S_{M\otimes X}((r,0))  & = &  (0,0)\otimes S_X((3r,0))\\ 
S_{M\otimes X}((s,0)) & = &  (1,0)\otimes S_X((3s-1,0)).
\end{array}
\]
Then we claim that  
\[d_{M\otimes X}(S_{M\otimes X}((r,0)), S_{M\otimes X}((s,0))) = |r - s|.\]
 \[
  \begin{tikzpicture}[>=stealth',shorten >=1pt,auto,node distance=2cm,semithick,scale=1.5]
  \draw [help lines]  grid (3,3);
    \draw (.3,0)  node (a) {$\bullet$};
        \draw (1.7,0)  node (b) {$\bullet$};
 \draw (1,1) -- (2,2);
\draw (1,2) -- (2,1);
\end{tikzpicture}
\qquad 
   \begin{tikzpicture}[>=stealth',shorten >=1pt,auto,node distance=2cm,semithick,scale=1.5]
  \draw [help lines]  grid (3,3);
    \draw (.3,0)  node (a) {$\bullet$};
        \draw (1.7,0)  node (b) {$\bullet$};
          \draw (1,.6)  node (c) {$\bullet$};   
\draw[line width=2pt] (.3,0) -- (1,.6);
\draw[line width=2pt] (1,.6) -- (1.7,0);
 \draw (1,1) -- (2,2);
\draw (1,2) -- (2,1);
\end{tikzpicture}
  \qquad  
  \begin{tikzpicture}[>=stealth',shorten >=1pt,auto,node distance=2cm,semithick,scale=1.5]
  \draw [help lines]  grid (3,3);
    \draw (.3,0)  node (a) {$\bullet$};
        \draw (1.7,0)  node (b) {$\bullet$};
\draw[line width=2pt] (.3,0) -- (.5,1);
\draw (.5,1) node {$\bullet$};
\draw[line width=2pt] (.5,1) -- (1,2);
\draw (1,2) node {$\bullet$};
\draw (2,2) node {$\bullet$};
\draw[line width=2pt] (1,2) -- (2,2);
\draw[line width=2pt] (2,2) -- (2,1);
\draw (2,1) node {$\bullet$};
\draw[line width=2pt] (2,1) -- (1.7,0);
\draw (1,1) -- (2,2);
\draw (1,2) -- (2,1);
\end{tikzpicture}
\]
Moreover, this same formula holds for points in the top (or bottom, or left, or right) edges of suitably neighboring squares
in $M\otimes X$.

Here is the reason.   The points involved are shown as on the left above.   To evaluate the distance in $M\otimes X$,
we return to $M\times X$ and consider paths between the points.  Since we want to minimize the score, by Lemma~\ref{distancelemma}, we can consider alternating paths.   
The most obvious such path would be as in the middle, where we add in a third point as shown.
Then the work we did in Example~\ref{example-first-new} shows 
that the score of such a path is at least $|r-s|$, and moreover that we can get a path with exactly this score
by taking the third point to be the corner.
But in this result, we need to consider other alternating paths besides this ``obvious one.'' 
To have a path of minimal score, we should not repeat points in the edges.  Indeed, we will see in our proof of Lemma \ref{basecaseforM} below, we cannot even repeat the elements of $M$.

One representative path would be the one shown on 
the right.  
In this path, the elements of $M$ (starting with the point in the bottom left) are
\[
(0,0), (0,1), (0,2), (1,2), (2,2), (2,1), (0,1)
\]
But the score of this path is  greater than the score of the path in the middle: each time one
crosses a square from side to side, the score adds $\onethird$ by $(\sqtwo)$.
So the score is at least $\frac{5}{3}$.
And $|r-s| \leq \twothirds$.

So in fact, the path (\ref{regularWP}) is a regular witness path, since it witnesses the distance, and we would not be able to obtain a shorter path since they are in different copies (so alternating path between them must have at least four entries). 
\end{example}

Now we will prove the conditions required to apply Theorem~\ref{quotientmetric}.
\begin{lemma}\label{basecaseforM} For any $x,y\in X$ and $m\in M$, for any path 
\begin{equation}
  \label{basecaseforMequation}  
(m,x)=(m_0,x_0), \ldots, (m_p,x_p) = (m,y),
\end{equation}
we have
\begin{equation}\label{basecaseforMequationtw}  
    \onethird d_X(x,y) \leq \displaystyle{\sum_{k=0}^{p-1}}d_{M\times X}((m_k,x_k),(m_{k+1},x_{k+1})).
\end{equation}   
\end{lemma}

\begin{proof}
 Fix $m$ throughout this proof.
By Lemma ~\ref{distancelemma},  since the path $(m,x),(m,y)$ has score $\leq \frac{2}{3} < 2$, we may assume that our path is an alternating path, 
and we write it as
$$(m,x)=(m_0,x_0),(m_0,x_0')\sim \ldots (m_{p-1},x_{p-1}')\sim (m_p,x_p),(m_p,x_p') = (m,y).$$
By Lemma~\ref{technicalpathlemma}, \ref{enteringonsamesidelemma}, and~\ref{theverylasttechnicalpathlemma}, we may also assume the following for $i, j$ such that $0\leq i<j\leq p$ :

\begin{itemize}
    \item If $x_i$ and $x_i'$ are on the same side of $S_X[M_0]$, then they are both corners,
    \item If $x_i'$ and $x_j$ are on the same side of $S_X[M_0]$, then at least one of them is a corner,
    
    \item 
    If $x_i$ and $x_j$ are on the same side of $S_X[M_0]$, then 
    $x_i'$ and $x_j'$ are \emph{not} on the same side of $S_X[M_0]$,
    \item 
    If  $x_i$ and $x_j'$ are on the same side of $S_X[M_0]$, then 
    $x_i'$ and $x_j$ are \emph{not} on the same side of $S_X[M_0]$.
\end{itemize}

We show by strong induction on the natural number $k\geq 1$ that for every 
path  as in (\ref{basecaseforMequation}) 
between points whose first coordinate is $m$, if 
$|\set{i : m_i = m}| = k$,  then the estimate in (\ref{basecaseforMequationtw}) holds.
So we fix $k\geq 1$, assume our result for numbers $<k$, and then show it for $k$.
We argue by cases on $k$.
If $k =1$, we have $p = 0$, and the path from $(m,x)$ to $(m,y)$ is just $(m,x),(m,y)$.  
Its length is $\frac{1}{3}d_X(x,y)$, by \eqref{distancesInMotimesX}.
When $k\geq 3$, the path in (\ref{basecaseforMequation}) has 
$m_{\ell} = m$ for some $0< \ell < p$.  
We cut this path into two subpaths, the part between $(m_0,x_0)$ and $(m_\ell, x_\ell)$,
 and the part from $(m_\ell, x_\ell)$ to $(m_p,x'_p)$. In both subpaths, the number of $j$ such that
 $m_j = m$ is $< k$.   
  So the 
induction hypothesis applies to the subpaths. 
By this and the triangle inequality, we show the desired inequality. 
\[
    \onethird d_X(x,y) \leq   \onethird d_X(x,x_\ell) +   \onethird d_X(x_\ell,y)
\leq    
    \displaystyle{\sum_{i=0}^{p-1}}d_{M\times X}((m_i,x_i),(m_{i+1},x_{i+1}))
\]

The remaining case is when $k = 2$.
Thus, we may assume that the only pairs $(m_i,x_i)$ in our path with $m_i = m$
are $m_0$ and $m_p$. 
By the conditions listed at the beginning of the proof, $x_0'$ and $x_p$ can only be on the same side of $S_X[M_0]$ if at least one of them is a corner.  However, by examining such a path, we can use a triangle inequality argument with ($\sqone$) to shorten the path.

As an example, we take $m = (1,0)$
and consider the following path (where $x,y$ are arbitrary elements of $X$, not necessarily in $S_X[M_0]$):

\[\begin{array}{rcl}
((1,0),x),((1,0),S_X((1,1)))&\sim& ((2,1),S_X((0,0))), \\
((2,1), S_X((\frac{1}{2},0)))&\sim &((2,0),S_X((\frac{1}{2},1))),\\ ((2,0),S_X((0,\frac{1}{2})))&\sim& ((1,0),S_X((1,\frac{1}{2}))), ((1,0),y)\\
\end{array}\]

\[\begin{tikzpicture}[>=stealth',shorten >=1pt,auto,node distance=2cm,semithick,scale=1.5]
  \draw [help lines]  grid (3,3);

    \draw (1.2, 0.2) node (f) {$y$};    
    \draw (1.5, 0.2) node (a) {$\bullet$};

    \draw[line width=2pt] (1.5,0.2) -- (2,0.5);
    \draw(2,0.5) node (b) {$\bullet$};
    \draw[line width =2pt] (2,0.5) -- (2.5, 1);
    \draw(2.5,1) node (c) {$\bullet$};
    \draw[line width=2pt] (2.5,1) -- (2,1);
    \draw(2,1) node (d) {$\bullet$};
    \draw[line width=2pt] (2,1)-- (1.5,0.8);
    \draw(1.5,0.8) node (e) {$\bullet$};
    \draw(1.2,0.8) node (g) {$x$};
\end{tikzpicture}
\qquad\begin{tikzpicture}[>=stealth',shorten >=1pt,auto,node distance=2cm,semithick,scale=1.5]
  \draw [help lines]  grid (3,3);

    \draw (1.2, 0.2) node (f) {$y$};    
    \draw (1.5, 0.2) node (a) {$\bullet$};

    \draw[line width=2pt] (1.5,0.2) -- (2,0.5);
    \draw(2,0.5) node (b) {$\bullet$};
    \draw[line width=2pt] (2,0.5) -- (2,1);
    \draw(2,1) node (d) {$\bullet$};
    \draw[line width=2pt] (2,1)-- (1.5,0.8);
    \draw(1.5,0.8) node (e) {$\bullet$};
    \draw(1.2,0.8) node (g) {$x$};
\end{tikzpicture}
\]
By ($\sqone$), 
\[\begin{array}{rl} &d_{M\times X}(((2,1),S_X((0,0))),((2,1),S_X((\frac{1}{2},0)))) \\
= &d_{M\times X}(((2,0),S_X((0,1))),((2,0),S_X((\frac{1}{2},1))))\\
\end{array}\]
Next, by the triangle inequality in $X$ and the definition  of $d_{M\times X}$, 
\[\begin{array}{rl}
& d_{M\times X}(((2,0),S_X((0,1))),((2,0),S_X((0,\frac{1}{2}))))\\
\leq & d_{M\times X}(((2,0),S_X((0,1))),((2,0),S_X((\frac{1}{2},1))))\\
+&  d_{M\times X}(((2,0), S_X((\frac{1}{2},1))),((2,0),S_X((0,\frac{1}{2}))))\\
\end{array}\]
So we have the path pictured on the right above and shown below:
\[\begin{array}{rcl}
((1,0),x),((1,0),S_X((1,1)))&\sim& ((2,0),S_X((0,1))), \\
((2,0),S_X((0,\frac{1}{2})))&\sim& ((1,0),S_X((1,\frac{1}{2}))), ((1,0),y)\\
\end{array}\]
It has fewer entries and score at most that of the original. 
And then we use the same type of argument again.  By ($\sqone$), 
\[\begin{array}{rl}
& d_{M\times X}(((2,0),S_X((0,1))),((2,0),S_X((0,\frac{1}{2}))))\\
= & d_{M\times X}(((1,0),S_X((1,1))),((1,0),S_X((1,\frac{1}{2}))))\\
\end{array}\]
So using the triangle inequality twice, we get that $d_{M\times X}(((1,0),x),((1,0),y))$ is less than or equal to the score of the path, as required.  

We also need to consider paths which enter and exit the $m$-copy of $X$ on (strictly) different sides.  Such a path will have at least as many entries as one of the two following possibilities (up to rotation): 
 \[
   \begin{tikzpicture}[>=stealth',shorten >=1pt,auto,node distance=2cm,semithick,scale=1.5]
  \draw [help lines]  grid (3,3);
    \draw (.3, .6)  node (a) {$\bullet$};
        \draw (.6,.3)  node (b) {$\bullet$};
          \draw (1,.5)  node (c) {$\bullet$};   
\draw[line width=2pt] (.3,.6) -- (.3,1);
\draw (.3,1) node (d) {$\bullet$};
\draw[line width=2pt] (.3, 1) -- (.6, 2);
\draw (.6, 2) node (d) {$\bullet$};
\draw (1,2.4) node (e) {$\bullet$};
\draw[line width=2pt] (.6,2) -- (1,2.4);
\draw (2,2.5) node (f) {$\bullet$};
\draw[line width=2pt] (1,2.4) -- (2,2.5);
\draw (2.4,2) node (g) {$\bullet$};
\draw[line width=2pt] (2,2.5) -- (2.4, 2);
\draw (2.2, 1) node (h) {$\bullet$};
\draw[line width=2pt] (2.4,2) -- (2.2,1);
\draw (2,.7) node (i) {$\bullet$};
\draw[line width=2pt] (2.2, 1) -- (2, .7);
\draw[line width=2pt] (2,.7) -- (1,.5);
\draw[line width=2pt] (1,.5) -- (.6, .3);
 \draw (1,1) -- (2,2);
\draw (1,2) -- (2,1);
\end{tikzpicture}
  \qquad  
     \begin{tikzpicture}[>=stealth',shorten >=1pt,auto,node distance=2cm,semithick,scale=1.5]
  \draw [help lines]  grid (3,3);
    \draw (1.3, .2)  node (a) {$\bullet$};
        \draw (1.6,.3)  node (b) {$\bullet$};
          \draw (1,.5)  node (c) {$\bullet$};   

\draw (.3,1) node (d) {$\bullet$};
\draw[line width=2pt] (.3, 1) -- (.6, 2);
\draw (.6, 2) node (d) {$\bullet$};
\draw (1,2.4) node (e) {$\bullet$};
\draw[line width=2pt] (.6,2) -- (1,2.4);
\draw (2,2.5) node (f) {$\bullet$};
\draw[line width=2pt] (1,2.4) -- (2,2.5);
\draw (2.4,2) node (g) {$\bullet$};
\draw[line width=2pt] (2,2.5) -- (2.4, 2);
\draw (2.2, 1) node (h) {$\bullet$};
\draw[line width=2pt] (2.4,2) -- (2.2,1);
\draw (2,.7) node (i) {$\bullet$};
\draw[line width=2pt] (2.2, 1) -- (2, .7);
\draw[line width=2pt] (.3,1) -- (1,.5);
\draw[line width=2pt] (1,.5) -- (1.3,.2);
\draw[line width=2pt] (1.6,.3) -- (2, .7);

 \draw (1,1) -- (2,2);
\draw (1,2) -- (2,1);
\end{tikzpicture}
\]

In both of these cases, due to $(\sqtwo)$, each of these paths will have score $\geq 1$, since each time we have a segment which goes between opposite sides of a copy of $X$, we contribute at least $\frac{1}{3}$ to the score.  So the score of a path of this form will be at least $1 \geq\frac{1}{3} d_X(x,y)$.  
\end{proof}

We rephrase the result as follows: 

\begin{corollary}\label{basecaseforregularwitnesspaths}
    For any $x,y\in X$ and $m\in M$, $d_{M\otimes X}(m\otimes x,m\otimes y) = \frac{1}{3}d_X(x,y)$.  
\end{corollary}

\begin{proof} By Lemma~\ref{basecaseforM} and the fact that $(m,x),(m,y)$ is a path.\end{proof}

\begin{corollary}\label{MotimesXisametricspace}

    $M\otimes X$ is a metric space.
\end{corollary}
\begin{proof}
   This follows from Corollary~\ref{quotientmetriccorollary}.
\end{proof}

Now we turn our attention to understanding the quotient metric in more detail. 

\begin{theorem}
\label{theorem-regular-witness-paths}
Let $X$ be a square metric space.  Let $x,y\in M\otimes X$.
Then there exists a regular witness path between $x$ and $y$, and every regular witness path between $x$ and $y$ looks like one of the paths
shown in Figure~\ref{figure-regular-witness-paths}.
\end{theorem}

\begin{figure}[t]
\[
  \begin{tikzpicture}[>=stealth',shorten >=1pt,auto,node distance=2cm,semithick,scale=1.5]
  \draw [help lines] (0,4) grid (1,5);
 \draw (.3,4.6)  node  (a1) {$\bullet$};
  \draw (.5,4.8)  node   {$x$};
  \draw (.7,4.2)  node  (a2) {$\bullet$};
  \draw (.5,4.1)  node   {$y$};
\draw[line width=2pt] (.3,4.6) -- (.7,4.2);
  \draw [help lines] (2,3) grid (3,5);
 \draw (2.5,4.5)  node (a3)  {$\bullet$};
  \draw (2.3, 4.5) node {$x$};
  \draw (2.1,3.1)  node  (a4) {$\bullet$};
  \draw (2.3, 3.1) node {$y$};  
  \draw (2.8,4) node (a5) {$\bullet$};
\draw[line width=2pt] (2.5,4.5) --(2.8,4);
\draw[line width=2pt] (2.1,3.1) -- (2.8,4);  
  \draw [help lines] (4,3) grid (5,5);  
    \draw [help lines] (5,4) grid (6,5);
   \draw (4.5,3.5)  node  (k1)  {$\bullet$}; 
   \draw (4.65,3.4) node {$y$};
      \draw (5.5,4.5)  node (k2)   {$\bullet$};  
      \draw (5.7, 4.5) node {$x$};
   \draw (5,4) node    (k3) {$\bullet$};
\draw[line width=2pt](4.5,3.5) -- (5,4);
\draw[line width=2pt] (5,4) -- (5.5,4.5);
\draw [help lines] (7,2) grid (10,5);  
   \draw (7.2,4.6)  node (h1)   {$\bullet$};
   \draw (7.2, 4.8) node {$x$};
  \draw (8.5,2.5)  node (h2)  {$\bullet$};
  \draw (8.5,2.3) node {$y$};  
    \draw (8,3)  node (h3)  {$\bullet$};  
     \draw (8,4)  node (h4)  {$\bullet$};  
\draw[line width=2pt] (8.5,2.5) -- (8,3); 
\draw[line width=2pt] (8,3) --(8,4); 
\draw[line width=2pt] (8,4) -- (7.2,4.6); 
\draw (8,3) -- (9,4);
\draw (9,3) -- (8,4);
  \draw (9,4)  node (h5)  {$\bullet$};
   \draw (9,3)  node (h6)  {$\bullet$};
 \draw[line width=2pt] (9,4) --(9,3);
 \draw[line width=2pt] (9,3) -- (8.5,2.5);
\draw[line width=2pt] (8,4) -- (9,4);
  \end{tikzpicture}
\]

\[
  \begin{tikzpicture}[>=stealth',shorten >=1pt,auto,node distance=2cm,semithick,scale=1.5]
  \draw [xshift=-0.5cm,help lines] (4,0) grid (7,3);    
    \draw (4.1,2.7)  node (n2)   {$\bullet$};
    \draw (4.3, 2.7) node {$x$};   
      \draw (4.2,0.4)  node  (n1)   {$\bullet$};
      \draw (4.2,0.2) node {$y$}; 
\draw (4.3,1) node (n3) {$\bullet$};
   \draw (4.3,2) node (n4) {$\bullet$}; 
\draw[line width=2pt](4.2,0.4)  --  (4.3,1);
\draw[line width=2pt]  (4.3,1) -- (4.3,2);
\draw[line width=2pt] (4.3,2) -- (4.1,2.7);
\draw (4.5,1) -- (5.5,2);
     \draw (5.5,1) -- (4.5,2);      
  \draw [help lines] (7,0) grid (10,3);    
    \draw (8.6,2.5)  node (m2)   {$\bullet$};
    \draw (8.6, 2.7) node {$x$};
      \draw (8.4,0.5)  node  (m1)   {$\bullet$};
      \draw (8.4,0.3) node {$y$}; 
  \draw (8,1) node (m3) {$\bullet$};
   \draw (8,2) node (m4) {$\bullet$}; 
\draw[line width=2pt] (8.4,0.5) -- (8,1);
\draw[line width=2pt] (8,1) -- (8,2);
\draw[line width=2pt] (8,2) -- (8.6,2.5);
\draw (8,1) -- (9,2);
      \draw (9,1) -- (8,2);
   \draw [xshift=0.5cm,help lines] (10,0) grid (13,3);
   \draw (11,2.5)  node  (j2)   {$\bullet$};
    \draw (10.8, 2.5) node {$x$};   
      \draw (13.1,0.6)  node  (j1)   {$\bullet$};
      \draw (13.3,0.6) node {$y$};  
           \draw (12.5,0.2)  node  (j3)   {$\bullet$};  
     \draw (11.5,2)  node  (j5)   {$\bullet$}; 
   \draw (11.5,1) node (j4) {$\bullet$};   
 \draw[line width=2pt] (13.1,0.6) -- (12.5,0.2);  
 \draw[line width=2pt] (12.5,0.2) -- (11.5,1); 
 \draw[line width=2pt] (11.5,1) -- (11.5,2);
  \draw[line width=2pt](11.5,2) -- (11,2.5);
  \draw (11.5,1) -- (12.5,2);
    \draw (12.5,1) -- (11.5,2);
 \end{tikzpicture}
\]
\caption{Regular witness paths in $M\otimes X$.
\label{figure-regular-witness-paths}}
\end{figure}

\begin{remark}\label{remark-tensoring}
Theorem~\ref{theorem-regular-witness-paths} is stated somewhat loosely,
   but we believe that a patient reader could make it completely precise, 
   and also that it is more comprehensible to state it the way we do.
  Here is a bit more about what we mean.   We are aiming at a classification
  of all of the regular witness paths between pairs of points in $M\otimes X$.
  The first case is where $x$ and $y$ are in the same copy of $X$;
  that is, there is some $m\in M$ such that $x,y\in m\otimes X :=\{m\otimes x:x\in X\}$, which is what we proved in Corollary~\ref{basecaseforregularwitnesspaths}.
  In this case, our result is that every regular witness path stays inside
  $m\otimes X$.  The second case is when $x$ and $y$ are in adjacent copies.
  In this case, our result is very similar to what was shown in Example~\ref{example-second-new}.
The next case is when $x$ and $y$ lie in diagonally connected squares, such as $(0,1)\otimes X$
and $(1,0)\otimes X$.  In this case, the result is that the only regular witness paths are 
the ones that go through the shared corner, as shown.

Continuing, we have pairs of points in squares related by ``a move of the chess knight''.
In this case, there are two possible ``shapes'' that a regular witness path could have,
indicated by the two paths from $x$ to $y$.  Both go through the upper-left 
corner of the ``hole'',
but they differ after that.  For different $X$, $x$, and $y$, a regular witness path might 
look like one or the other of these paths; in general, we do not have enough information to tell.
And in some sense, we do not need to tell.   We only need a classification of what the 
minimal witness paths look like, and this is the topic of our theorem.

Then we have the case of squares on opposite sides, such as $(0,0)\otimes X$ and $(2,0)\otimes X$,
or $(1,0)\otimes X$ and $(2,0)\otimes X$.  The interesting point here is that this case
splits into two subcases, depending on whether one must ``navigate around the central hole''
due to the fact that $(1,1)\notin X$.
 Finally, we have the case of squares on opposite corners: 
 $(0,0)\otimes X$ and $(2,2)\otimes X$, or  $(2,0)\otimes X$ and $(0,2)\otimes X$.
 In this case, there is no need to indicate another path around the hole, since we 
 are only working ``up to rotation/reflection'', and the other path is a rotation of the path shown.
 
\end{remark}

\begin{proof}
 First note that for any $m\otimes x$ and $n\otimes y$ in $M\otimes X$, there exists an alternating path between then by the way that $E$ is defined. Consider such a path, 

$$(m,x)=(m_0,x_0),(m_0,x_0')\sim\ldots \sim (m_p,x_p),(m_p,x_p')=(n,y).$$

We will show by induction on $p$ that there is a regular witness path of one of the forms indicated in Figure~\ref{figure-regular-witness-paths} whose score is less than or equal to that of the path.

For $p = 0$, by Remark \ref{length0path}, $(m,x)=(n,y)$, or $(m,x)=(m_0,x_0),(m_0,x_0')=(n,y)$, so by Corollary~\ref{basecaseforregularwitnesspaths}, the regular witness path is $(m,x),(m,y)$, which is the first entry in Figure~\ref{figure-regular-witness-paths}.

 Let us assume our result for $p$ and prove it for $p+1$.  

Before we do this, we will check by inspection that in each of the cases in the figure, if $x$ or $y$ is in $m\otimes S_X[M_0]$ (that is, is on the boundary of its copy of $X$), and we add one more point in an adjacent copy, then we obtain another case from Figure~\ref{figure-regular-witness-paths}, or we can obtain a path with a smaller score by replacing it with one of the other cases in Figure~\ref{figure-regular-witness-paths}. 

For example, in the first entry in the figure, if $y = m\otimes (r,0)$, that is, it is on the bottom boundary of its copy of $X$, and we add on $z$ in the copy below, then we get an instance of the second entry in Figure~\ref{figure-regular-witness-paths}.  

For a more involved example, in the bottom left entry, suppose $y=(0,0)\otimes (1,\frac{1}{2})$, so it is on the right boundary of its copy of $X$.  Then suppose that we add on another entry $z=(1,0)\otimes (\frac{1}{2},\frac{1}{2})$.

\[
  \begin{tikzpicture}[>=stealth',shorten >=1pt,auto,node distance=2cm,semithick,scale=1.5]
  \draw [xshift=-0.5cm,help lines] (4,0) grid (7,3);    
    \draw (4.1,2.7)  node (n2)   {$\bullet$};
    \draw (4.3, 2.7) node {$x$};   
      \draw (4.5,0.4)  node  (n1)   {$\bullet$};
      \draw (4.2,0.2) node {$y$}; 
      \draw (5,0.4) node (n5) {$\bullet$};
      \draw (5,0.2) node {$z$};
\draw (4.3,1) node (n3) {$\bullet$};
\draw[line width=2pt] (4.5,0.4)  --(5,0.4);
   \draw (4.3,2) node (n4) {$\bullet$}; 
\draw[line width=2pt] (4.5,0.4) -- (4.3,1);
\draw[line width=2pt] (4.3,1) -- (4.3,2);
\draw[line width=2pt] (4.3,2) -- (4.1,2.7);
\draw (4.5,1) -- (5.5,2);
     \draw (5.5,1) -- (4.5,2);      
  \draw [help lines] (7,0) grid (10,3);    
    \draw (7.6,2.7)  node (m2)   {$\bullet$};
    \draw (7.8, 2.7) node {$x$};
      \draw (8.4,0.5)  node  (m1)   {$\bullet$};
      \draw (8.4,0.3)node {$z$}; 
  \draw (8,1) node (m3) {$\bullet$};
   \draw (8,2) node (m4) {$\bullet$}; 
\draw[line width=2pt](8.4,0.5) -- (8,1);
\draw[line width=2pt] (8,1) -- (8,2);
\draw[line width=2pt] (8,2) -- (7.6,2.7);
\draw (8,1) -- (9,2);
      \draw (9,1) -- (8,2);
 \end{tikzpicture}
\]

We can replace the path on the left with the path on the right, and using an argument similar to that in Example~\ref{example-second-new}, we see that this has a smaller or equal score.  The path on the right is an instance of one of the paths in the top right entry of Figure~\ref{figure-regular-witness-paths}.  The rest of the cases are similar.

From here, if we consider a path with $p+1$ entries and remove one entry, by the induction hypothesis, we can replace it with one of the paths from Figure~\ref{figure-regular-witness-paths} without increasing the score.  Then when we add it back, either we obtain one of the paths from the figure, or, as we argued, we can find a path from Figure~\ref{figure-regular-witness-paths} with a shorter or equal score. 
\end{proof}

\subsection{$M\otimes X$ as a functor on square metric spaces}

We restate Corollary~\ref{basecaseforregularwitnesspaths} with a little more information which will be useful later on when we apply $M\otimes-$ $k$ many times. 
We will always denote this repeated application by $M^k\otimes-$, and similarly for $N$, defined later on in Section \ref{section-N-otimes}. 

\begin{corollary}\label{distanceinonecopyM}  For an object $X$ in $\SquaMS$, $m\in M$, and $x,y\in X$, 
\[ d_{M\otimes X}(m\otimes x,m\otimes y ) = \onethird d_X(x,y).\]
In particular, $d_{M\otimes X}(m\otimes x, m\otimes y) \leq \frac{2}{3}$, and
for all $k$,
\[ d_{M^k\otimes X}(m_1\otimes\ldots \otimes m_k\otimes x,m_1\otimes\ldots\otimes m_k\otimes y) \leq \tfrac{2}{3^k}.
\]
\end{corollary}

Corollary~\ref{MotimesXisametricspace} tells us that
for every object $X$ in $\SquaMS$, 
$M\otimes X$ is a metric space, and we have defined $S_{M\otimes X}:M_0\rightarrow M\otimes X$ ((\ref{SMotimesXdef1})-(\ref{siszero})).    
We next  check that $S_{M\otimes X}$ satisfies the non-degeneracy requirements.

\begin{lemma}\label{nondegen} $M\otimes X$ with $S_{M\otimes X}$ satisfies ($\sqone$) and ($\sqtwo$). 
\end{lemma}

\proof
To show ($\sqone$), without loss of generality, we will examine $S_{M\otimes X}((r,0))$
and  $S_{M\otimes X}((s,0))$  for $r,s\in [0,1]$ (since all of the sides will behave the same way).  
We show \begin{equation}\label{rs}
d_{M\otimes X}(S_{M\otimes X}((r,0)),S_{M\otimes X}((s,0))) = |s-r|.\end{equation}

First, suppose that $S_{M\otimes X}((r,0))$ and $S_{M\otimes X}((s,0))$ are in the same copy of $X$.
(This means that there is at least one $m\in M$ 
such that these points belong to $m\otimes X$.
 Our result follows immediately from 
 Corollary~\ref{distanceinonecopyM} 
and (\ref{distancesInMotimesX}).  

Next, suppose that $S_{M\otimes X}((r,0))$ and $S_{M\otimes X}((s,0))$ are in adjacent copies of $X$.  
One way that this could happen would be when
$0\leq r\leq \frac{1}{3}$ and $\frac{1}{3}<s\leq \frac{2}{3}$.  
In this case,  Example~\ref{example-second-new}
shows that the distance is $|r-s|$. 
The details in all other cases are similar, and we omit them.

It remains to check that this holds for $S_{M\otimes X}((r,0))$ and $S_{M\otimes X}((s,0))$ in non-adjacent copies of $X$ (that is, $0\leq r\leq \frac{1}{3}$ and $\frac{2}{3}\leq s\leq 1$).
We note again that we have a path with length $|s-r|$ via the bottom corners of the $(1,0)$ copy of $X$. 
Another option that stays ``on the bottom'' is shown below:
\[
\begin{tikzpicture}[>=stealth',shorten >=1pt,auto,node distance=2cm,semithick,scale=1.5]
  \draw [help lines]  grid (3,3);
    \draw (.3,0)  node (a) {$\bullet$};
        \draw (2.7,0)  node (b) {$\bullet$};
\draw[line width=2pt] (.3,0) -- (1,1);
\draw (1,1) node {$\bullet$};
\draw[line width=2pt] (1,1) -- (2,1);
\draw (2,1) node {$\bullet$};
\draw[line width=2pt] (2,1) -- (2.7,0);
\draw (1,1) -- (2,2);
\draw (1,2) -- (2,1);
\end{tikzpicture}
\]
But then an argument using the fact that $X$ satisfies ($\sqtwo$)
shows that the score on the path shown above is at least as large as the score mentioned
above, $|r-s|$. 

If we go the other way around $M\otimes X$, our path will have score greater than $1$ because of the non-degeneracy requirement, and using a triangle inequality argument similar to the adjacent copies case, we see that this is minimized by going through the corners. 

Now to check ($\sqtwo$), let $(r,s)$ and $(t,u)$ in $M_0$ be given, and consider a 
witness path between $S_{M\otimes X}((r,s))$ and $S_{M\otimes X}((t,u))$.  First note that each pair of entries in the path contributing positively to the score will be on the sides of a copy of $X$, so we can take advantage of ($\sqtwo$) in $X$.  We will show that the sum of horizontal and vertical components of each entry of our path will be at least the sum of the horizontal and vertical components of the distance between $(r,s)$ and $(t,u)$, and thus, our distance will be bounded below by the taxicab metric. 

The different cases for relative placement of $(r,s)$ and $(t,u)$ are similar, so we will examine $S_{M\otimes X}((0,s))$ and $S_{M\otimes X}((r,0))$ with $\frac{1}{3}<r, s< \frac{2}{3}$ in detail.  

 \[
   \begin{tikzpicture}[>=stealth',shorten >=1pt,auto,node distance=2cm,semithick,scale=1.5]
  \draw [help lines]  grid (3,3);
 \draw (1.5,0)  node (a)  {$\bullet$};
  \draw (1,.5)  node {$\bullet$};
    \draw (.5,1)  node {$\bullet$};
   \draw (0,1.5)  node  (b) {$\bullet$};
     \draw (-1,1.5)  node {$S_{M\otimes X}((0,s))$};
  \draw (1.5,-.5)  node {$S_{M\otimes X}((r,0))$};
  \draw (0.75, 1.2) node {$A$};
   \draw (1.3, 0.5) node {$B$};
\draw[line width=2pt] (1.5,0) -- (0,1.5);
  \end{tikzpicture}
\]  

By examining cases, we can show that a shortest path will be of the form pictured (though we may have $A =B = (0,0)\otimes S_X((1,1))$).  Then by ($\sqtwo$) in $X$, we see that   
its length is
\[d_{M\otimes X}(S_{M\otimes X}((0,s)), A) + d_{M\otimes X}(A,B) + d_{M\otimes X}(B,S_{M\otimes X}((r,0))).\]
Note that $S_{M\otimes X}((0,s)) = (0,1)\otimes S_X((0,3s-1))$ and $S_{M\otimes X}((r,0)) = (1,0)\otimes S_X((3r-1,0))$.  Let 
\[
\begin{array}{lclcl}
A & =&  (0,1)\otimes S_X((t,0)) & = & (0,0)\otimes S_X((t,1))\\
B & =& (0,0)\otimes S_X((1,u)) & = & (1,0)\otimes S_X((0,u))
\end{array}
\]
Then the distance is 
\def\arraystretch{1.1}
\[
\begin{array}{cll}
& \frac{1}{3}d_X(S_X((0,3s-1)), S_X((t,0))) + \frac{1}{3}  d_X(S_X((t,1)),S_X((1,u)))\\
 & \qquad + \frac{1}{3}d_X(S_X((0,u)),S_X((3r-1,0)))\\

\geq&  \frac{1}{3}(|3s-1-0| + |t-0| + |1-t| + |1-u| + |u-0| + |3r-1-0|)
\\
   = & \frac{1}{3}(3|s-0| + 3|r-0|)\\
   = & |s-0| + |r-0|,
\end{array}
\]
as required.
In the first inequality, we used
the fact that $X$ is an object in $\SquaMS$. 
\endproof

\def\arraystretch{1}

At this point we know that $M\otimes X$ is an object in $\SquaMS$. 
That is, we know how the functor $M\otimes -$ works on objects
of $\SquaMS$.
Now let $f:X\rightarrow Y$ be a morphism  in 
$\SquaMS$, and define $M\otimes f:M\otimes X\rightarrow M\otimes Y$ by $M\otimes f(m\otimes x) = m\otimes f(x)$.  By Lemma \ref{welldefinedmorphisms}, we know that $M\otimes f$ is well defined.  
To check that $M\otimes f:M\otimes X\rightarrow M\otimes Y$ is a morphism, first note that for $(r,s)$ on the boundary of the unit square, 
\[ (M\otimes f)(S_{M\otimes X}((r,s))) = 
(M\otimes f)(m\otimes S_X((r',s')))\]
for some $m\in M$ and $(r',s')\in M_0$, by the definition of $S_{M\otimes X}$.
This equals
\[
\begin{array}{cll}
 &  m\otimes f(S_X((r',s')))\\
= & m\otimes S_Y((r',s'))&
\mbox{since $f$ preserves $S_X$}\\
= & S_{M\otimes Y}((r,s))\\
\end{array}
\]
The last equality holds since $S_{M\otimes Y}$ is defined using the same scheme as $S_{M\otimes X}$. 
So $M\otimes f$ preserves $S_{M\otimes X}$.  

To see that $M\otimes f$ is a short map, let 
$m\otimes x, n\otimes y\in M\otimes X$.  If $d_{M\otimes X}(m\otimes x,n\otimes y) =2$, then $d_{M\otimes X}(m\otimes x,n\otimes y)=2\geq d_{M\otimes Y}(M\otimes f(m\otimes x),M\otimes f(n\otimes y))$.  Otherwise, let 
\[m\otimes x = m_0\otimes x_0,\ldots,m_p\otimes x_p = n\otimes y\]
be a  witness path between them (this is shorthand, each entry is the equivalence class of adjacent entries which are related by $\sim$).

Then if $d_{M\otimes X}(m_k\otimes x_k,m_{k+1}\otimes x_{k+1}) \neq 0$, it is because $m_{k+1} \otimes x_{k+1} = m_k\otimes x_{k+1}' $ for some $x_{k+1}'\in X$.

So
\[
\begin{array}{cll}
& d_{M\otimes X}(m_k\otimes x_k,m_{k+1}\otimes x_{k+1}) \\
= & d_{M\times X}((m_k,x_k),(m_{k},x_{k+1}'))\\
= & \frac{1}{3}d_X(x_k,x_{k+1}')\\
\geq &  \frac{1}{3} d_Y(f(x_k),f(x_{k+1}')) &
\mbox{since $f$ is a short map}\\
= & d_{M\times Y}((m_k,f(x_k)),(m_{k},f(x_{k+1}')))\\
=  & d_{M\otimes Y}((M\otimes f)(m_k\otimes x_k),(M\otimes f)(m_{k+1}\otimes x_{k+1}))
\end{array}
\]

Thus, since there is a path in $M\otimes Y$ from $(M\otimes f)(m\otimes x)$ to $(M\otimes f)(m\otimes y)$ whose score is bounded above by the score of a shortest path in $M\otimes X$, $M\otimes f$ is a short map.

Finally, note that $M\otimes -$ preserves compositions and identity maps, as required. 

\begin{theorem} $M\otimes -$ is a functor on $\SquaMS$. \end{theorem}

Finally, we want to take advantage of the following lower bound on paths.  Recall that
\begin{equation}
 \label{d-Uzero}   
d_{U_0}((x,y),(x_1,y_1)) = |x-x_1|+|y-y_1|
\end{equation}
is the taxicab metric on the unit square. 

\begin{proposition}\label{distanceonedgesforM}  Let $B$ be an object in $\SquaMS$ and consider $m\otimes S_B((r,s))$ and $n\otimes S_B((t,u))$ in $M\otimes B$.  Then 
\[ 
d_{M\otimes B}(m\otimes S_B((r,s)),n\otimes S_B((t,u)))\geq d_{M\otimes U_0}(m\otimes S_{U_0}((r,s)),n\otimes S_{U_0}((t,u))).\]
\end{proposition}

\proof
By Theorem~\ref{quotientmetric} there is a witness path in $M\otimes B$ of the form 
\[ 
\begin{array}{l}m\otimes S_B((r,s)) = m_0\otimes S_B((r_0,s_0)), m_0\otimes S_B((r_0',s_0'))\sim m_1\otimes S_B((r_1,s_1)),\\
\ldots, m_{p-1}\otimes S_B((r_{p-1},s_{p-1})),\\
m_{p-1}\otimes S_B((r_{p-1}',s_{p-1}'))\sim m_p\otimes S_B((r_p,s_p)) = n\otimes S_B((t,u))
\end{array}
\]
\def\arraystretch{1.1}
Then note that 
\[ 
\begin{array}{cl}
 & d_{M\otimes B}(m_k\otimes S_B((r_k,s_k)), m_k\otimes S_B((r_k',s_k')))
\\
= & \frac{1}{3} d_B(S_B((r_k,s_k)),S_B((r_k',s_k')))
\\
\geq & \frac{1}{3} d_{U_0}(S_{U_0}((r_k,s_k)),S_{U_0}((r_k',s_k')))
\\
= &  d_{M\otimes U_0}(m_k\otimes S_{U_0}((r_k,s_k)), m_k\otimes S_{U_0}((r_k',s_k')))
\end{array}
\]
\def\arraystretch{1}
where the inequality follows from $(\sqtwo)$.  So the score of the shortest path in $M\otimes B$ is bounded below by the score of the corresponding path in $M\otimes U_0$, which is an upper bound of the distance between the corresponding points in $M\otimes U_0$. (But there may be a shorter path
in $M\otimes U_0$, and this is why our result has an inequality.)
\endproof

\subsection{Defining $N\otimes-$ in $\SquaMS$}
\label{section-N-otimes}

It will be useful for us to augment $M$ in the following way.  Let $N = \{0,1,2\}^2 = M\cup \{(1,1)\}$, which will correspond to the full $3\times 3$ grid.  We aim to expand the work from the previous section to show that $N\otimes -$ is also a functor.  We will use this in later sections. 

 \[
   \begin{tikzpicture}[>=stealth',shorten >=1pt,auto,node distance=2cm,semithick,scale=1.2]
  \draw [help lines] (0,0) grid (3,3);
 \draw (.5,.5)  node {$(0,0)$};
  \draw (1.5,.5)  node {$(1,0)$};
   \draw (2.5,.5)  node {$(2,0)$};
    \draw (.5,1.5)  node {$(0,1)$};
  \draw (1.5,1.5)  node {$(1,1)$};
   \draw (2.5,1.5)  node {$(2,1)$};
    \draw (.5,2.5)  node {$(0,2)$};
  \draw (1.5,2.5)  node {$(1,2)$};
   \draw (2.5,2.5)  node {$(2,2)$};
  \end{tikzpicture}
\] 
In our pictures of $N\otimes X$, we do not show an {\sf X} over the 
square $(1,1)$ the way we did with $M\otimes X$.

First we want to apply Corollary~\ref{quotientmetriccorollary} to see that $N\otimes X$ is in fact a metric space with the quotient metric.  The majority of the work for us is done.  The definition of $S_{N\otimes X}$ will coincide with $S_{M\otimes X}$, and we will need to expand $E$ to include 
\[
\begin{array}{l}
((0,1),(1,r))\approx ((1,1),(0,r))\\
((1,2),(r,0))\approx ((1,1),(r,1))\\
((2,1),(0,r))\approx ((1,1),(1,r))\\
((1,0),(r,1))\approx ((1,1),(r,0))\\
\end{array}
\]
for $r\in [0,1]$.  Call this relation $\widehat{E}$.  

\begin{lemma} $\widehat{E}$ is quotient suitable on $N\times M_0$ (Definition~\ref{quotientsuitabledefinition}).\end{lemma}

The proof is the same as that of Lemma~\ref{EisQuotientSuitable}

\begin{lemma}\label{distanceinonecopy}
Let $X$ be any object in $\SquaMS$.
Let $n\in N$ and let $x,y\in X$.
Let $(n_0,x_0),\ldots,(n_p,x_p)$ be a path, where $(n_0,x_0) = (n,x)$ and $(n_p,x_p) = (n,y)$. (Notice that the same $n$ is used in both the start and end of the path.) 
Then
\[\onethird d_X(x,y) \leq \displaystyle{\sum_{k=0}^{p-1}} d_{N\times X}((n_k,x_k),(n_{k+1},x_{k+1})).\]
\end{lemma}

\begin{proof}
Fix $n$ throughout this proof.  By Lemma ~\ref{distancelemma}, since the path $(n,x),(n,y)$ has score $\leq \frac{2}{3} <2$, we may assume that our path is an alternating path, 

$$(n,x)=(n_0,x_0),(n_0,x_0')\sim \ldots (n_{p-1},x_{p-1}')\sim (n_p,x_p),(n_p,x_p') = (n,y).$$

By Lemma \ref{technicalpathlemma}, Lemma \ref{enteringonsamesidelemma}, and Lemma \ref{theverylasttechnicalpathlemma}, we may also assume the following for $i, j$ such that $0\leq i<j\leq p$

\begin{itemize}
    \item If $x_i$ and $x_i'$ are on the same side of $S_X[M_0]$, then they are both corners,
    \item If $x_i'$ and $x_j$ are on the same side of $S_X[M_0]$, then at least one of them is a corner,
     \item  If $x_i$ and $x_j$ are on the same side of $S_X[M_0]$, then 
    $x_i'$ and $x_j'$ are \emph{not} on the same side of $S_X[M_0]$,
    \item If  $x_i$ and $x_j'$ are on the same side of $S_X[M_0]$, then 
    $x_i'$ and $x_j$ are \emph{not} on the same side of $S_X[M_0]$.
\end{itemize}

As in Lemma \ref{basecaseforM}, we will proceed by induction on the natural number $k$ where $k=|\{i:m_i=m\}|$.

When $k=1$, $p=0$, and the path from $(n,x)$ to $(n,y)$ is just $(n,x),(n,y)$, whose length is $\frac{1}{3}d_X(x,y)$.  When $k \geq 3$, our path has $n_l=n$ for some $0<l<p$.  We cut this path into two subpaths, the part between $(n_0,x_0)$ and $(n_l,x_l)$, and the part from $(n_l,x_l')$ to $(n_p,x_p')$.  In both subpaths, the number $j$ such that $n_j = n$ is $<k$, so the induction hypothesis applies to the subpaths.  By this and the triangle inequality, we show the desired inequality.
\[
\frac{1}{3} d_X(x,y) \leq \frac{1}{3}d_X(x,x_{\ell}) + \frac{1}{3}d_X(x_{\ell},y) \leq \displaystyle{\sum_{i=0}^{p-1}} d_{N\times X}((n_i,x_i),(n_{i+1},x_{i+1})).
\]








The remaining case is when $k=2$.  Up until this point, the proof as been the same as Lemma~\ref{basecaseforM}, and the remaining part is similar, but we have a few more cases to consider since we include the center copy of $X$. 
There are three possible cases we must consider, up to rotation and reflection:
when our two points are in a corner copy of $X$
(a copy indexed by $(0,0)$, $(0,2)$, $(2,0)$, or $(2,2)$);
when they are in a copy of $X$ around the outside which is not a corner
(a copy indexed by $(0,1)$, $(1,0)$, $(2,0)$, or $(0,2)$);
and when they are in the middle copy of $X$ (indexed by $(1,1)$).

First we will consider the case when both points are, without loss of generality, in the $(0,0)$ copy of $X$.
Then we need not consider paths which contain points in the $(0,2)$, $(1,2)$, $(2,2)$, $(2,1)$, or $(2,0)$ copies of $X$, since any such path (like the one pictured at the end of  the proof of Lemma~\ref{basecaseforM}) would contribute $\frac{2}{3}$ to the score by $(\sqtwo)$, by crossing two copies of $X$ (one going out and one coming back).  
So we need only consider the case 
which is shown in Figure~\ref{figure-working}(a).
(This case did not come up in the proof of Lemma~\ref{basecaseforM} because we did not have a middle copy of $X$.)
Using the same argument as in Example~\ref{example-first-new}, we can find a path with lesser or equal score by moving the point $A$ to the shared bottom corner of the $(0,1)$ and $(1,1)$ copies of $X$
as in Figure~\ref{figure-working}(b).
 But then the path $(n,x),A,(n,y)$ has score which is less than or equal to the score of this path, and again, by the triangle inequality (since we are now entirely in the $(0,0)$ copy of $X$), $(n,x),(n,y)$ is a path with smaller or equal score, which is $\frac{1}{3}d_X(x,y)$.  

\begin{figure}[t]
\[ \secfoura \qquad \secfourb  \]
\caption{The first two figures in this proof, called (a) and (b) in this proof.
\label{figure-working}}
\end{figure}

Next, without loss of generality, we consider the case
 shown in Figure~\ref{figure-working-two}(c),
when $n=(1,0)$.  Again, we need not consider paths which contain points in the $(0,2)$, $(1,2)$ or $(2,2)$ copies of $X$, since these will contribute $\frac{2}{3}$ to the score.  So we have two subcases.  
The first subcase is the same as the case we examined when $n=(0,0)$ (but shifted to the right).  The second subcase, shown in Figure~\ref{figure-working-two}(d), 
can be shorted in a similar fashion by moving $A$ and $B$ to the respective corner points (see Figure~\ref{figure-working-two}(e)), as we saw in the proof of Lemma~\ref{nondegen}.  
Again, this gives us a path with another point in $(1,0)\otimes X$, so we see that the score is bounded below by the score of $(n,x),(n,y)$, which is $\frac{1}{3}d_X(x,y)$.

\begin{figure}[h]
\[ \secfourc \qquad \secfourd \qquad \secfoure \]
\caption{The next three figures in this proof, called (c), (d), and (e).
\label{figure-working-two}}
\end{figure}

\begin{figure}[h]
\[ \secfourf \qquad \secfourg \qquad \secfourh \]
\caption{The last three figures in this proof, called (f), (g) and (h).
\label{figure-working-three}}
\end{figure}

Finally, consider the case when $n=(1,1)$.  We can split this into two subcases.  First, when the path exits and enters the $(1,1)$ copy of $X$ from adjacent sides, say the top and the right.  
The path contains points in eight of the nine copies of $X$, as 
 shown in Figure~\ref{figure-working-three}(f). In this subcase,
by $(\sqtwo)$, the score is greater than $\frac{2}{3}$.  So we do not have a witness path in this case.
The second subcase is 
when the path contains points in only four copies of $X$, as shown in  Figure~\ref{figure-working-three}(g).
This is the same as the case we examined when $n=(0,0)$.  So in both of these subcases, the score is bounded below by $\frac{1}{3}d_X(x,y)$.

The final case is shown in Figure~\ref{figure-working-three}(h), when
 the path exits and enters the $(1,1)$ copy of $X$ on opposite sides.  

But this is essentially the same as the case we examined when $n=(1,0)$.

By examining cases, we have shown that
 every path between points in the same copy of $X$ has a score which is bounded below by $\frac{1}{3}d_X(x,y)$.  
\end{proof}

As with $M$, we have the following useful corollary for $N\otimes-$.

\begin{corollary}\label{distanceinonecopyN} For an object $X$ in $\SquaMS$, $n\in N$, and $x,y\in X$, 
\[ d_{N\otimes X}(n\otimes x,n\otimes y) = \frac{1}{3} d_X(x,y).\]
In particular, $d_{N\otimes X}(n\otimes x,n\otimes y) \leq \frac{2}{3}$, and if we apply $N\otimes-$ to $X$ $k$ many times, we get $d_{N^k\otimes X}(n_1\otimes\ldots \otimes n_k\otimes x,n_1\otimes\ldots\otimes n_k\otimes y) \leq \frac{2}{3^k} $. \end{corollary}

So by Corollary~\ref{quotientmetriccorollary}, $N\otimes X$ is a metric space.  The proof that $N\otimes X$ satisfies ($\sqone$) is the same as that for $M\otimes X$, and ($\sqtwo$) follows for $N\otimes X$ from ($\sqtwo$) in $X$, in the same way as it does for $M\otimes X$ in Lemma~\ref{nondegen}.  

Morphisms will be preserved by $N\otimes -$ just as they are by $M\otimes -$.  Hence, we have the following: 

\begin{proposition}\label{augmentedM} $N\otimes -$ is a functor on $\SquaMS$. \end{proposition}

We also have an analogous lower bound on distances between boundary points to Proposition~\ref{distanceonedgesforM}. The proof is the same. 

\begin{lemma}\label{lemma-crux}  Let $B$ be an object in $\SquaMS$ and consider $m\otimes S_B((r,s))$ and $n\otimes S_B((t,u))$ in $N\otimes B$.  Then 
\[ d_{N\otimes B}(m\otimes S_B((r,s)),n\otimes S_B((t,u)))\geq d_{N\otimes U_0}(m\otimes S_{U_0}((r,s)),n\otimes S_{U_0}((t,u))).\]
\end{lemma}

\subsection{Distances between corner points in iterates of $N\otimes -$ on $M_0$ and $U_0$}

In much of this paper, we are going to be interested in 
iterating the functor $M\otimes -$ on the unit square
$U_0$, or on the initial square space $M_0$,
or more generally on square spaces $B$ which admit a morphism $B\to M\otimes B$.
But at this point, we need some results about the iteration of $N\otimes -$ on such spaces.
In fact, results on $M\otimes -$ often go through results on $N\otimes -$, partly because
this latter functor is easier to study.

\begin{definition}\label{def-cp}
Let $B$ be either $M_0$ or $U_0$.
\label{def-CPk}
The set $CP_k$ of  \emph{corner points} of $N^k\otimes B$ is defined as follows:
\[
\begin{array}{lcl}
CP_0 & = & \set{(0,0), (0,1), (1,0), (1,1)} \\
CP_{k+1} & = & \{n\otimes x \mid n\in N,x\in CP_k\}\\
\end{array}
\]

Let $f_k\colon CP_k\to U_0$ be (as expected): $f_0$ is the inclusion, and $f_{k+1}(n\otimes x) = \frac{1}{3}n + \frac{1}{3} f_k(x)$.

Later on in Definition~\ref{definition-alpha-N}, we will define $\alpha_{N}:N\otimes U_0\rightarrow U_0$ similarly, so that $f_{k+1}(n\otimes x) = \alpha_N(n\otimes f_k(x))$.

We regard $CP_k$ as a metric space with distances inherited from  $N^k\otimes B$.
\end{definition}

The main result in this section shows 
that it does not matter whether we take $B$ to be $M_0$ or $U_0$
in Definition~\ref{def-CPk}:
 the distances between corner points are the same.

\begin{definition}\label{stardef}
For any sequence of $p$ numbers $i_1, \ldots, i_p\in \{0,1,2\}$ and for any $r\in [0,1]$, we define
the number $\norm{i_1, \ldots, i_p;r}\in [0,1]$ in the following way.  
\[
\begin{array}{lcl}
\norm{r} & = & r \\
\norm{i_1,i_2, \ldots, i_p;r}  & = & \frac{1}{3}i_1 + \frac{1}{3}\norm{i_2, \ldots, i_p;r}
\end{array}
\]
In a more explicit presentation, 
\[\norm{i_1,i_2, \ldots, i_p;r} = \frac{r}{3^p} + \sum_{m=1}^p \frac{i_m}{3^m} 
\]
\end{definition}

\begin{lemma}
For all $i_1, \ldots, i_p\in \{0,1,2\}$ and all $r\in [0,1]$,
\[
\begin{array}{lcl}
S_{N^p\otimes B}((\norm{i_1,i_2, \ldots, i_p;r},0))
& = & (i_1,0)\otimes \ldots \otimes (i_p,0) \otimes S_{B}((r,0))\\
S_{N^p\otimes B}((\norm{i_1,i_2, \ldots, i_p;r},1)) & = & (i_1,2)\otimes \ldots \otimes (i_p,2) \otimes S_{B}((r,1))\\
\end{array}
\]
\end{lemma}

\begin{proof}
 We prove this by induction on $p$.  For $p=0$,
 the result is clear.
Assume our result for $p$, and fix $r$ and $i_1,\ldots, i_p, i_{p+1}$.
Let $r^* = \norm{i_2, \ldots, i_p;r}$.
By induction hypothesis,
\[ S_{N^p\otimes B}((r^*,0)) = (i_2,0)\otimes \ldots \otimes (i_{p+1},0) \otimes S_{B}((r,0)),\]
and similarly for $S_{N^p\otimes B}((r^*,1))$.
To save on a little notation, write $r^{**}$ for $ \norm{i_1,i_2, \ldots, i_p;r}$.
When $i_1 = 0$,  $r^{**} = \frac{1}{3}r^*$.
Using (\ref{siszero}), (\ref{sisone}), and the induction hypothesis,
\def\arraystretch{1.5}
\[
\begin{array}{lcl}
S_{N\otimes(N^p\otimes B)}((r^{**},0)) & = & (0,0)\otimes 
S_{N^p\otimes B}((r^*,0))\\ & = & 
(i_1,0)\otimes \ldots \otimes (i_{p+1},0) \otimes S_{B}((r,0))
\\
S_{N\otimes(N^p\otimes B)}((r^{**},1)) & = & (0,2)\otimes 
S_{N^p\otimes B}((r^*,1))\\ & = & 
(i_1,2)\otimes \ldots \otimes (i_{p+1},2) \otimes S_{B}((r,1))
\\
\end{array}
\]
\def\arraystretch{1}
If $i_1 = 1$, we have $r^{**} = \frac{1}{3}r^* + \frac{1}{3}$,
and if $i_1 = 2$,  $r^{**} = \frac{1}{3}r^*+ \frac{2}{3}$.   
In all of these cases, the verifications are similar.
\end{proof}

Here is another fact about this notation.

\begin{lemma}\label{hvinequality}
For all $p\geq 1$ and all $i_1, \ldots, i_p;r$,
$\frac{i_1}{3} \leq \norm{i_1, \ldots, i_p;r} \leq \frac{i_1+1}{3} $.
The only way to have $\norm{i_1, \ldots, i_p;r} = 0$ is for $i_1, \ldots, i_p;r = 0,\ldots,0;0$.
The only way to have $\norm{i_1, \ldots, i_p;r} = 1$ is for $i_1, \ldots, i_p;r = 2,\ldots,2;1$.
\end{lemma}

\begin{proof}
 By induction on $p$.  When $p=1$, this is clear from the definition of $|i_1;r|$.  
 Assume our result for $p$, and take $i_1, \ldots, i_{p+1};r$.
By induction hypothesis,
 $ 0 \leq \norm{i_2, \ldots, i_{p+1};r} \leq 1$.

So since $|i_1,\ldots, i_{p+1};r| = \frac{1}{3}i_1+\frac{1}{3}|i_2,\ldots, i_{p+1};r|$,

\[\begin{array}{rcccl}
   0 &\leq& \norm{i_2, \ldots, i_{p+1};r} &\leq& 1\\
     0 &\leq& \frac{1}{3}\norm{i_2, \ldots, i_{p+1};r} &\leq& \frac{1}{3}\\
\frac{i_1}{3} &\leq& \frac{i_1}{3}+\frac{1}{3}\norm{i_2, \ldots, i_{p+1};r} &\leq& \frac{i_1}{3}+\frac{1}{3}\\
\frac{i_1}{3} &\leq & |i_1,\ldots,i_{p+1};r|&\leq & \frac{i_1+1}{3}\\

\end{array}\]


The last assertions in our result are easy to check by induction on $p$.
\end{proof}

The last few definitions allowed $r$ and $s$ to be any numbers in $[0,1]$.  For the next main results
we need to restrict to corner points
(see Definition~\ref{def-cp}).
 The key result in this section,
Lemma~\ref{lemma-taxicab-like} below, is false without the restriction to corner points.

Next we turn our attention to the possible witness paths between elements of $N\otimes X$.  In Lemma~\ref{upandtotherightinN}, we will show that (up to rotation and reflection) we can always find a path which goes ``up and to the right'' between copies of $X$. 
In it, we use an ordering $\triangleleft$ on $\set{0,1,2}\times\set{0,1,2}$.
That ordering is the strict part of the product ordering determined  by the natural order $0<1<2$ on $\set{0,1,2}$. 
In other words, $(i,j) \triangleleft (k,\ell)$ iff $i \leq k$ and $j\leq \ell$ and at least one of these inequalities is strict.

Here is a preliminary observation on this notation.  Consider 
$\set{0,1,2}\times\set{0,1,2}$ as a graph $G$, where there is an edge from 
$(i,j)$ to $(k,\ell)$ iff either ($i = k$ and $|j - \ell| = 1$) or else
 ($j = \ell$ and $|i - k| = 1$). Suppose that $(a,b) \triangleleft (c,d)$.
 Then there is a geodesic (a path of minimal length) in $G$ from $(a,b)$ to $(c,d)$
 consisting of points which ``goes up'' in the order $\triangleleft$.
 (For example, the  distance in $G$ from $(0,0)$ to $(2,2)$ is $4$,
 and we have a path which ``goes up'' in $\triangleleft$:  
 $(0,0) \triangleleft (1,0) \triangleleft (1,1) \triangleleft (1,2) \triangleleft (2,2)$.)

Lemma~\ref{upandtotherightinN} just below is an analogous fact, but not for
the graph $G$ but instead for a square space of the
form $N\otimes B$.

\begin{lemma}\label{upandtotherightinN}
Let $B$ be in $\SquaMS$ and $x,y\in N\otimes B$. 
Suppose that $x=m_0\otimes x_0$ and $y=m_k\otimes y_0$ with $m_0\triangleleft m_k$ 
with respect to this partial order (we can rotate or reflect if necessary). 
Then there is a witness path of the form 
\[ x=m_0\otimes x_0, m_0\otimes S_B((r_0,s_0)),m_1\otimes S_B((r_1,s_1)) ,\ldots,\]
\[m_{k-1}\otimes S_B((r_{k-1},s_{k-1})), m_k\otimes y_0 =y\]
where each for each $0\leq i< k$, $m_i\otimes S_B((r_i,s_i))= m_{i+1} \otimes S_B((r_i',s_i'))$ for some $(r_i',s_i')\in M_0$, and $m_i
\triangleleft m_{i+1}$.
\end{lemma}

\begin{proof}
We have several cases.   The first is when $m_0$ and $m_k$
are neighbors in $G$
(for example, $(0,2)$ and $(1,2)$).  In this case, the work which we did
in the proof of Theorem~\ref{theorem-regular-witness-paths} adapts easily to 
give the result which we want.  That theorem dealt with the functor $M\otimes -$ 
and not $N\otimes -$, but for this case the work there shows that the witness
paths from $x$ to $y$ look like
\[
  \begin{tikzpicture}[>=stealth',shorten >=1pt,auto,node distance=2cm,semithick,scale=1.5]
  \draw [help lines] (2,3) grid (3,5);
 \draw (2.5,4.5)  node (a3)  {$\bullet$};
  \draw (2.3, 4.5) node {$y$};
  \draw (2.1,3.1)  node  (a4) {$\bullet$};
  \draw (2.3, 3.1) node {$x$};  
  \draw (2.8,4) node (a5) {$\bullet$};
 \draw[line width=2pt] (2.5,4.5) --(2.8,4);
\draw[line width=2pt] (2.1,3.1) -- (2.8,4);  
  \end{tikzpicture}
  \]
The one case which we need to add is when we have a path that uses the middle square and looks like the path from $x$ to $y$ in the picture on the left below:
\[
  \begin{tikzpicture}[>=stealth',shorten >=1pt,auto,node distance=2cm,semithick,scale=1.5]
   \draw [help lines] (0,0) grid (3,3);   
      \draw (1.2,1.4)  node  (m1)   {$\bullet$};         \draw (1.2,1.2) node{\protect{\scriptsize $x$}};
\draw (.4,2) node (m2) {$\bullet$};
   \draw (1,2.5) node (m3) {$\bullet$};       
   \draw (2,2.5) node (m4) {$\bullet$};  
\draw (2.2,2.75) node{\protect{\scriptsize $z$}}; 
      \draw (2.6,2) node (m5) {$\bullet$};      
            \draw (2.5,1.4) node (m7) {$\bullet$};   
            \draw (2.5,1.2) node{\protect{\scriptsize $y$}};
\draw[line width=2pt] (1.2,1.4) -- (.4,2);
\draw[line width=2pt] (.4,2) --   (1,2.5);
\draw[line width=2pt]  (1,2.5) -- (2,2.5);
\draw[line width=2pt] (2,2.5) -- (2.6,2);
\draw[line width=2pt] (2.5,1.4) -- (2.6,2);
\end{tikzpicture}
\qquad\qquad
\begin{tikzpicture}[>=stealth',shorten >=1pt,auto,node distance=2cm,semithick,scale=1.5]
  \draw [xshift=-0.5cm,help lines] (4,0) grid (7,3);    
     \draw (4.1,0.3)  node (n0)   {$\bullet$}; 
          \draw (3.9,0.3)  node (n0)   {$x$};   
       \draw (4.2,2.7)  node (nend)   {$\bullet$};    
    \draw (4.0, 2.7) node {$y$};     
        \draw (4.8, 2.7) node {$y'$};   
  \draw[line width=2pt] (4.5,0.32) -- (5.3,1); 
   \draw[line width=2pt] (4.1,0.3) -- (4.5,0.32);
\draw (5.3,1) node (n3) {$\bullet$};
   \draw (5.1,2) node (n4) {$\bullet$}; 
\draw[line width=2pt] (5.3,1) -- (5.1,2);
\draw (5.3, .7) node {$z_1$};
\draw (5.1, 2.3) node {$z_2$};
\draw[line width=2pt] (5.1,2) -- (4.5,2.7);

\draw[line width=2pt] (4.2,2.7) -- (4.5,2.7); 

    \draw (4.5, 2.7) node {$\bullet$}; 
      \draw (4.5, 0.32) node {$\bullet$};   
            \draw (4.7, 0.22) node {$x'$};   
 \draw (4.5,1) node  {$\bullet$};
  \draw (4.5,2) node  {$\bullet$};
 \draw (4.2,1) node  {$w_1$};
  \draw (4.2,2) node  {$w_2$};  
\end{tikzpicture}     
\]
But here the part from $x$ to $z$ may be shortened: since $x$ and $z$
lie in neighboring squares, this is 
the content of our previous observations.  And once we shorten this path 
so that it does not take the long trip by starting out going left from $x$,
it is then isomorphic 
to a path in $M\otimes B$,
we are in a position to use our previous work to produce from it a witness 
path with the appropriate feature: the first components go up in $\triangleleft$.

Next, let us consider the case when 
$m_0$ and $m_k$ have distance $2$ in $G$.
The classification of 
regular witness paths for $M\otimes -$ which we saw in Figure~\ref{figure-regular-witness-paths} applies except for small changes.
We need an addition for a situation as on the right above.  We have a path from $x$ to $y$ going through the middle square, as shown.
It is $x, x', z_1, z_2, y', y$.
This path is not what we want because at the end we have a pair with
$\neg((1,2)\triangleleft (0,2))$.  However, let us consider the path $x, w_1, w_2, y$.  This new path is increasing in $\triangleleft$.
We claim that its score is at most that of the original path
We have $d(w_1, w_2) =  \onethird$, by ($\sqone$),
and $d(z_1, z_2) \geq \onethird$, by $(\sqtwo)$.
Moreover, 
\[ d(x, w_1) \leq d(x,x') + d(x',w_1) \leq d(x,x') + d(x',z_1),\]
using the triangle inequality and $(\sqtwo)$.
The same calculations apply on the other end of the path, and we put things together
to see that indeed the score of the new path is at most the score of the old.

There are very similar arguments when $m_0$ and $m_k$ have distance $3$ or $4$ in $G$.
Indeed, the cases which we have considered make the arguments short in these cases.
We omit the details.
\end{proof}


The following is a lemma about distances between points in $N\otimes B$ (where $B$ is an arbitrary square metric space) in different copies of $B$.

\begin{lemma}\label{pathsinadjacentNcopies}
Let $B$ be a square metric space and let $x,y\in B$.  Consider the points $(0,0)\otimes x$, $(0,1)\otimes y$, $(0,2)\otimes y$, and $(1,1)\otimes y$ in $N\otimes B$. 

\begin{enumerate}
\item There is a witness path from $(0,0)\otimes x$ to $(0,1)\otimes y$ of the form
\[
 (0,0)\otimes x, (0,0) \otimes S_B((r,1)) = (0,1) \otimes S_B((r,0)), (0,1)\otimes y ,
\] 
where $r\in [0,1]$.
\item There is a witness path from $(0,0)\otimes x$ to $(0,2)\otimes y$ of the form
\[ \begin{array}{c}   (0,0)\otimes x  ,  (0,0) \otimes S_B((r_1,1)) = (0,1) \otimes S_B((r_1,0)),\\

          (0,1) \otimes S_B((r_2,1)) = (0,2)\otimes S_B((r_2,0)), (0,2)\otimes y,
          \end{array}\] 
          where $r_1,r_2\in [0,1]$.
\item  There is a witness path from $(0,0)\otimes x$ to $(1,1)\otimes y$ of the form 
\[ (0,0)\otimes x ,(0,0)\otimes S_B((1,1)) = (1,1)\otimes S_B((0,0)) , (1,1)\otimes y.\]
\item More generally, by rotating or reflecting, we get the analogous results for copies of $B$ which are in the same row or column (1. and 2.), or which share a corner (3.).
\end{enumerate}
\end{lemma}

\begin{proof}

Parts 1. and 2. follow from Lemma~\ref{upandtotherightinN}.  In part 3., we know from Lemma~\ref{upandtotherightinN}, we know that there is a path of the form 
\[\begin{array}{c}  (0,0)\otimes x, (0,0)\otimes S_B((r_1,1))  =  (0,1)\otimes S_B((r_1,0)),\\
(0,1)\otimes S_B((1,s_2))  =  (1,1)\otimes S_B((0,s_2)), (1,1)\otimes y
\end{array}
\]
or of the form 
\[\begin{array}{c}
 (0,0)\otimes x, (0,0)\otimes S_B((1,s_1))  = (1,0)\otimes S_B((0,s_1)),\\
 (1,0)\otimes S_B((r_2,1)) =   (1,1)\otimes S_B((r_2,0)),
 (1,1)\otimes y.
\end{array}
\]
Without loss of generality, suppose it is the former. 
We are going to use the triangle inequality and $(\sqtwo)$ to show that the score of such a path is minimized when 
\[ (0,1)\otimes S_B((r_1,0)) = (0,1)\otimes S_B((1,0)) = (0,1)\otimes S_B((1,s_2)).\]  That is, our path has the smallest score when it is of the form 
\[(0,0)\otimes x, (0,0)\otimes S_B((1,1)) = (1,1)\otimes S_B((0,0)), (1,1)\otimes y,\] 
as required. 

For ease of notation, let 
\[ \begin{array}{lcl} a &  = & d_{N\otimes B}((0,0)\otimes x,(0,0)\otimes d_B((1,1)))\\
b & = &d_{N\otimes B}((1,1)\otimes S_B((0,0)),(1,1)\otimes y).
\end{array}
\]
Our goal is to show that the score of the proposed path is $\geq a+b$. Let
\[ \begin{array}{lcl} c & =  &  d_{N\otimes B}((0,0)\otimes x, (0,0)\otimes S_B((r_1,1))),\\
e & =& d_{N\otimes B}((0,1)\otimes S_B((r_1,0)),(0,1)\otimes S_B((1,s_2))),\\
f & =& d_{N\otimes B}((1,1)\otimes S_B((0,s_2)),(1,1),y).\\
\end{array}
\]
So we want to show $c+e+f\geq a+b$.  
Observe that $e\geq (1-r_1)+s_2$ by $(\sqtwo)$. In addition, $c+(1-r_1)\geq a$ and $s_2+f\geq b$ by the triangle inequality.  
Thus, \[ c+e+f \geq a-(1-r_1)+(1-r_1)+s_2 +b-s_2 = a+b,\] as required.  
\end{proof}

\begin{lemma}
\label{lemma-taxicab-like}
Let $B$ be a square metric space, and 
let $p\geq 0$.

\begin{enumerate}
\item 
\label{not-necessarily-corner}
For all $x,y\in N^p\otimes B$ of the form
\begin{equation}
    \label{lemma-taxicab-like-xy}
\begin{array}{lcl}
x & = &  (i_1,j_1)\otimes (i_2,j_2)\otimes \ldots \otimes (i_p,j_p)\otimes S_B((r,s))
\\
y & = &  (k_1,\ell_1)\otimes (k_2,\ell_2)\otimes \ldots \otimes (k_p,\ell_p)\otimes S_B((t,u))
\end{array}
\end{equation}
we have the following distance formula:\footnote{The notation
$|i_1, i_2, \ldots, i_p; r|$ was
introduced in Definition~\ref{stardef}.}
\begin{equation}\label{eq-taxicab-like}
\begin{array}{lcl}
d_{N^p\otimes B}(x,y) & \geq &   | i^* - k^*| + |j^*- \ell^*|,
\end{array}
\end{equation}
where
\[
\begin{array}{lcl}
i^* & = & |i_1, i_2, \ldots, i_p; r|,\\
j^* & =  & |j_1, j_2, \ldots, j_p; s|,\\
\end{array}
\qquad
\begin{array}{lcl}
k^* & = & |k_1, k_2, \ldots, k_p; t|,\\
\ell^* & = & |\ell_1, \ell_2, \ldots, \ell_p; u|.\\
\end{array}
\]
\item \label{part-in-case-corner}
Assume that $B$ is either $M_0$ or $U_0$ and that 
$x$ and $y$ are corner points.  Then we may improve (\ref{eq-taxicab-like})
to a ``taxicab-like'' formula:
\[
\begin{array}{lcl}
d_{N^p\otimes B}(x,y) & = &    | i^* - k^*| + |j^*- \ell^*|.
\end{array}
\]
\end{enumerate}
\end{lemma}

 \begin{proof}
By induction on $p$.  

The base case is $p = 0$.  Part (\ref{not-necessarily-corner}) is just the statement 
of $(\sqtwo)$.
We turn to part (\ref{part-in-case-corner}).  This is where we use the assumption that we are dealing with
corner points and the overall space $B$ is either $M_0$
or $U_0$.  That is, the distance among points $(0,0)$, $(0,1)$, $(1,0)$, and $(1,1)$ 
may be calculated as if we were using the taxicab metric, even though the space $M_0$ uses
the path metric; the formula in this lemma is in general false and holds mainly for the corner points.

Let us check both (\ref{not-necessarily-corner}) 
and (\ref{part-in-case-corner})
for $p+1$, assuming them for $p$.
The argument breaks into cases depending on which copy of $N^p\otimes B$ our points $x$ and $y$
belong to.  

The first case is when $x$ and $y$ are in the same copy of $N^p\otimes B$.
That is, $(k_1, \ell_1) = (i_1,j_1)$.
We are going to check (\ref{part-in-case-corner}); the argument for (\ref{not-necessarily-corner})
is similar.  So $x$ and $y$ are corner points, and $B$ is $M_0$ or $U_0$.
Let 
\[
\begin{array}{lcl}
x' & = &  (i_2,j_2)\otimes \ldots \otimes (i_p,j_p)\otimes S_B((r,s))
\\
y' & = &  (k_2,\ell_2)\otimes \ldots \otimes (k_p,\ell_p)\otimes S_B((t,u))
\end{array}
\]
So $x = (i_1,j_1)\otimes x'$ and $y = (i_1,j_1)\otimes y'$.
In this case, $x'$ and $y'$ are corner points as well.
By induction hypothesis, $d_{N^{p}\otimes B}(x',y') = |i^*_2-k^*_2 | + |j^*_2 - \ell^*_2|$, where
$i^*_2  =  |i_2, i_3, \ldots, i_p; r|$, and similarly for $j^*_2$,  $k^*_2$, and $l^*_2$
(note that these start with second entry of the non-subscripted version, hence the $2$).
Now $i^* = \onethird i_1 + \onethird i^*_2$,
and similarly for the others.
 By Corollary~\ref{distanceinonecopyN},
\[
d_{N^{p+1}\otimes B}(x,y)
= \frac{1}{3} d_{N^{p}\otimes B}(x',y')
= \biggl|\frac{1}{3} i^*_2 - \frac{1}{3}k^*_2\biggr| + \biggl| \frac{1}{3}j^*_2 - \frac{1}{3} \ell^*_2 \biggr| 
= |i^*- k^*| + |j^* - \ell^*|.
\]
using $(k_1,\ell_1) = (i_1,j_1)$ in the last step. 

This concludes our work for (\ref{part-in-case-corner})
in this first case of the induction step, and as we said, 
({not-necessarily-corner}) is similar.

The other cases are when $x$ and $y$ are in different copies of $N^p\otimes B$.  
We are going to give full details for 
the case when $x$ and $y$ are in copies which share an edge.
Concretely, we shall work with the assumption 
$(i_1,j_1) = (0,0)$ and $(k_1,l_1) = (0,1)$.
Let $x$ and $y$ be as in (\ref{lemma-taxicab-like-xy}), but with $p+1$ terms
$(i,j)$ or $(k,\ell)$
instead of $p$.
Let $x'$ 
and $y'$ be as shown below, where we reiterated $x$ and $y$ for convenience:
\[
\begin{array}{lcl}
x& = & (0,0) \otimes (i_2,j_2) \otimes \ldots \otimes (i_{p+1},j_{p+1}) \otimes S_B((r,s))\\
x'& = & (0,0) \otimes (i_2,2) \otimes \ldots \otimes (i_{p+1},2) \otimes S_B((r,1))\\
 &  = & (0,1) \otimes (i_2,0) \otimes \ldots \otimes (i_{p+1},0) \otimes S_B((r,0))\\
y & = & (0,1) \otimes (k_2,\ell_2) \otimes \ldots \otimes  (k_{p+1},\ell_{p+1}) \otimes S_B((t,u))\\
y' & = & (0,1) \otimes (k_2,0) \otimes \ldots \otimes  (k_{p+1},0) \otimes S_B((t,0))\\
&=&(0,0)\otimes (k_2,2)\otimes\ldots\otimes (k_{p+1},2)\otimes S_B((t,1)).\\
\end{array}
\]

We check (\ref{not-necessarily-corner}) first.
For this, take any witness path from
$x$ to $y$. 

It follows from Lemma~\ref{pathsinadjacentNcopies} that 
we may find such a path
consisting of $x$ and $y$ connected by an element $z = (0,0)\otimes z' = (0,1)\otimes z''$,
where  $z' = S_{N^p\otimes B}((r',1))$ and $z'' = S_{N^p\otimes B}((r',0))$ for some $r'\in [0,1]$.

Before showing the full details, here is the idea.
Consider the points $x$, $x'$, and $z$.
These all lie in one and the same copy of $N^p\otimes B$, and so we may drop the outermost $(0,0)$
from their
expressions and apply part (1)  of the induction hypothesis
and also  Corollary~\ref{distanceinonecopyN}. 
We can also take $y$, $y'$, and $z$ and drop the outermost $(0,1)$ from their expressions
and use the induction hypothesis.  Further, $x'$ and $y'$ each have two expressions,
and we can use the induction hypothesis.   
So in this way, we may get  lower bounds on $d(x,z)$, $d(y,z)$, and $d(x',y')$.
Adding these gives a lower bound on the score of the path from $x$ to $y$ using $z$.
We will see that it is $\geq |i^*-k^*| + |j^* - \ell^*|$.
Part (2) in this lemma concerns the case when all the points involved are corner points.
In this case, we can make a judicious choice of $z$ (namely either $x'$ or $y'$) and
match this lower bound. This is how we verify the exact formula for $d(x,y)$ in this case.

\begin{claim}\label{claimforxy} We have $i^* = \onethird i^*_2$, $j^* = \onethird j^*_2$, $k^* = \onethird k^*_2$,
and $\ell^* = \onethird + \onethird \ell^*_2$.
Moreover, the following hold:
\[
\begin{array}{lclcl}
d_{N^{p+1}\otimes B}(x,x') & \geq & \frac{1}{3} - j^*
\\
d_{N^{p+1}\otimes B}(y,y') & \geq &  \frac{1}{3}\ell^*_2
\\
d_{N^{p+1}\otimes B}(x',y') & \geq & 
\onethird |i^*_2 - k^*_2| = |i^* - k^*|\\
\end{array}
\]
\end{claim}

\begin{proof}
The first assertions are easy from the definitions of the $*$ notation; in the last one, we use the
fact that $\ell_1 = 1$.
 All remaining assertions
 are proved similarly, and so we only go into details about the first 
 assertion. Let $w,w'\in N^p\otimes B$ be as below, so that $x = (0,0)\otimes w$, and 
 $x' = (0,0)\otimes w'$.
 \[
\begin{array}{lcl}
w& = &  (i_2,j_2) \otimes \ldots \otimes (i_{p+1},j_{p+1}) \otimes S_B((r,s))\\
w'& = &  (i_2,2) \otimes \ldots \otimes (i_{p+1},2) \otimes S_B((r,1))\\
\end{array}
\]
By our induction hypothesis on $p$,
\[\begin{array}{rcl}
d_{N^p\otimes B}(w,w')   &\geq  &
|i_2^*-i_2^*| + |j_2^* - (\frac{1}{3^p} + {\sum_{i=1}^p}\frac{2}{3^i})|\\
&=&0+  |j^*_2-1 |\\
&= &1 - j^*_2.\\
\end{array}
\]

The first inequality follows from the induction hypothesis and Definition~\ref{stardef}.  The second line is because 
\[ \frac{1}{3^p} + \displaystyle{\sum_{i=1}^p}\frac{2}{3^i}= \frac{1}{3^p} + \frac{2}{3}\biggl(\frac{1-\frac{1}{3^{p}}}{1-\frac{1}{3}}\biggr) = \frac{1}{3^p} + \frac{3^p-1}{3^p} = 1.\]  Finally, $j^*_2\leq 1$ by the same calculation, since $j_i\leq 2$ and $s\leq 1$. 

Because $x$ and $x'$ lie in the same copy of $N^p\otimes B$, 
we may use Corollary~\ref{distanceinonecopyN} to get the first inequality:
\[
d_{N^{p+1}\otimes B}(x,x') = \onethird d_{N^p\otimes B}(w,w') \geq \onethird (1 - j^*_2)
= \onethird - j^* .
\]
The proofs of the other two parts of this claim are similar applications of the induction hypothesis.
\end{proof}

\def\arraystretch{1.1}
 Using the claim,
 \begin{equation}
\label{techinicalinequality}
\begin{array}{lcl}
d_{N^{p+1}\otimes B}(x,x') + d_{N^{p+1}\otimes B}(y,y')&  \geq &  \frac{1}{3} - j^* +  \onethird \ell^*_2 \\
 & = & ( \frac{1}{3} +  \onethird \ell^*_2 ) - j^*\\
 & = & \ell^* - j^* \\
 & = & | j^* - \ell^*|
\end{array}
\end{equation}
At the end, we used
 Lemma~\ref{hvinequality}:
$\ell^* \geq \frac{1}{3} \geq \frac{1}{3}j^*_2 =  j^*$, which again uses the fact that $j_2^*\leq 1$.  
\def\arraystretch{1}

Recall that we had a point $z = (0,0)\otimes z' =  (0,1)\otimes z''$.
We need some estimates concerning $d(x,z)$ and $d(z,y)$.
Let us introduce notation for $z'$ and $z''$:  
\[
\begin{array}{lcl}
z'& = & (u_2,2) \otimes \ldots \otimes (u_{p+1},2) \otimes S_B((v,1))\\
z'' & = & (u_2,0) \otimes \ldots \otimes (u_{p+1},0) \otimes S_B((v,0))\\
\end{array}
\]
Our induction hypothesis applies to $w, z'\in N^p\otimes B$.
Since $x = (0,0) \otimes w$ and $z = (0,0)\otimes z'$,
we have 
\[
\begin{array}{lclcl}
d(x,z) &  =  & \onethird d(w,z') 
& \geq &  \onethird | u^*_2 - i^*_2 | +   (\frac{1}{3} - \frac{1}{3}j^*_2). \\
\end{array}
\]
Similarly,
\[
\begin{array}{lcl}
d(y,z) & 
 \geq & \onethird | u^*_2 - k^*_2 | 
+ \onethird \ell^*_2 .\\
\end{array}
\]
Recall that for any real numbers, $|a -b| + |b-c| \geq |a - c|$.
We get a lower estimate for the score of our witness path:
\[
\begin{array}{lclcl}
d(x,z) + d(z,y) & \geq &\onethird | i^*_2 - k^*_2 |
 + |j^* -\ell^*|  
& = & |i^* - k^*| + |j^* -\ell^*|.
\end{array}
\]
We also used the calculations which we saw in (\ref{techinicalinequality}).
Since $d(x,y)$ is the score of some witness path, by Lemma~\ref{pathsinadjacentNcopies} we see that 
 indeed  \begin{equation}\label{ineqtoref}d(x,y) \geq |i^*-k^*| + |j^* - \ell^*|.\end{equation}

We continue with
our work under the assumption 
$(i_1,j_1) = (0,0)$ and $(k_1,l_1) = (0,1)$,
turning to part (\ref{part-in-case-corner}).
In this case, $x$ and $y$ are corner points.
It follows that $x'$ and $y'$ are also corner points.
We restate Claim~\ref{claimforxy}, adding to the assumptions that $x$, $y$, $x'$, and $y'$ 
are corner points, and strengthening the conclusions by replacing $\leq$ with $=$
throughout.
The proof goes through because 
$w$ and $w'$ are again corner points, so we
are entitled to use (\ref{part-in-case-corner}) for $p$ on them.
In particular,  $d(x',y') = |i^*-k^*|$.
We then infer an additional fact: $d(x,x') + d(y,y') = |j^*-\ell^*|$.
This is shown exactly as in (\ref{techinicalinequality}), but with the $\geq$ assertion
 replaced by equality.
Then by the triangle inequality, 
\[
\begin{array}{lclclcl}
d(x,y) & \leq & d(x,x') + d(x',y') + d(y',y) \\
 & = &  d(x,x')  + d(y',y) + d(x',y')
 & =  & |i^*-k^*| + |j^*- \ell^*|.
\end{array}
\]
By (\ref{ineqtoref}),
we have equality.  This shows part (\ref{part-in-case-corner}) in the case that 
$(i_1,j_1) = (0,0)$ and $(k_1,l_1) = (0,1)$.
Similar work applies in the other cases when $x$ and $y$ are in 
 copies  of $N^p\otimes B$ which share an edge.

The other cases in this induction step are similar.
\end{proof}

We have the following proposition; it will be more important for us going forward than
the formula in Lemma~\ref{lemma-taxicab-like}.

\begin{proposition} \label{prop-CP1}
For all $k$:
\begin{enumerate}
    \item $f_k:CP_k\rightarrow U_0$ (below Definition~\ref{def-cp}) is an isometric embedding.
\item For $m_1,\ldots,m_k,n_1,\ldots,n_k\in N$ and $x,y\in M_0\subset U_0$ which are corner points, \[\begin{array}{rl}
&d_{N^k\otimes M_0}(m_1\otimes\ldots\otimes m_k\otimes x,n_1\otimes\ldots\otimes n_k\otimes y)\\
=& d_{N^k\otimes U_0}(m_1\otimes \ldots\otimes m_k\otimes x,n_1\otimes\ldots\otimes n_k\otimes y). \end{array}\]

That is, corresponding corner points have the same distance whether we are viewing them in $N^k\otimes M_0$ or $N^k\otimes U_0$. 
\end{enumerate}
\end{proposition}

\subsection{The natural transformation $\iota$}

Recall that as a set, $M$ is a subset of $N$.
We are next interested in the relation between the two
functors $M\otimes -$ and $N\otimes -$.

\begin{proposition}
\label{prop-eta-M-N}
There is a natural transformation
$\iota\colon (M\otimes -) \to (N \otimes -)$.
\end{proposition}

\proof
For a space $X$, $\iota_X$ is the inclusion
of spaces $M\otimes X \to N \otimes X$.
This is a short map because every witness path in $M\otimes X$ between
points is a path between the same points in $N\otimes X$.
Naturality is the assertion that the diagram below 
commutes:
\[
\xymatrix@C+1pc{M\otimes X \ar[r]^-{M\otimes f} \ar[d]_-{\iota_X} 
    & M\otimes Y \ar[d]^-{\iota_Y}\\ 
   N\otimes X   \ar[r]_-{N\otimes f}  &  N\otimes Y
   } 
 \]
For each $m\otimes x\in M\otimes X$,
 the upper passage gives $m\otimes f(x)$,
 and this is exactly what the lower passage gives.
\endproof

\subsection{The Cauchy completion functor}\label{Cauchycompletionfuctor}

To obtain the final $M\otimes-$ and $N\otimes-$ coalgebras, we will use the technique in~\cite{Bhat} of using the completion of the initial algebra.  Here we recall some facts about $C$, the Cauchy completion functor.

Consider a category $\mathcal{C}$ of metric spaces whose morphisms are short maps,  and for $X$ an object in $\mathcal{C}$, let $CX$ be its Cauchy completion, where we identify equivalent Cauchy sequences (that is, $(x_i)_i$ and $(y_i)_i$ such that $d_X(x_i,y_i)$ tends to $0$).
For Cauchy sequences $(x_i)_i$ and $(y_i)_i$ from an object $X$ in $\mathcal{C}$, $d_{CX} ((x_i)_i,(y_i)_i) = \displaystyle{\lim_{i\rightarrow\infty}}d_X(x_i,y_i)$, which is 
well-defined (as it will be $0$ for equivalent Cauchy sequences).  If $(x_i)_i$ and $(y_i)_i$ are not equivalent, then $d_{CX}((x_i)_i,(y_i)_i)>0$.     For $f:X\rightarrow Y$ a morphism in $\mathcal{C}$, let $Cf:CX\rightarrow CY$ be defined by $(x_i)_i\mapsto (f(x_i))_i$.  Since $(x_i)_i$ is a Cauchy sequence in $X$ and $f$ is a short map, $(f(x_i))_i$ is a Cauchy sequence in $Y$;  this, too, is well-defined.   We assume that $\mathcal{C}$ is closed under $C$ and that $Cf$ is a morphism in $\mathcal{C}$ whenever $f$ is.
This defines $C$ as a functor on $\mathcal{C}$.
Finally, each space $X$ embeds in $CX$ by taking constant sequences, and we have a natural transformation
 $i: Id \to C$.

We specialize all of this to the case when
$\mathcal{C}$ is $\SquaMS$.

\begin{lemma}
\label{lemma-Cauchy}
$\SquaMS$ is closed under $C$.
$C$ may be considered as an endofunctor on $\SquaMS$.  
As such,  $i: Id \to C$ is a natural transformation.
\end{lemma}

\proof
Let $X$ be an object in $\SquaMS$ and consider $CX$. $CX$ is a metric space bounded by $2$, since $d_{CX}((x_i)_i,(y_i)_i) = \displaystyle{\lim_{i\rightarrow\infty}} d_X(x_i,y_i)\leq 2$ since $d(x_i,y_i)\leq 2$ for all $i$.

We endow the set $CX$ with the square set structure $S_{CX} = i_X \o S_X$.
Since $i_X$ and $S_X$ are injective, so is $S_{CX}$.

Since $i_X$ is an isometric embedding, $CX$ is not only a square set, it is a square metric space.

For example, to verify the first requirement of ($\sqone$) in Definition~\ref{definitionofsquams}, let $i\in {0,1}$ and $r,s\in [0,1]$.

\[\begin{array}{rcl}
d_{CX}(S_{CX}((i,r)),S_{CX}((i,s))) &= &d_{CX}(i_X(S_X((i,r))),i_X(S_X((i,s))))\\
&=& d_{CX}((S_X((i,r)))_k,(S_X((i,s)))_k)\\
&=& \displaystyle{\lim_{k\rightarrow\infty}} d_X(S_X((i,r)),S_X((i,s)))\\
&=& d_X(S_X((i,r)),S_X((i,s)))\\
&=& |s-r|.\\
\end{array}\]

The other condition in $(\sqone)$ and the requirements of ($\sqtwo$) follow from a similar argument.

If $f: X\to Y$ is a morphism of square spaces, then $f\o S_X = S_Y$.
And so
\[ Cf\o S_{CX}  =
Cf \o i_X \o S_X = i_Y \o f  \o S_X  = i_Y \o S_Y = S_{CY}.
\]
We are using the naturality of $i$ between endofunctors on $\mathcal{C}$.
Thus, $C$ is an endofunctor on $\SquaMS$. 
The same calculation shows that $i: Id \to C$ is a natural transformation
between functors on square spaces.
\endproof

We aim to show that
up to isomorphism,
$M\otimes -$ and $N\otimes -$ commute with $C$.  We will show the result for $M\otimes-$, but the proof for $N\otimes-$ is the same.
For any object $X$ in $\mathcal{C}$, consider 
\[ M\otimes C(X)\stackrel[]{\delta^M_X}{\longrightarrow} C(M\otimes X)\stackrel[]{\rho^M_X}{\longrightarrow} M\otimes C(X).\]
given by 
\[\begin{array}{lcl}
\delta^M_X(m\otimes (x_0,x_1,\ldots))
& =  & (m\otimes x_0,m\otimes x_1,\ldots) \\
\rho^M_X((m_k\otimes x_k)_k)
& =  &  m^*\otimes (x_{k_0},x_{k_1},\ldots),
\end{array}
\]
where $m^*$ is the first index in $M$ (via some order of the finite set $M$) which occurs infinitely many times in $(m_k\otimes x_k)_k$ and $x_{k_0},x_{k_1},\ldots$ are the corresponding elements of $X$.

\begin{lemma}\label{natisom} $\delta^M$ and $\rho^M$ are natural isomorphisms.
\end{lemma}
\proof
It is routine to check that for all $X$ in $\mathcal{C}$, $\delta^M_X$ and $\rho^M_X$ are well-defined,
that they are inverse functions (modulo equivalence of Cauchy sequences), that they are short maps, and thus, isometries.

We need to check that $\delta^M_X$ and $\rho^M_X$ preserve $S_X$.  
For $\delta^M_X$, let $(r,s)\in M_0$ and consider $S_{M\otimes C(X)}((r,s))$.  Then for some $m$ and $(r',s')$ which do not depend on $C(X)$, $S_{M\otimes C(X)}((r,s)) = m\otimes S_{C(X)}((r',s'))$.  $S_{C(X)}((r',s'))$ can be viewed as the limit of the constant sequence $(S_X((r',s')))$.    So $\delta^M_X(S_{M\otimes C(X)}((r,s)))= \delta^M_X(m\otimes (S_X((r',s')))) = (m\otimes S_X((r',s')))$, which is equal to the constant sequence $(S_{M\otimes X}((r,s)))$, whose limit is $S_{C(M\otimes X)}((r,s))$, as required.

For $\rho^M_X$, if $(r,s)\in M_0$, we can view $S_{C(M\otimes X)}((r,s))$ as the limit of the constant sequence $(S_{M\otimes X}((r,s)))$, which is equal to the constant sequence $(m\otimes S_X((r',s')))$ for some $m\in M$ and $(r',s')\in M_0$ only depending on $(r,s)$.  Then $\rho^M_X(S_{C(M\otimes X)}((r,s))) = m\otimes (S_X((r',s'))) = m\otimes S_{C(X)}((r',s')) = S_{M\otimes C(X)}((r,s))$.  
\endproof

We get analogous natural isomorphisms $\delta^N$ and $\rho^N$ for $N\otimes-$ defined in the same way.  Thus, we have the following.

\begin{proposition}\label{cauchycompletion} For $X$ in $\SquaMS$, $\delta_X^M:M\otimes C(X) \rightarrow C(M\otimes X)$ and $\delta_X^N:N\otimes C(X) \rightarrow C(N\otimes X)$ are isomorphisms. \end{proposition}

\section{The initial algebra of $M\otimes -$ obtained as the colimit of
its initial algebra $\omega$-chain}
\label{section-initial-algebra}

The overall message of this paper is that the Sierpinski carpet as a metric space
is bilipschitz equivalent to 
a final coalgebra of the endofunctor $M\otimes -$ on the category of square metric spaces.
However, to show this, we need a lot of material on a dual concept, \emph{initial algebras}.
It turns out that in our setting the final coalgebra is the Cauchy completion of the initial algebra.

\begin{definition}
Let $\A$ be a category and $F\colon\A\to \A$
an endofunctor.
An \emph{algebra for $F$} is a pair $(A,f)$, where
$A$ is an object, and $f: FA\to A$ is a morphism.
We call $A$ the \emph{carrier} and $f$ the 
\emph{structure (morphism)}.
A \emph{pre-fixed point} of $F$ is an
algebra whose structure is a monomorphism.

Let $(A,f)$ and $(B,g)$ be algebras for $F$.  An \emph{algebra morphism} 
from $(A,f)$ to $(B,g)$
is a morphism $\phi:A\to B$ in $\AA$ such that $\phi\o f = g \o F\phi$:
\[
\xymatrix@C+1pc{FA \ar[r]^-{f} \ar[d]_-{F\phi} 
    & A \ar[d]^-{\phi}\\ 
   FB   \ar[r]_-{g}  &  B
   } 
 \]
This gives a category $\AlgF$ of $F$-algebras, and an \emph{initial algebra} is an initial object in  $\AlgF$.
As expected, if such an algebra exists at all, it is unique up to isomorphism
in  $\AlgF$. 
 \end{definition}
 
 We recall a standard result in category theory, \emph{Lambek's Lemma}:
 if $(A,f)$ is an initial algebra, then $f$ is an isomorphism in the base category $\A$.

 \subsection{A pre-fixed point of $M\otimes -$.}

The main result of this section is the existence of an 
 initial algebra $M\otimes G \rightarrow G$ in $\SquaMS$.
 Before we start in on that,
 we exhibit a pre-fixed point related to the topic of this paper. 
Let $U_0=[0,1]^2$, equipped with the taxicab metric $d_{U_0}$, where
\[d_{U_0}((x,y),(x_1,y_1)) = |x-x_1|+|y-y_1|.\]
Define $\alpha_M: M\otimes U_0\rightarrow U_0$ by \begin{equation}\label{alphaMdef} (i,j)\otimes (r,s)\mapsto (\onethird(i+r),\onethird(j+s)).\end{equation} 
The following result is not immediate because the metrics are different in 
$U_0$ and $M\otimes U_0$.

\begin{lemma}\label{U0injective}
The map
 $\alpha_M:M\otimes U_0\rightarrow U_0$ is a monomorphism
of $\SquaMS$.   Thus, 
\[ (U_0,\alpha_M:M\otimes U_0\rightarrow U_0)\] is a pre-fixed point of $M\otimes-$.
\end{lemma}

\proof
First, it is easy to verify using the equivalences in $E$ that $\alpha_M$ preserves $S_{M\otimes U_0}$.  

We next show 
that $\alpha_M$ is injective.
To begin, if $\alpha_M((i,j)\otimes (r,s)) = \alpha_M((k,l)\otimes (t,u))$, when $(i,j)=(k,l)$, we must have $(r,s) = (t,u)$.  Otherwise, by examining cases we check that for any possible combination of $(i,j)$ and $(k,l)$, this equality forces $(r,s)$ and $(t,u)$ to be such that $(i,j)\otimes (r,s)$ and $(k,l)\otimes (t,u)$ are equal under the equivalence relation $E$. 

We next check that $\alpha_M$ is a short map.  Let $x=(i,j)\otimes (r,s)$ and $y=(k,l)\otimes (t,u)$ in $M\otimes U_0$.
Then
$x$ and $y$ fall into one of the following cases (up to possible rotation and reflection). 
\[
  \begin{tikzpicture}[>=stealth',shorten >=1pt,auto,node distance=2cm,semithick,scale=1.2]
  \draw [help lines] (0,4) grid (1,5);
 \draw (.3,4.6)  node   {$\cdot\, x$};
  \draw (.7,4.2)  node   {$\cdot\, y$};
  \draw [help lines] (2,3) grid (3,5);
 \draw (2.5,4.5)  node   {$\cdot\, x$};
  \draw (2.5,3.5)  node   {$\cdot\, y$};  
  \draw [help lines] (4,2) grid (5,5);  
   \draw (4.5,4.5)  node    {$\cdot\, x$};
  \draw (4.5,2.5)  node   {$\cdot\, y$};  
\draw [help lines] (6,2) grid (7,5);  
\draw [help lines] (7,2) grid (8,3);  
   \draw (6.5,4.5)  node    {$\cdot\, x$};
  \draw (7.5,2.5)  node   {$\cdot\, y$};  
  \end{tikzpicture}
\]

\[
  \begin{tikzpicture}[>=stealth',shorten >=1pt,auto,node distance=2cm,semithick,scale=1.2]
  \draw [help lines] (0,0) grid (1,3);
    \draw [help lines] (1,0) grid (3,1);
   \draw (0.5,2.5)  node    {$\cdot\, x$};   
      \draw (2.5,0.5)  node    {$\cdot\, y$};  
  \draw [help lines] (4,1) grid (5,3);  
    \draw [help lines] (5,2) grid (6,3);
   \draw (4.5,1.5)  node    {$\cdot\, y$};   
      \draw (5.5,2.5)  node    {$\cdot\, x$};  
  \draw [help lines] (7,0) grid (10,3);    
    \draw (8.5,2.5)  node    {$\cdot\, x$};   
      \draw (8.5,0.5)  node    {$\cdot\, y$}; 
 \end{tikzpicture}
\]

In each case it is reasonably routine to verify that $d_{M\otimes U_0}(x,y)\geq d_{U_0}(\alpha_M(x),\alpha_M(y))$, but we will examine one of these cases carefully, the one indicated in 
the lower-left corner.
Suppose that  $(i,j) = (0,2)$ and $(k,l) = (2,0)$, as shown.  
Note that
\[ d_{U_0}(\alpha_M(x),\alpha_M(y)) = \onethird |(i+r)-(k+t)|+\onethird |(j+s)-(l+u)|.
\]
Then, without loss of generality, the shortest path
in $M\otimes U_0$ 
between $x$ and $y$ is of the following form: 
\[
\begin{array}{l}
x=(0,2)\otimes (r,s), (0,2)\otimes (r_1,0)=(0,1)\otimes (r_2,1), \\
(0,1)\otimes (r_3,0)= (0,0)\otimes (r_4,1), (0,0)\otimes (1, s_1) = (1,0)\otimes (0,s_2),\\
(1,0)\otimes (1,s_3)=(2,0)\otimes (0,s_4), (2,0)\otimes (t,u)=y\\
\end{array}
\]

\[
  \begin{tikzpicture}[>=stealth',shorten >=1pt,auto,node distance=2cm,semithick,scale=1.2]
  \draw [help lines] (0,0) grid (1,3);
    \draw [help lines] (1,0) grid (3,1);
     \draw (0.5,2.5)  node     {$\bullet$};
       \draw (0.2,2.5)  node     {$x$}; 
      \draw (2.5,0.5)  node   {$\bullet$};  
            \draw (2.8,0.5)  node   {$y$};  
     \draw (0.6,2)  node  (a)  {$\bullet$};    
          \draw (0.6,1)  node   (b)   {$\bullet$};          
          \draw (1,.8)  node    {$\bullet$};   
              \draw (2,.8)  node    {$\bullet$};   
\draw (0.5, 2.5) -- (0.6,2);               
 \draw (0.6,2)  -- (0.6,1);    
  \draw (0.6,1) -- (1,.8);
    \draw (1,.8) --(2,.8);   
    \draw(2,.8) --(2.5,0.5) ;                    
   \end{tikzpicture}
\]
We estimate the score of this path.  First, we consider the horizontal components from each scaled copy of $U_0$. 
Their contribution to the score is
\def\arraystretch{1.1}
\[
\begin{array}{cl}
\geq & \onethird|r_1-r| + \onethird |r_2-r_1| +\onethird |1-r_2| +\onethird |1-0| + \onethird |t-0|\\
\geq & \onethird |2 + (t-r)|\\
= & \onethird |(i+r)-(k+t)|
\end{array}
\]
\def\arraystretch{1}
(The last equality holds because  $i=0$ and $k=2$.)
Similarly for the vertical components.  Thus, $d_{M\otimes U_0}(x,y)\geq d_{U_0}(\alpha_M(x),\alpha_M(y))$.

The other cases are similar. 

To conclude the proof, we recall that by Proposition~\ref{prop-mono}, injective functions give rise to monomorphisms in $\SquaMS$.
\endproof

 \subsection{Colimits of $\omega$-chains}
 
We apply Theorem~\ref{thm-adamek},
a well-known result in category theory,
to construct an initial algebra by taking 
\emph{the colimit of a certain $\omega$-chain} and verifying that the functor 
preserves this colimit.
We thus begin with a review of the definitions.  Even though we are mainly interested
in square metric spaces, we find it convenient to work somewhat more generally and 
also to study the situation in several related categories.

Let $\A$ be a category.
An 
 \emph{$\omega$-chain} in $\A$
is a functor from $(\omega, \leq)$ as a category into $\A$.
It is determined by an infinite sequence of objects and morphisms of $\A$ indexed by $\omega$:

 \begin{equation}\label{chain}
 \xymatrix{
A_0 \ar[r]^{a_{0}} & A_1 \ar[r]^{a_{1}} & \cdots & A_k \ar[r]^{a_k} & A_{k+1} & \cdots
 }
 \end{equation}
To turn this into a functor from $(\omega, \leq)$, we must specify
 \emph{connecting morphisms} $a_{k,\ell}$ for $k \leq \ell$.  
 We obviously take
 $a_{k,k} = \id_{A_k}$, and 
 then for $k < \ell$ we take $a_{k,\ell}$ to be the composition 
 $a_{\ell-1} \o a_{\ell-2} \o \cdots \o a_{k}$.

A \emph{cocone} of (\ref{chain}) is 
a pair $(B, (b_k)_k)$ consisting of
an object $B$ together with morphisms $b_k\colon A_k \to B$ so that that $b_k = b_\ell \o a_{k,\ell}$
when $k\leq \ell$. 
Sometimes we abuse notation slightly and write a cocone as $b_k\colon A_k\to B$, but technically
a cocone is an object together with a family of morphisms.
A \emph{colimit} of the chain  (\ref{chain}) 
is a cocone $(C,(c_k)_k)$ with the property that 
for every cocone  $(B, (b_k)_k)$ there is a unique morphism $f\colon C \to B$ so that $b_k = f \o c_k$ for all $k\in\omega$.

\begin{definition}
Consider an $\omega$-chain as in (\ref{chain}) with connecting morphisms $a_{k,\ell}$.  Let $(C,(c_k)_k)$ be a colimit. 
We say that $F$ \emph{preserves this colimit}
if 
the chain $FA_k$ with connecting morphisms $Fa_{k,\ell}$ has $(FC,(Fc_k)_k)$ as a colimit.
\end{definition}

Here is the reason that this is of interest in this paper.
\begin{theorem}[Ad\'amek~\cite{A74}]
\label{thm-adamek}
Let $\mathcal{A}$ be a category with initial object $0$.  Let $F:\mathcal{A}\rightarrow\mathcal{A}$ be an endofunctor.  
Consider the initial-algebra chain 
\begin{equation} 
\label{initialchain}
0\stackrel[]{!}{\longrightarrow} F0\stackrel[]{F!}{\longrightarrow} 
F^20\stackrel[]{F^2!}{\longrightarrow} \cdots F^k0\stackrel[]{F^k !}{\longrightarrow}  F^{k+1}0\cdots
\end{equation}
Suppose the colimit $G=colim_{k<\omega}F^k0$ exists, and write $g_k:F^k0\rightarrow  G$ for the cocone morphism.  
Suppose that $F$ preserves this colimit. 
Let $a:FG\rightarrow G$ be the unique morphism so that $a\circ Fg_k=g_{k+1}$ for all $k$.  Then $(G,a)$ is an initial algebra.  
\end{theorem}

We are especially concerned with the case
$\mathcal{A}=\SquaMS$, 
 $0= M_0$, and  $! =S_{M\otimes M_0}: M_0 \to M\otimes M_0$.
We shall show that with those choices, the colimit of the initial algebra $\omega$-chain exists, 
calling on much more general results.
Then we shall prove that the functor $M\otimes -$ preserves this colimit.

At various points in this paper
we are going to need  colimits of other $\omega$-chains in $\SquaMS$.
For every $M\otimes -$ coalgebra $(B,\beta)$, we need the chain below and its colimit.
\begin{equation}
B\stackrel[]{\beta}{\longrightarrow} M\otimes B\stackrel[]{M\otimes \beta}{\longrightarrow} M^2\otimes B\stackrel[]{M^2\otimes \beta}{\longrightarrow} 
\cdots M^k\otimes B\stackrel[]{M^k\otimes \beta}{\longrightarrow}  M^{k+1}\otimes B\cdots,
\label{eq-B-chain}
\end{equation}
We need the same colimit with $N$ replacing $M$, too.
We shall prove that the colimit of (\ref{eq-B-chain}) exists
and that it is preserved by the functor.
For this, we combine general facts about 
colimits in  \emph{pseudo-metric} spaces
with facts about the functors $M\otimes -$ and $N\otimes -$
which we have already seen.

We thus make a digression to study colimits of $\omega$-chains in greater generality.
We want to explore the colimits in sets, 
pseudo-metric spaces, metric spaces,
square sets, and square metric spaces.
In each case, we  characterize 
colimits of $\omega$-chains.

\subsection{Colimits of $\omega$-chains in Sets}

Suppose that we have an $\omega$-chain  in $\Set$
\begin{equation}\label{Achain}
A_0 \to A_1 \to \cdots \end{equation}
with connecting maps $a_{k,\ell}\colon A_k \to A_\ell$.
Suppose that we have a set $C$ and a cocone 
$(c_k)_{k\in\omega}$, where
$c_k\colon
    A_k \to C$.  Assume the following two properties:

    \begin{enumerate}
    \item[($\Set$1)]\label{P:setone}\  $C = \bigcup_{k} c_k[A_k]$, and
    \item[($\Set$2)]\label{P:settwo}\  Given $k\in\omega$ and elements
      $x,y\in A_k$ with $c_k(x) = c_k(y)$, there exists $\ell \geq k$ in
      $\omega$ such that $a_{k,\ell}(x) = a_{k,\ell}(y)$.
    \end{enumerate}
Note that ($\Set$2) implies a stronger form of the same statement:
if $x\in A_k$ and $y\in A_\ell$ with $k\leq \ell$ and $c_{k}(x)=c_\ell(y)$,
then there is $p\geq \max(\ell,k)$ such that $a_{k,p}(x) = a_{\ell,p}(y)$.
Here is how we see this.  Notice that $a_{k,\ell}(x)\in A_\ell$.   Apply ($\Set$2) to $a_{k,\ell}(x)$ and $y$ as elements
of $A_\ell$ to get some $p\geq \ell$ so that $a_{\ell,p}(a_{k,\ell}(x)) = a_{\ell,p}(y)$.   But $a_{\ell,p}(a_{k,\ell}(x)) = a_{k,p}(x)$.

We claim that $C$ with the morphisms $c_k:A_k\rightarrow C$ is a colimit.
Indeed, suppose that we are given a cocone $b_k\colon A_k\to B$.
We need to define a cocone morphism $f\colon C\to B$ and to prove that it is unique.
We define $f(c_k(x)) = b_k(x)$ for all $k\in \omega$ and $x\in A_k$.
This is a well-defined function due to  our observation in the previous paragraph.
It is defined on all of $C$, by ($\Set$1).
 It is a cocone morphism by definition. 
 And it is the unique such, since
 the condition   $f\o c_k = b_k$ gives the definition of $f$.
 
 \emph{Construction}
To prove the existence of a colimit of  (\ref{Achain}), we only need to find a set
$C$ and a cocone $(c_k)_k$ with
($\Set$1) and ($\Set$2).
Take the disjoint union $\sum_k A_k$, then take the  relation $\equiv$
given by
\[
(x,k) \equiv (y,\ell)\quad\mbox{iff}\quad \mbox{ $a_{k,p}(x) = a_{\ell,p}(y)$ for some $p\geq k,\ell$}
\]
In fact, this relation is an equivalence relation.
The quotient $C = (\sum_k A_k)/\!\!\equiv$ is then the colimit, with maps $c_k = A_k \to \sum_k A_k \to C$.  Conditions ($\Set$1) and ($\Set$2) are immediate.

\subsection{Colimits of $\omega$-chains in Pseudo-Metric Spaces} 
A pseudo-metric on a set $X$ is a \emph{distance function}
$d: X\times X \to [0,\infty]$
with the following properties:
$d(x,x) = 0$, $d(x,y) = d(y,x)$, and $d(x,z) \leq d(x,y) + d(y,z)$.
However, $d(x,y) = 0$ need not imply $x = y$.
A \emph{$2$-bounded} space has all distances bounded by $2$.
Let us consider the category $\PS$ of $2$-bounded
\emph{pseudo-metric spaces}.
As with the metric space categories in this paper, 
we take the morphisms in $\PS$ to be the short maps (also called non-expanding functions).
Let $\UU\colon \PS \to \Set$ be the forgetful functor.
As mentioned in~\cite{AdamekEA15} for the case of $1$-bounded spaces, 
$\PS$ is also 
cocomplete. That is, it has all colimits, not just colimits of $\omega$-chains.
We only need a special case of this, the result for colimits of $\omega$-chains.

\emph{Characterization of colimits of $\omega$-chains in $\PS$.}
Consider a chain 
\begin{equation}\label{AchainPS}
(A_0,d_0) \to (A_1,d_1) \to \cdots \end{equation}
with connecting short maps $a_{k,\ell}\colon A_k \to A_\ell$.
Suppose that we have a $\PS$-object $(C,d_C)$ with
short maps $c_k\colon
    A_k \to C$.  Assume the following three
    properties: 
    \begin{enumerate}
    \item [($\PS$1)] \label{P:setonePS} \ As sets, $C = \bigcup_{k} c_k[A_k]$, and
  \item[($\PS$2)]\label{P:settwoPS}\ Given $n\in\omega$ and elements
      $x,y\in A_k$ with $c_k(x) = c_k(y)$, there exists $\ell \geq k$ in
      $\omega$ such that $a_{k,\ell}(x) = a_{k,\ell}(y)$.

  \item[($\PS$3)]\label{P:fin-fb:3PS}\  For all $k$ and all $x,y\in A_k$,
     $d_C(c_k(x), c_k(y)) = \displaystyle{\inf_{p\geq k}}\  d_p(a_{k,p}(x), a_{k,p}(y))$.  
    \end{enumerate}
 We claim that $(C,d_C)$ is the   colimit of (\ref{AchainPS}) in
    $\PS$. 
    Due to ($\PS$1) and ($\PS$2), the underlying set $C$
    is a colimit of the $\omega$-chain in $\Set$ obtained by forgetting the pseudo-metric.
So, given a cocone $b_k\colon A_k\to B$ 
  in $\PS$,
we have a $\Set$ map $f\colon C\to B$
  (from above)
   given by $f(c_k(x)) = b_k(x)$.  We need only check that this map is 
  short.  First take a fixed $k$ and elements $x,y\in A_k$.   We want to show that 
  \[ d_{B}(f(c_k(x)), f(c_k(y))) \leq d_C(c_k(x), c_k(y)).\]
  This means that we want 
   \[ d_{B}(b_k(x), b_k(y))) \leq 
 \inf_{p\geq k} d_p(a_{k,p}(x), a_{k,p}(y)). \]  
For this, we can show that for all $p\geq k$, 
     \[ d_{B}(b_k(x), b_k(y))) \leq 
 d_p(a_{k,p}(x), a_{k,p}(y)) .\]  
 Now $b_p\colon A_p\to B$ is short, and $b_p\o a_{k,p} = b_k$ due to the cocone property.
 So 
 \[ d_p(a_{k,p}(x), a_{k,p}(y)) \geq 
 d_{B}(b_p (a_{k,p}(x)), b_p (a_{k,p}(y))   ) = d_{B}(b_k(x), b_k(y)). \]
 More generally, we need to consider $\ell\leq k$ and elements $x\in A_k$ and $y\in A_\ell$.
 In this case, $a_{\ell,k}(y) \in A_k$.
 So by what we just did,
  \[ d_{B}(f(c_k(x)), f(c_k(a_{\ell,k}(y)))) \leq d_C(c_k(x), c_k(a_{\ell,k}(y))).\] 
 But $c_k\o a_{\ell,k} = c_\ell$.
 So we have
   \[ d_{B}(f(c_k(x)), f(c_\ell(y))) \leq d_C(c_k(x), c_\ell(y)).\]

     \emph{Construction}
To prove the existence of a colimit of  (\ref{AchainPS}), we 
need only find a space with properties ($\PS$1) -- ($\PS$3) above.
Take the colimit $C$ in $\Set$.  This ensures ($\PS$1) and ($\PS$2).
Endow this set with the pseudo-metric 
\[
d^*(x,y) = \inf \set{d_k(x',y') : k< \omega,  x',y'\in A_k,
c_k(x') = x,  \mbox{ and } c_k(y') = y }.
\]
This ensures ($\PS$3).

\subsection{Colimits of $\omega$-chains in Metric Spaces}
Let $\MS$ denote the category of $2$-bounded metric spaces with short maps as morphisms, and suppose we have a chain in $\MS$:
\begin{equation}\label{AchainMS}
(A_0,d_0) \to (A_1,d_1) \to \cdots \end{equation}
with connecting short maps $a_{k,\ell}\colon A_k \to A_\ell$.
Suppose that we have a \emph{metric space} $(C,d_C)$ with
short maps $c_k\colon
    A_k \to C$.  Assume the following 
    properties:
    \begin{enumerate}
    \item[($\MS$1)] \label{P:setoneMS} \ As sets, $C = \bigcup_{k} c_k[A_k]$, and

   \item[($\MS$2)]\label{P:fin-fb:3MS}  \ For all $k$ and all $x,y\in A_k$,
     $d_C(c_k(x), c_k(y)) = \displaystyle{\inf_{p\geq k}}\  d_p(a_{k,p}(x), a_{k,p}(y))$.  
    \end{enumerate}
 We claim that $(C,d_C)$ is the   colimit of (\ref{AchainMS}) in
    $\MS$. 
Suppose that we have a     cocone $b_k\colon A_k\to B$ 
  in $\MS$.
  We want to define $f\colon C\to B$
as before, 
  by $f(c_k(x)) = b_k(x)$.
 To prove that $f$ is well-defined in $\Set$ or $\PS$, we had used a condition that we do not assume here,
 so the argument is different.

 Suppose that we have $k$ and $x,y\in A_k$ with $c_k(x) = c_k(y)$.
We want to show that $b_k(x) = b_k(y)$.   (The more general case of having $\ell,k$ and 
 $x\in A_k$, $y\in A_\ell$ with $c_k(x) = c_\ell(y)$ is treated similarly.)
By condition ($\MS$2), 
 \[ \inf_{p\geq k} d_p(a_{k,p}(x), a_{k,p}(y)) = 0 = d_C(c_k(x), c_k(y)).\]
Fix $\eps > 0$. 
There is some $p\geq k$ so that $ d_p(a_{k,p}(x), a_{k,p}(y)) \leq \eps$.
 Since $b_p \colon A_p \to B$ is short, 
 \[
 d_B(b_k(x), b_k(y)) = 
d_B(b_p(a_{k,p}(x)),  b_p(a_{k,p}(y)) ) \leq  d_p(a_{k,p}(x), a_{k,p}(y)) \leq \eps.
 \]   
This holds for all $\eps > 0$.  So $ d_B(b_k(x), b_k(y)) = 0$.
  Since $B$ is a metric space, $b_k(x) =  b_k(y)$.
  
This proves that $f$ is well-defined.  
 The same argument which we gave for $\PS$ shows that it is the colimit map in $\MS$.

     \emph{Construction}
To prove the existence of a colimit of  (\ref{AchainMS}), we 
need only find a metric space with properties ($\MS$1) and ($\MS$2) above.
Take the colimit $C$ in $\Set$,   
and endow it with the same pseudo-metric from before
\[
d^*(x,y) = \inf \set{d_k(x',y') : k<\omega, x',y'\in A_k,
i_k(x') = x,  \mbox{ and } i_k(y') = y }.
\]
Then let $x \sim y$ iff $d^*(x,y) = 0$.   This is an equivalence relation,
and so we can take the quotient $C/\!\!\sim$.  This quotient is (importantly) a metric space.
The natural map $C\to C/\!\!\sim$ does not change any non-zero distances.
From this, ($\MS$1) and ($\MS$2) follow easily.

\begin{example}

 Let $X_n=\set{u_n,v_n}$ be the space with two points and $d_n(u_n,v_n) = 2^{-n}$.
 Let $a_n: X_n \to X_{n+1}$ be the short map given by $a_n(u_n) = u_{n+1}$,
 and $a_n(v_n) = v_{n+1}$.
 We thus have an $\omega$-chain of metric spaces, and we take the colimit 
 in $\PS$ and in $\MS$.
In $\PS$, the colimit is a pseudo-metric space  consisting of two points of distance $0$.
This is not a metric space.
In $\MS$, the colimit is a single point.  These examples motivate the difference between
the three conditions ($\PS$1) --($\PS$3) and the two conditions ($\MS$1)--($\MS$2).

\end{example}

\subsection{Colimits of $\omega$-sequences in $\SquaSet$}

Suppose that we have a chain 
\[
A_0 \to A_1 \to \cdots 
\]
in $\SquaSet$.
Suppose that we have a cocone $a_k\colon A_k \to C$ in $\SquaSet$, and assume ($\Set$1) and ($\Set$2).
Then we claim that our cocone is the colimit in $\SquaSet$.   To see this,  we need only endow 
$C$  with a square set structure
$S_{C}\colon M_0 \to C$ and also show that given a cocone 
$b_k\colon A_k\to B$ in $\SquaSet$, 
the colimit map $f\colon C\to B$ preserves this structure.

We define $S_C$ by $c_0 \o S_{A_0} : M_0 \to A_0 \to C$.
To see that this works, note that since $b_0 \colon A_0 \to B$ is a square space map, $b_0\o S_{A_0} = S_B$.
Thus
\[ f \o S_C = f\o  c_0 \o S_{A_0} =  b_0 \o S_{A_0} = S_B.
\]
Since $S_B$ is injective and $f\o S_C = S_B$,
$S_C$ is also injective.

\subsection{Colimits of $\omega$-sequences in $\SquaMS$}

Consider next a chain 
\[
(A_0, d_0) \to (A_1, d_1) \to \cdots 
\]
in $\SquaMS$.
Suppose that we have a cocone $a_k\colon A_k \to C$ in $\SquaMS$, and assume ($\MS$1) and ($\MS$2).
Then we claim that our cocone is the colimit in $\SquaMS$.  

We know how to take the colimit $C$ in $\SquaSet$, endowing $C$ with a $\SquaSet$ structure.
We also know how to take the colimit in $\MS$.
    So the only remaining point is to check the non-degeneracy requirements 
($\sqone$) and ($\sqtwo$).   
To check $(\sqone)$, let $r,s\in [0,1]$ and consider $S_{C}((r,0))$ and $S_{C}((s,0))$ (the other cases are similar).  Then 
\[d_{C}(S_{C}((r,0)),S_{C}((s,0))) = 
\displaystyle{\inf_{k<\omega}} d_k(S_{A_k}((r,0)),S_{A_k}((s,0))) = |r-s|,\] 
since each $A_k$ is an object in $\SquaMS$.  
Similarly, to check $(\sqtwo)$, let $(r,s),(t,u)\in M_0$.  Then 
\[ d_{C}(S_{C}((r,s)),S_{C}((t,u))) = 
\displaystyle{\inf_{k<\omega}}\ d_k(S_{A_k}((r,s)),S_{A_k}((t,u)))\geq |r-t| +  |s-u|,\]

\subsection{$M\otimes -$ preserves colimits of $\omega$-chains}

We next show that the functor $M\otimes -$ preserves colimits of $\omega$-chains.
This result is used in 
Section~\ref{applyAdamek}, where we apply
 Ad\'amek's Theorem~\ref{thm-adamek} to construct the initial algebra of this functor.

\begin{lemma}
\label{M-preserves}
The endofunctor $M\otimes -$ preserves colimits of $\omega$-chains.
\end{lemma}

\proof
    Consider a chain
\[
(A_0, d_0) \to (A_1, d_1) \to \cdots 
\]
in $\SquaMS$, and let its colimit be
the space $(C,d_C)$ with colimit cocone 
$(c_k)_k$, where $c_k:A_k\to C$.
We are going to show that the colimit of
\begin{equation}\label{secondchain}
(M\otimes A_0,  d_{M\otimes A_0}) \to (M\otimes  A_1, d_{M\otimes A_1}) \to \cdots 
\end{equation}
is $(M\otimes C, (M\otimes c_k)_k)$.
To begin, 
we already know that the cocone 
$(C, (c_k)_k)$
has properties ($\MS$1) and ($\MS$2) for the original chain.
We need only check that 
$(M\otimes C, (M\otimes c_k)_k)$ has these same properties ($\MS$1) and ($\MS$2)   for the chain in (\ref{secondchain}).

For ($\MS$1), it is clear that as sets,
\[ M\otimes C = M \otimes \bigcup_k c_k[A_k] = 
\bigcup_k \biggl( M \otimes  c_k[A_k]\biggr)  = \bigcup_k (M\otimes c_k)[A_k].
\]
For ($\MS$2), we want to show that
for all $k\in\omega$ and all $x,y\in A_k$, and all $m, n\in M$,
\begin{equation}
\label{want}
 d_{M\otimes C}(m\otimes c_k(x), n\otimes c_k(y)) = \inf_{p\geq k} d_{M\otimes A_p}(m \otimes a_{k,p}(x), n\otimes a_{k,p}(y)).
 \end{equation} 
 We first consider the case when $n = m$.
In this case, 
\[
\begin{array}{lcll}
d_{M\otimes C}(m\otimes c_k(x), m\otimes c_k(y)) 
   & = & \frac{1}{3} d_{C}(c_k(x), c_k(y)) \\ \\
   & = & \frac{1}{3}\ \displaystyle{\inf_{p\geq k}}\  d_p(a_{k,p}(x), a_{k,p}(y)) \\ \\
  & = &  \displaystyle{\inf_{p\geq k}} \ \frac{1}{3} d_p(a_{k,p}(x), a_{k,p}(y))\\ \\
  & = &  \displaystyle{\inf_{p\geq k}}\ d_{M\otimes A_p} (m\otimes a_{k,p}(x),m \otimes a_{k,p}(y)).
\end{array}
\]
With this special case done, we consider the general case.    We use the fact from Theorem~\ref{quotientmetric}
that in $M\otimes C$, there is a fixed path that attains the distance between our points $m\otimes c_k(x)$ and $n\otimes c_k(y)$.
This path has finitely many sub-paths (at most $5$ in fact), and each subpath is in one and the same copy of $C$.
It follows from our first observation that  (\ref{want}) holds.
 
 This concludes the proof.\endproof

We also have a result 
exactly like Lemma~\ref{M-preserves}
but for the functor $N\otimes -$.
The details are basically the same.

\subsection{Using colimits to obtain the initial algebras of $M\otimes-$ and $N\otimes -$}
\label{applyAdamek}

At this point, we recall Ad\'amek's Theorem (Theorem~\ref{thm-adamek}), 
 and apply this to $\SquaMS$, with $F$ either $M\otimes -$ or $N\otimes -$.
As we know, colimits of all $\omega$-chains exist in our category.
We are of course interested in the colimit of 
the initial-algebra chain (\ref{initialchain}).
The functors preserve this colimit, since they preserves all colimits of $\omega$-chains.
Thus, there is are initial algebras.
We write these as 
\begin{equation}\label{GW}
\begin{array}{l}
(G,\eta\colon M\otimes G\to G) \\
(W,\lambda\colon N\otimes W \to W) \\
\end{array}
\end{equation}
In both cases, the algebra structures  are isometries, by Lambek's Lemma.

Further, the colimit morphisms are given by the natural equivalence relations.
For example, consider $M$ (the functor where we need this remark).
We have 
\begin{equation}\label{ikmaps}
    g_k\colon M^k\otimes M_0\rightarrow G
\end{equation}
given by $g_k(x) = [x]$, where the equivalence relation involved here 
relates, for $l\leq m$,  $y\in M^\ell\otimes M_0$ with $z\in M^m\otimes M_0$ 
iff $a_{\ell,m}(y) = z$, 
where $a_{\ell,m}\colon  M^\ell\otimes M_0 \to  M^m\otimes M_0$ is the 
evident map.

\section{Final coalgebras for $N\otimes-$ and $M\otimes -$}
\label{section-final-coalgebras}

This section discusses final coalgebras for the two main functors in this paper,
$N\otimes -\colon \SquaMS\to \SquaMS$, and $M\otimes -\colon \SquaMS\to \SquaMS$.
The main results are that the unit square $U_0$ with the taxicab metric is a final coalgebra for
$N\otimes -\colon \SquaMS\to \SquaMS$, and that this coalgebra is the Cauchy completion
of the initial algebra.   Turning to $M\otimes -$, we show that again the 
 Cauchy completion
of the initial algebra is the final coalgebra.  
It would have been pleasing if this final coalgebra had been the Sierpinski carpet $\SC$. 
But this is not to be: the bijective map $\SC\to M\otimes\SC$ is not a short map.
Nevertheless, we shall prove later than $\SC$ is bilipschitz equivalent to the carrier of
 final coalgebra
of $M\otimes -$.  
In a different direction, forgetting the metric, $\SC$ is the final coalgebra of our functor on $\SquaSet$.

\begin{definition}
Let $H\colon \A\to \A$ be an endofunctor on any category.
A \emph{coalgebra for $H$} is a pair $(A,a)$, where  $a \colon A \to HA$.
Given two coalgebras $(A,a)$ and $(B,b)$ for this functor, 
a \emph{coalgebra morphism} is a 
morphism $h\colon A \to B$ in $\A$ such that $b\o h  = Hh\o a$:
\[
\xymatrix@C+1pc{A \ar[r]^-{a} \ar[d]_-{h} 
    & HA \ar[d]^-{Hh}\\ 
   B   \ar[r]_-{b}  &  HB
   } 
 \]

$(A,a)$ is a \emph{final coalgebra} if for every coalgebra $(X,e)$ there is a unique coalgebra morphism $e^\dag:X\rightarrow A$.
Equivalently, it is a final object in the category of coalgebras. 
\end{definition}

Final coalgebras need not exist, but when they do, they are unique up to isomorphism.  Moreover, 
if $(C,\gamma)$ is a final coalgebra, then by Lambek's Lemma (the dual of the form that we stated earlier),
$\gamma$ is an isomorphism in the base category $\A$.

\subsection{Corecursive algebras}
\label{section-corecursive-algebras}

Our work on final coalgebras involves a secondary notion: \emph{corecursive algebras}.
We bring corecursive algebras into the paper because they generalize final coalgebras
and because the Sierpinski carpet turns out to be a corecursive algebra in $\SquaMS$.

\begin{definition}
Let $H\colon \A\to \A$ be an endofunctor on any category.
An algebra $a\colon HA \to A$ is \emph{corecursive} if for every
coalgebra $e\colon X\to HX$ there is a unique 
\emph{coalgebra-to-algebra
morphism $e^\dag\colon X \to A$}. This means that 
$e^\dag = a\o He^\dag \o e$:
\[
\xymatrix@C+1pc{X \ar[r]^-{e} \ar[d]_-{e^\dag} 
    & HX \ar[d]^-{He^\dag}\\ 
   A  & \ar[l]^-{a}   HA
   } 
 \]
The map $e^\dag$ is also called \emph{the solution to $e$
 in the algebra $(A,a)$}.
\end{definition}

The following is the dual form for Proposition 7 in~\cite{CAPRETTA2006437}.

\begin{proposition}\label{prop-invertible}
If a corecursive $H$-algebra $(A,a)$ has an invertible structure map $a$,
then $(A,a^{-1})$ is a final coalgebra for the same functor. 
If $(A, a)$ is a final coalgebra, then  $(A,a^{-1})$ is a corecursive algebra.
\end{proposition}

\begin{lemma} \label{lemma-morphisms} 
Let $e\colon X \to HX$ and $f\colon Y \to HY$
be coalgebras, and let $h\colon X\to Y$ be 
a coalgebra morphism.
Let $a\colon HA \to A$ be a corecursive algebra.
Then $e^\dag = f^\dag\o h$.
\end{lemma}

The proof of this may be found in Example 3.2 in \cite{lmcs:707}.

Recall that $N = \{0,1,2\}^2$, and that $U_0 = [0,1]^2$.
We are going to consider the functor $H_0 X = N\times X$ on $\Set$.  

Recall from the previous section
 our definition of $\alpha_M:M\otimes U_0\rightarrow U_0$, which we proved was an injective morphism in Lemma~\ref{U0injective}.  Here we will introduce some notation towards defining an analogous morphism $\alpha_N:N\otimes U_0\rightarrow U_0$.

Let $\shrink\colon N \to U_0$ be given by
\[\shrink(i,j) =
(\onethird i, \onethird j).
\]
We have an $H_0$-algebra structure $\alpha_0\colon N\times U_0 \to U_0$ given by 
\[\alpha_0((i,j),(x,y)) =  \shrink(i,j)+ \onethird (x,y).\]

\begin{lemma}
\label{prop-U0-N}
$(U_0,\alpha_0\colon N\times U_0 \to U_0)$ is a corecursive algebra for $H_0$ on $\Set$.
\end{lemma}

\begin{proof} Although it is possible to give a self-contained elementary proof, this result also
follows from Corollary 2.11 in~\cite{amvElgot} (see also~\cite[Example 7.3.10]{AMM}). 
We must check a few hypotheses to apply that result.   We discuss these one-by-one.

Let $\CMS$ be the category of complete metric spaces with distances bounded by $2$.  We have a forgetful functor $U:\CMS\to \Set$.\footnote{A forgetful functor is standardly denoted by $U$.   For us, this has an unfortunate clash with our notation $U_0$ for the unit square.  Bringing this to the reader's attention should help avoid any confusion.}  We verify three hypotheses.

First,  the functor $H_0 = N\times X: \Set\to \Set$ lifts to $\CMS$.  This has nothing to do with our specific set $N$, it holds for all sets $N$.
Here is what this means.   Consider $N$ as a discrete space with distance $2$ between all points.  Then we have a functor $H_1\colon\CMS\to\CMS$ given by $H_1 X = N\times X$, with the 
metric defined as follows:
\[
d(((i,j),(x,y)), ((i',j'),(x',y')))
= \left\{ \begin{array}{ll} 2 & \mbox{if $(i,j) \neq (i',j'))$} \\
                  \onethird d_{U_0}((x,y),(x',y')) & \mbox{if $(i,j) = (i',j'))$}
\end{array}
\right.
\] 
$H_1$ works as expected on morphisms.  The lifting property is that $U\o H_1 = H_0\o U$, and this is easy to check.

Second, this lifted functor $H_1$ is \emph{locally contracting}.  Indeed, 
for all ``parallel pairs'' of  $\CMS$-morphisms $f,g: X\to Y$, $d(H_1 f, H_1 g) = \frac{1}{3} d(f,g)$.
This is a routine verification using the supremum metric on function spaces and the distance formula above.

Finally, the $\Set$-morphism $\alpha_0\colon N\times U_0 \to U_0$ also is a $\CMS$-morphism $\alpha_0\colon H_1 U_0 \to U_0$.
This means that $\alpha_0$ is short.   To check this, take two elements of $H_1 U_0$, say
$p = ((i,j), (x,y))$ and $p' = ((i',j'), (x',y'))$.   If $(i,j) \neq (i',j')$, then their distance in $N$ is $2$, and hence
the distance between $p$ and $p'$ is also $2$.  But the distance between $\alpha(p)$ and $\alpha(p')$ is at most $2$.
In the other case, $(i,j)= (i',j')$.  In this case,
\[ d_{H_1 U_0}(\alpha(p),\alpha(p')) =
\onethird d_{U_0}((x,y), (x',y')) = d_{U_0}(\onethird(x,y), \onethird(x',y'))  .\]
These hypotheses then imply that $(U_0, \alpha_0)$ is  a corecursive algebra for $H_0$ on $\Set$.
\end{proof}

Lemma~\ref{prop-U0-N} was a preliminary
result; the main point is
Lemma~\ref{lemma-U0-N-otimes},
its  
 adaptation  for
 the category $\SquaSet$ of square sets.

\begin{definition}
\label{definition-alpha-N}
Let  $\alpha_N\colon N \otimes U_0 \to U_0$ be given by
\[
\begin{array}{lcl}
\alpha_N(n\otimes z) & = & \shrink(n) + \frac{1}{3}(z).
\end{array}
\]
\end{definition}

Notice that $n\in N$ here is a pair; earlier we
had written it as $(i,j)$.  Similarly, $z\in U_0$;
earlier we wrote it as $(r,s)$.  It takes quite a few
routine elementary calculations to check that $\alpha_N$
is well-defined.  That is, we must check that 
if $(n,z) \approx (n',z')$, then $ \shrink(n) + \frac{1}{3}(z) = 
 \shrink(n') + \frac{1}{3}(z')$.
 For example, we  have 
 $((0,0), (r,1))\approx((0,1),(r,0))$.
 And
 \[
 \shrink(0,0) + \onethird(r,1) = 
 (\rthirds, \onethird) = 
  \shrink(0,1) + \onethird(r,0).
 \]

 Furthermore, it is easy to verify that $\alpha_N$ preserves $S_{N\otimes U_0}$, so $\alpha_N$ is a $\SquaSet$ morphism. 
 
\begin{lemma}
\label{alphaNisomorphism}
In $\SquaMS$, 
$\alpha_N\colon N\otimes U_0 \to U_0$
is an isomorphism: it maps $N\otimes U_0$ one-to-one onto $U_0$,
and it is an isometry.
\end{lemma}

\begin{proof}
Clearly $\alpha_N$ is surjective: given $(r,s)\in U_0$, let $(i,j)\in N$ be the greatest in the lexicographic order such that $\frac{1}{3}i\leq r$ and $\frac{1}{3}j\leq s$.  Then $\alpha_N((i,j)\otimes (3r-i,3s-j)) = (r,s)$.  

To see that $\alpha_N$ is injective, we will show that it is an isometry.

First, let us check that $\alpha_N$ is a short map.
\def\arraystretch{1.5}
Taking $p=1$ and $B = U_0$ in (\ref{eq-taxicab-like}), we see that 
\[\begin{array}{cl}
 &  d_{N\otimes U_0}((i,j)\otimes S_{U_0}((r,s)), (k,\ell)\otimes S_{U_0}((t,u)))
\\
\geq &
\bigl| |i;r| - |k;t|\bigr| - \bigl| |j;s| - |\ell;u|\bigr| \\
= &\bigl| \onethird(i+r) -\onethird(k+t)\bigr| - \bigl| \onethird(j+s) -\onethird(\ell+u)\bigr| \\  
= &  d_{\Taxicab}((\onethird(i+r), \onethird(j+s)), (\onethird(k+t), \onethird(\ell+u)))\\ 
= & d_{\Taxicab}(\alpha_N((i,j)\otimes (r,s)), \alpha_N((k,\ell)\otimes (t,u))).\\
\end{array}
\]
\def\arraystretch{1}

Now to see that this is an isometry, consider $\alpha_N(x)$ and $\alpha_N(y)$ in $U_0$.  The idea is that we can introduce a grid to $U_0$ which corresponds to the boundaries of copies of $U_0$ in $N\otimes U_0$, and look at the intersections of a segment between $\alpha_N(x)$ and $\alpha_N(y)$ with that grid.  We then use this to construct a path in $N\otimes U_0$ between $x$ and $y$ whose distance is equal to that between $\alpha_N(x)$ and $\alpha_N(y)$ in $U_0$, and this will be an upper bound of the distance between $x$ and $y$ in $N\otimes U_0$. 

Rather than work through the thick notation of a general case, we will present the following illustrative example.

 We consider 
 \[
   \begin{tikzpicture}[>=stealth',shorten >=1pt,auto,node distance=2cm,semithick,scale=1.5]
  \draw [help lines]  grid (3,3);
 \draw (.8, 2.7)  node (start)  {$\bullet$};
  \draw (.5, 2.7)  node (startletter)  {$A$};
 \draw (1,2.38) node (a)   {$\bullet$};
\draw (.7,2.38) node (a)   {$B$};
\draw (1.2375,2) node (b) {$\bullet$};
\draw (.9375,2) node (b) {$C$};

\draw (1.8625,1) node (d)  {$\bullet$};
\draw (2.1625,1.06) node (d)  {$D$}; 
\draw (2,.78) node (c)   {$\bullet$};
\draw (2.3,.75) node (c)   {$E $}; 
\draw (.8,2.7) -- (2.3,0.3);
  \draw (2.3,.3)  node (end) {$\bullet$};
    \draw (2.6,.3)  node (endletter) {$F$};
  \end{tikzpicture}
\]
These are points in $N\otimes U_0$, shown explicitly on the left below.
The column on the right gives their images under $\alpha_N$, as elements of $U_0$.
 
\[
\begin{array}{lclcl}
A & = & (0,2) \otimes (.8,.7) &  & 
\\
B & = & (0,2) \otimes (1,.38) & = & (1,2) \otimes  (0,.38)    \\
C & = &(1,2) \otimes   (.2375,0)  & = &  (1,1) \otimes (.2375,1) \\
D & = & (1,1) \otimes (.8625,0) & = &  (1,0) \otimes (.8625,1) \\
E & = & (1,0) \otimes (1,.78) & =& (2,0) \otimes (0,.78)\\
F & = & (2,0) \otimes  (.3,.3) &               \\
\end{array}
\qquad
\begin{array}{lcl}
\alpha_N(A) & = & (.267, .9) \\
\alpha_N(B) & = & (.333,.793) \\
\alpha_N(C) & = & (.4125,.667) \\
\alpha_N(D) & = & (.621,.333) \\
\alpha_N(E) & = & (.667,.26) \\
\alpha_N(F) & = & (.767, .1) \\
\end{array}
\]
The way we got these was to find the line between
$\alpha_N(A)$ and $\alpha_N(F)$, then to find the intersection points
of this line with the relevant grid lines, and finally to find 
the preimages under $\alpha_N$.  For $B$, $C$, $D$, and $E$, we have two preimages.

We are going to verify that \[d_{U_0}(\alpha_N(A), \alpha_N(F)) \geq
d_{N\otimes U_0}(A,F).\]
Recall, we are using the taxicab metric in $U_0$ (\ref{eq:taxicab}).  So 
\[ d_{U_0}(\alpha_N(A), \alpha_N(F)) = |.767-.267| + |.1-.9| = .5 + .8 = 1.3.\]
To show that $d_{N\otimes U_0}(A,F) \geq 1.3$, we find an alternating path
(a sequence of points in $N\times U_0$ as described in Definition~\ref{simAlternatingPath})
from $A$ to $F$ and check that the score of this path is again $1.3$.
The alternating path we want is suggested by $A$, $\ldots$, $F$. 
It is

\[
\begin{array}{lclclclc}
((0,2) , (.8,.7))   ,  
((0,2) , (1,.38))  \sim   ((1,2),  (0,.38))  ,\\
((1,2) , (.2375,0))  \sim   ((1,1) , (.2375,1)), \\ 
((1,1), (.8625, 0))  \sim  ((1,0), (.8625,1)),\\
((1,0) , (1,.78)) \sim  ((2,0) , (0,.78)),
((2,0) , (.3,.3)).
\end{array}
\]

The score of this alternating path is 
\[
\onethird (.2 + .32 + .2375 + .38 +.625 + 1 + .1375 + .22 + .3 + .48).
\]
The relationship between this and our calculation of $ d_{U_0}(\alpha_N(A), \alpha_N(F))$
is clarified if we separate the horizontal and vertical contributions.
Our score above is
\[\begin{array}{cl}
 & \onethird((.2 +  .2375 + .625 +   .1375   + .3  )
+ (.32 + .38 + 1 + .22 + .48))\\
= & \onethird (1.5 + 2.4) \\
= & 1.3
\end{array}
\]
This is as desired.  This all is merely an example, but 
the general case is similar.  We conclude that $d_{N\otimes U_0}(x,y) \leq d_{U_0}(\alpha_N(x), \alpha_N(y))$.

Thus, $\alpha_N$ is an isometry, so it is injective, and hence an isomorphism in $\SquaMS$.  
\end{proof}

 For every $\SquaSet$ $B$ there is a canonical quotient map in $\Set$,  $\nu_B: N\times B \to N\otimes B$.
 It is given by $\alpha_X((n,x)) = n\otimes x$.

\begin{proposition}\label{prop-unnatural}
For every $\SquaSet$ morphism $f: B\to C$ the evident ``naturality square''  commutes:
 $(N\otimes f) \o \nu_B = \nu_C \o (N\times f)$.
\end{proposition}
 
 \begin{proof}
  The for $(n,x)\in N\times X$,
\[
(\eta_Y\o (N\times f))((n,x)) = n\otimes f(x)
((N\otimes f)\o \eta_X)(n,x).
\]
\end{proof}

\begin{lemma}
\label{lemma-U0-N-otimes}
$(U_0,\alpha_N\colon N\otimes U_0 \to U_0)$ is a corecursive algebra for $N\otimes X$ on $\SquaSet$.
\end{lemma}

\proof
We are given a coalgebra $e\colon B \to N\otimes B$,
 and it is our task to show that there
is a unique $e^\dag: B\to U_0$ in  $\SquaSet$
such that $e^\dag = \alpha_N\o (N\otimes e^\dag) \o e$.
We first make a diagram in $\Set$:
 \[
    \begin{tikzcd}
  M_0   \arrow[swap]{d}{S_B} \arrow{r}{\widehat{S}_{N\otimes M_0}} 
  & N\times M_0  \arrow[swap]{d}{N\times S_B} & \\
 B \arrow{r}{\widehat e}
      \arrow{d}[swap]{e^\dag}
      &
N\times B  
        \arrow[d,swap, "N\times e^\dag"]
      \arrow{r}{i}
&
N\otimes B
         \arrow{d}{N\otimes e^\dag}
      \\
    U_0
      &
      N\times U_0       \arrow{l}{\alpha_0}
      \arrow{r}{\nu}
      &
       N\otimes U_0
             \arrow[bend left=30]{ll}{\alpha_N}
    \end{tikzcd}
  \] 
$S_B$ and $\alpha_N$ are $\SquaSet$ morphisms, and so what we mean above
is the same maps in $\Set$. 
 The maps $\nu_B$
 and $\nu_{U_0}$ are the  ones we saw in Proposition~\ref{prop-unnatural}.
We will discuss the map $N \otimes e^\dag$ later, after we define $e^\dag$ and verify that
it is a square set morphism.

Please note that we have changed the notation on the one of the ``hat'' map, writing
$\widehat{S}_{N\otimes M_0}$ instead of $\widehat{S_{N\otimes M_0}}$.

The morphisms $\widehat{e}$ and $\widehat{S}_{N\otimes M_0}$
are defined in a canonical way,
as follows. 
Fix an ordering $<$ on $N$, say the lexicographic order.
 First, consider $\widehat{e}$. 
 Let $\widehat{e}(b)$ be any pair $(n,b') \in N\times B$ such that $e(b) = [\widehat{e}(b)]$
 and  $n$ is $<$-least in $N$ such that some $b'$ exists with this property.   This defines $n$ uniquely,
 and it is easy to see that $b'$ is also unique.  This is because if $(n,b) \approx (n,b')$,
 then $b = b'$.  We see easily that $e = i\o \widehat{e}$. 
The morphism $\widehat{S}_{N\otimes M_0}$ is similar.   
As an example, $\widehat{S}_{N\otimes M_0}(1/3,0) = ((0,0),(1,0))$.
By the way, despite the notation,
$\widehat{S}_{N\otimes M_0}$ is a morphism in $\Set$ here.

 By Proposition~\ref{prop-U0-N}
 applied to $\widehat{e}$,  we get $(\widehat{e})^\dag$, making the square in the corner commute.  We will shorten this to $e^\dag$, as this will turn out to be the $\SquaSet$ morphism we want.  
  The definitions of 
 $\widehat{e}$ and $\widehat{S}_{N\otimes M_0}$ and the fact that $e$ is a $\SquaSet$ morphism 
 imply that the square in the upper-left commutes.  For example, consider $(\frac{1}{3},0)\in M_0$. 
 Now $(N\times S_B)\circ \widehat{S}_{N\otimes M_0}((\frac{1}{3},0)) = N\times S_B(((0,0),(1,0))) = ((0,0),S_B((1,0)))$  because $((0,0), S_B((1,0)))$ is the first representative of the class $S_{N\otimes B}((\frac{1}{3},0))$ according to the lexicographic order on $N$.  Similarly, $\widehat{e}\circ S_B((\frac{1}{3},0)) = \widehat{e}(S_B((\frac{1}{3},0))) = ((0,0), S_B((1,0)))$ since $(0,0)\otimes S_B((1,0))$ is the first representation of $S_{N\otimes B}((\frac{1}{3},0))$ according to the lexicographic ordering on $N$ (as opposed to $(1,0)\otimes S_B((0,0))$).  
So we get a coalgebra
 morphism for the functor $N\times-$.
 By Lemma~\ref{lemma-morphisms},
  $\widehat{S}_{N\otimes M_0}^\dag = e^\dag\o S_B$.
  
We claim that  $\widehat{S}_{N\otimes M_0}^\dag = S_{U_0}$.
 That is, 
 we claim that
 $S_{U_0}$ satisfies the corecursive algebra condition which uniquely defines
 $\widehat{S}_{N\otimes M_0}^\dag$:
 \[S_{U_0} = \alpha_0 \circ (N\times S_{U_0})
    \circ    \widehat{S}_{N\otimes M_0} 
 \]
  We verify an example.   For example, for 
  $1/3 < r \leq 2/3$,
 \[\begin{array}{lcl} 
(\alpha_0 \circ (N\times S_{U_0})
    \circ \widehat{S}_{N\otimes M_0})(r,0)
   &  =  & \alpha_0 \circ (N\times S_{U_0})((1,0),(3r-1 ,0))\\
    & = & \alpha_0((1,0), S_{U_0}((3r-1,0))) \\
     & = & \alpha_0((1,0),  (3r-1,0))\\
      & = &  (r,0) \\
      & = & S_{U_0}((r,0))\\
 \end{array}
 \]
 All of the other cases are similar.
 By uniqueness of
 solutions, $S_{U_0}$ \emph{is} the solution.
 The upshot is that at this point we know that
  $e^\dag\o S_B =
  \widehat{S}_{N\otimes M_0}^\dag = 
  S_{U_0}$, and hence that
 $e^\dag$ is a $\SquaSet$ morphism.

Now that we know that $e^\dag$ is a $\SquaSet$ morphism, we use the functor $N\otimes -$ on $\SquaSet$ to get
 $N\otimes e^\dag\colon N\otimes B \to N\otimes U_0$; recall that this is defined by
$(N\otimes e^\dag)(n\otimes b) = n\otimes e^\dag(b)$. 
The square in the bottom commutes by Proposition~\ref{prop-unnatural}.
The region on the bottom commutes: $ \alpha_N \o \nu = \alpha_0$.
Recalling that $e = \nu_B\o \widehat{e}$, a diagram chase shows that we have the 
desired equality  $e^\dag = \alpha_N\o (N\otimes e^\dag) \o e$.

We also check that  $e^\dag$ is the  unique solution of $e$ in $\SquaSet$.
Suppose that we have a $\SquaSet$ morphism $e^*$ so that 
$e^* = \alpha_N\o (N\otimes e^*) \o e$.  We show that $e^* = e^\dag$.  
Consider the diagram below:
 \[
    \begin{tikzcd}
 B \arrow{r}{\widehat e}
      \arrow[bend left=30]{rr}{e}
      \arrow{d}[swap]{e^*}
      &
N\times B  
        \arrow[d,swap, "N\times e^*"]
      \arrow{r}{\nu_B}
&
N\otimes B
         \arrow{d}{N\otimes e^*}
      \\
    U_0
      &
      N\times U_0       \arrow{l}{\alpha_0}
      \arrow{r}{\nu_{U_0}}
      &
       N\otimes U_0
             \arrow[bend left=30]{ll}{\alpha_N}
    \end{tikzcd}
  \]
  Since $e^*$ is a morphism in $\SquaSet$, we are entitled to write
$N\otimes e^*$, as shown.  But the diagram
  above is in $\Set$.  
The top and bottom commute, as we have seen.
The square on the right commutes, easily.
The verification here is similar to what we saw in the first part of the proof.
And now a diagram chase shows that the square on the left commutes as well.
But this means that $e^*$ is a solution to the $N\times-$ coalgebra $(B,\widehat{e})$.
And so by uniqueness of solutions in $(U_0,\alpha_0)$, $e^* = e^\dag$.
 \endproof

The next main result is that $(U_0,\alpha_N^{-1})$ is a final 
$N\otimes-$coalgebra in square \emph{metric spaces}.
Here the metric on $U_0$ is the taxicab metric.
We need a few preliminary lemmas.  
In these, we fix an $(N\otimes-)$-coalgebra in $\SquaMS$,
 $(B,\beta:B\rightarrow N\otimes B)$.
   We already know that there is a unique $\SquaSet$ morphism $\beta^\dag\colon B\to U_0$ such that $\beta^\dag = \alpha_N\circ (N\otimes\beta^\dag)\circ \beta$.
   Also, $\alpha_N$ is an isometry (see Lemma~\ref{alphaNisomorphism})
  hence $\alpha_N^{-1}$ is short.
  Our main work in this section shows that $\beta^\dag$ is short (on all of $B$), 
of course using that $\beta$ is a short map.
 The surprising feature of our proof is that we must consider other coalgebras in order to prove the shortness of $\beta^\dag$.
Notice that $(N\otimes B, N\otimes\beta)$
  is also an $(N\otimes-)$-coalgebra.  Furthermore, $\beta\colon B\to N\otimes B$ is a coalgebra morphism.

\begin{lemma}\label{lemma-2N}
$(N\otimes \beta)^\dag  =  \alpha_N \o (N\otimes \beta^\dag)$.
\end{lemma}

\proof
Consider the diagram below in $\SquaSet$. 
   It makes sense because $\alpha_N$ is invertible
   (see Lemma~\ref{alphaNisomorphism}).
\[
\xymatrix@C+1pc{
N\otimes B \ar[d]_-{N\otimes\beta} \ar[r]^-{N\otimes \beta^\dag}  &
N\otimes U_0 \ar[d]_-{N\otimes \alpha_N^{-1}} \ar[r]^-{\alpha_N} \ar@{=}[dr] &
 U_0 \\
N \otimes N\otimes B \ar[r]_-{N\otimes N\otimes\beta^\dag} &
N\otimes (N\otimes U_0) \ar[r]_-{N\otimes\alpha_N} &
 N\otimes U_0 \ar[u]_-{\alpha_N} 
 }
\]
The triangles commute.  The square on the right
(rotated $90^{o}$ and reflected) then
  shows that  $(N\otimes \alpha_N^{-1})^\dag = \alpha_N$.
The square on the left  commutes because when we remove $N$ and turn the arrow
on the right around (from $\alpha_N^{-1}$ to $\alpha_N$), we have the definition of $\beta^\dag$.
That square thus 
shows that $N\otimes \beta^\dag$ is a coalgebra morphism.
Applying  Lemma~\ref{lemma-morphisms} to it, we see that
\[ (N\otimes \beta)^\dag = (N\otimes \alpha_N^{-1})^\dag  \o (N\otimes \beta^\dag) = \alpha_N \o (N\otimes \beta^\dag).
\]
\endproof

\begin{definition}
Let $Z\subseteq B$.  We say that $\beta^\dag$ is \emph{short on $Z$} if for all $b,c\in Z$,
$d_{U_0}(\beta^\dag (b), \beta^\dag (c)) \leq d_B(b,c)$.

Also, we write $N\otimes Z$ for  $\set{n\otimes b : n\in N \mbox{ and } b \in Z}$.
\end{definition}

 \begin{lemma}\label{lemma-manyLines}
 Let $Z\subseteq B$ be any set that includes the image $S_{B}[M_0]$.
If $\beta^\dag$ is short on $Z$, then $(N\otimes \beta)^\dag$ is short on $N\otimes Z$.
\end{lemma}

\proof
Let $b,c\in Z$ and $n_1,n_2\in N$. 
We may assume that $n_1 \neq n_2$, since if $n_1 = n_2$ 
this follows easily from the fact that $\alpha_N$ is a short map and $\beta^\dag$ is short on $Z$.
There are $(r_1,s_1), (r_2,s_2)\in M_0$ such that 
a witness path in $N\otimes B$ from $n_1\otimes b$  to $n_2\otimes c$ contains
$n_1\otimes S_{B}((r_1,s_1))$ and $n_2 \otimes S_{B}((r_2,s_2))$.
We are going to write $S$ for $S_B$ to save on some notation.
We have

\begin{longtable}{clll}
 & $d_{N\otimes B}(n_1\otimes b, n_2\otimes c)$  \\ \\
 = & 
$d(n_1\otimes b, n_1\otimes S((r_1,s_1))) + d(n_1 \otimes S((r_1,s_1)), n_2 \otimes S((r_2,s_2)))$  \\
 & $ +$ $d(n_2 \otimes S((r_2,s_2)), n_2\otimes c)$
&(1) \\ \\
 = & $\onethird d_B(b, S((r_1,s_1))) + d_{N\otimes B}(n_1 \otimes S((r_1,s_1))), n_2 \otimes S((r_2,s_2))$\\
 &  $ +$ $\onethird d_B(S((r_2,s_2)), c)$
 &(2)  \\ \\
  $\geq$ & 
 $\onethird d_{U_0}(\beta^\dag(b), (r_1,s_1))  + d_{N\otimes U_0}(n_1 \otimes (r_1,s_1),  n_2 \otimes (r_2,s_2))$\\
 & $+$ 
 $\onethird d_{U_0}((r_2,s_2), \beta^\dag(c)) $
&(3) 
  \\ \\
 $\geq$ &  $d_{U_0}( \shrink(n_1) + \onethird \beta^\dag(b),  \shrink(n_1) +\onethird (r_1,s_1))$  \\
&  $+$ $d_{U_0}(\shrink(n_1) + \onethird(r_1,s_1),  \shrink(n_2)  +  \onethird(r_2,s_2))$  \\
  & $+$
$d_{U_0}( \shrink(n_2) +\onethird (r_2,s_2), \shrink(n_2) + \onethird \beta^\dag(c))$ 
&(4)
 \\ \\
$\geq$ &  $d_{U_0}(\shrink(n_1) +\onethird \beta^\dag (b), \shrink(n_2) + \onethird\beta^\dag (c))$
&(5)  \\ \\ 
$=$ & $d_{U_0}(\alpha_N((N\otimes\beta^\dag)(n_1,b)), \alpha_N((N\otimes\beta^\dag)(n_2,c))) $
 & (6) \\ \\
$=$  & $d_{U_0}((N\otimes \beta)^\dag(n_1\otimes b), (N\otimes \beta)^\dag(n_2\otimes c)) $  
&(7)\\ 
\end{longtable}

In (1), the distances are in $N\otimes B$.  (1) holds by the choice of $(r_1,s_1)$ and $(r_2,s_2)$ (such that $n_1\otimes S_B((r_1,s_1))$ and $n_2\otimes S_B((r_2,s_2))$ are on a witness path from $n_1\otimes b$ to $n_2\otimes c$).
In (2), we are using Corollary~\ref{quotientmetriccorollary}, the result on distances in a single copy of  $X$ inside of $N\otimes X$.
(3) uses
the assumption that $\beta^\dag$ is short on $Z$,  and the fact that $\beta^\dag \o S_B((r_i,s_i)) = (r_i,s_i)$.
It also uses Lemma~\ref{lemma-crux} in the middle.

(4) uses two facts about distances in $U_0$.  Let $x,y,z\in U_0$.
First, $c \cdot d(x,y) = d(c\cdot x,c\cdot y)$ when $0\leq c \leq 1$.
Second, $d(x,y) = d(x+z,y+z)$, provided $x +z$ and $y+z$ belong to $U_0$.
And in the middle summand of (4), we used the fact that $\alpha_N\colon N\otimes U_0 \to U_0$ is
a short map, and the definition of $\alpha_N$.

(5) uses the triangle inequality in $U_0$.
(6) uses the definition of $\alpha_N$ and $N\otimes \beta^\dag$.
(7) is by Lemma~\ref{lemma-2N}.

This completes the proof.
\endproof

\begin{lemma}\label{lemma-technical}
Let $(B,\beta\colon B\to N\otimes B)$, and let $k\in \omega$.
There is a coalgebra 
\[ (C,\gamma\colon C\to N\otimes C),\]
a coalgebra morphism $g\colon B\to C$, and 
a set $Z\subseteq C$ so that 
\begin{enumerate}
\item $S_C[M_0] \subseteq Z$.
\item
$\gamma^\dag$ is short on $Z$.
\item For every $c_1\in C$ there is some $c_2\in Z$ such that
$d_C(c_1,c_2) \leq \frac{2}{3^k}$, and also
 
 $d_{U_0}(\gamma^{\dag}(c_1),\gamma^\dag(c_2)) \leq \frac{2}{3^k}$.
\end{enumerate}
\end{lemma}

\proof
By induction on $k$.
For $k = 0$, we take $(C,\gamma) = (B,\beta)$,   $g = \id_B$, and $Z = S_B[M_0]$.
Every point in $B$ is at a distance $\leq 2$ from $S_B((0,0))$, and every point in $U_0$ is 
distance at most $2$ from every other point. $\gamma^\dag = \beta^\dag$ is 
short on $Z=M_0$ because of ($\sqtwo$), which requires that distances on the boundary are bounded below by the distances determined by the taxicab metric. 
 
Assume our result for $k$, and fix $(C,\gamma)$, $g$, and $Z$ with the required properties.
The map $\gamma$ is a coalgebra morphism $\gamma\colon C\to N\otimes C$.
Consider $(N\otimes C,N\otimes \gamma)$, $\gamma\o g$ and $N\otimes Z$.

We check that $S_{N\otimes C}[M_0] \subseteq N\otimes Z$.
Let $(r,s)\in M_0$. 
Recall the $\SquaSet$ structure
$S_{N\otimes M_0}\colon M_0\to N\otimes M_0$.
It is a general feature of how
$N\otimes-$ works as a functor that the diagram below commutes:
\[
\xymatrix@C+1.5pc{M_0 \ar[r]^-{S_{N\otimes M_0}}\ar[d]_-{S_C} 
\ar[dr]^-{S_{N\otimes C}}
  & N \otimes M_0 \ar[d]^-{N\otimes S_C}\\
  C \ar[r]_-{\gamma} & N\otimes C
  }
\]
 Write $S_{N\otimes M_0}((r,s))$ as $n\otimes  (r',s')$, where $(r',s')\in M_0$ and $n\in N$.
Then 
\[S_{N\otimes C}((r,s)) = (N\otimes S_C)(n\otimes  (r',s'))  = n\otimes S_C((r',s')) \in
N\otimes S_C[M_0] 
\subseteq N\otimes Z.\]
By Lemma~\ref{lemma-manyLines}, $(N\otimes\gamma)^\dag$ is short on $N\otimes Z$.

Finally, we verify the last point.
Fix a point $n\otimes c_1\in N\otimes C$.  
Let $c_2\in Z$ be such that 
$d(c_1,c_2) \leq \frac{2}{3^k}$, and
$d_{U_0}(\gamma^\dag(c_1),\gamma^\dag(c_2)) \leq \frac{2}{3^k}$. 
   Then $n\otimes c_2\in N\otimes Z$, and 
   \[ d_{N\otimes C}(n\otimes c_1, n\otimes c_2) = 
   \onethird d_C(c_1,c_2) \leq  \frac{2}{3^{k+1}}.\]
   (We are using the same $n$ as chosen at the start of this paragraph.)
Recall that $(N\otimes \gamma)^\dag = \alpha_N \o (N\otimes \gamma^\dag)$ by Lemma~\ref{lemma-2N}.
And 
\[
\begin{array}{cl}
& d_{U_0}((N\otimes \gamma)^\dag(n\otimes c_1),(N\otimes \gamma)^\dag(n\otimes c_2))\\
= & d_{U_0}( \alpha_N \o (N\otimes \gamma^\dag) (n\otimes c_1),  \alpha_N \o (N\otimes \gamma^\dag)(n\otimes c_2)) \\
= & d_{U_0}(\alpha_N(n\otimes \gamma^\dag(c_1)), \alpha_N(n\otimes \gamma^\dag(c_2)))\\
= & d_{U_0}(\shrink(n) + \onethird\gamma^\dag(c_1), \shrink(n) + \onethird\gamma^\dag(c_2))\\
= &  \frac{1}{3}d_{U_0}(\gamma^\dag(c_1),\gamma^\dag(c_2)) \\
\leq & \frac{2}{3^{k+1}}\\
\end{array}
\]
This completes the proof.
\endproof

\begin{lemma} \label{lemma-short}
$\beta^\dag\colon B\to U_0$ is short.
\end{lemma}

\proof Fix $\eps > 0$. 
Let $b_1,b_2\in B$.  
Let $k$ be large enough so that $2/3^k < \eps/4$.
Let $C$, $g$,  $Z$, $c_1$ and $c_2$  be as in Lemma~\ref{lemma-technical}
so that $c_1,c_2\in Z$,  $d_C(g(b_i), c_i)\leq \eps/4$,
and also $d_{U_0}(\gamma^\dag(g(b_i)), \gamma^\dag(c_i)) \leq \eps/4$ for $i = 1, 2$.
Then $d_C(c_1,c_2)\leq d_C(g(b_1),g(b_2)) + \eps/2$.
And
\[
\begin{array}{cll}
 & d_{U_0}(\beta^\dag (b_1), \beta^\dag (b_2)) \\
= & d_{U_0}(\gamma^\dag (g (b_1)), \gamma^\dag (g(b_2))) & (1)\\
\leq & d_{U_0}(\gamma^\dag(g (b_1)), \gamma^\dag (c_1)) + 
  d_{U_0}(\gamma^\dag (c_1), \gamma^\dag (c_2)) +  d_{U_0}(\gamma^\dag (c_2), \gamma^\dag (g (b_2)))\\
 \leq & \eps/4 + d_C(c_1, c_2) + \eps/4 & (2) \\
 \leq & \eps/2 + (d_C(g(b_1), g(b_2)) + \eps/2)  & \\
 \leq & \eps + d_B(b_1, b_2) & (3) \\
  \\ 
 \end{array}
\]
Point (1) uses Lemma~\ref{lemma-morphisms}.
Point (2) uses the shortness of $\gamma^{\dag}$ on $Z$.  Point (3) uses the shortness of $g$.
This for all $\eps >0$ proves our result.
\endproof

\begin{theorem} \label{theorem-Uzero-corec}
$(U_0,\alpha_N)$ is a corecursive algebra for $N\otimes-$ on $\SquaMS$,
and  $(U_0,\alpha_N^{-1})$ is a final coalgebra for this same functor.
\end{theorem}

\proof
We already know that if we forget the metric,  $(U_0,\alpha_N)$ is a corecursive algebra for $N\otimes-$ on $\SquaSet$.
In the case that we have a short coalgebra structure, $(B,\beta)$,  the unique  $\SquaSet$ map $\beta^\dag$ is short,
by Lemma~\ref{lemma-short}.   The forgetful functor $\SquaMS\to\SquaSet$ is faithful, and so 
$\beta^\dag$ is the unique coalgebra-to-algebra map in $\SquaMS$.   This shows the first assertion in our result.
The second follows since $\alpha_N$ is invertible 
(see Lemma~\ref{alphaNisomorphism}).
\endproof

\subsection{$U_0$ is isomorphic to the completion of
the initial algebra for $N\otimes -$}

Recall from (\ref{GW}) that the initial algebra of $N\otimes -$ on $\SquaMS$ is denoted
$(W,\lambda\colon N\otimes W \to W)$.  Recall also that in Definition~\ref{definition-alpha-N}  
we saw
an algebra $\alpha_N\colon  N\otimes U_0\to U_0$.
By initiality
there is a unique $(N\otimes-)$-algebra morphism 
\
\[
\psi\colon W \to U_0
\]
In addition, for the same functor $N\otimes -$,  $\lambda^{-1}$ is a coalgebra and $U_0$ is corecursive,
and therefore $(\lambda^{-1})^\dag = \psi$.
This discussion is in $\SquaMS$, and so $\psi$ is a short map.
Recall also that $W$ is the colimit of the initial sequence
\begin{equation}
M_0 \stackrel[]{! = S_{N\otimes M_0}}{\longrightarrow} 
N\otimes M_0\stackrel[]{N\otimes !}{\longrightarrow}  N^2\otimes M_0\stackrel[]{N^2\otimes !}{\longrightarrow}  N^3\otimes M_0\stackrel[]{N^3\otimes !}{\longrightarrow}  \cdots
 N^k\otimes M_0\stackrel[]{N^k\otimes !}{\longrightarrow}   N^{k+1}\otimes M_0\cdots
\label{eq-chain-N}
\end{equation}
We write $w_k\colon N^k\otimes M_0\to W$ for the colimit injection.   

For all $k$, let $\ell_k\colon N^k\otimes M_0 \to U_0$ be given by $\ell_k = \psi \o w_k$.

Recall the sets $CP_k$ from Definition~\ref{def-CPk},
and also the maps $f_k\colon CP_k\to U_0$, which satisfy the equations
$f_0((r,s)) = (r,s)$, and 
$f_{k+1}(n\otimes x) = \alpha_N(n\otimes f_k(x))$.

\begin{proposition}
\label{proposition-ell-cocone}
\begin{enumerate}
\item The family $(\ell_k)_k$ is a cocone of the inital sequence:
for all $k$, $\ell_k =  \ell_{k+1}\o (N^k\otimes !)$.    

\item 
For all $k$, the diagram below commutes:
\[
\begin{tikzcd}[column sep=huge] 
N^k \otimes M_0 \ar{r}{N^k \otimes !} \ar{d}[swap]{w_k} 
  & N^{k+1} \otimes M_0 \ar{d}{N\otimes w_k} 
\ar{dl}[swap]{w_{k+1}}   \\
W  \ar{r}[swap]{\lambda^{-1} }  &  N\otimes W   \\
\end{tikzcd}
\]

\item $f_k$ is the restriction of the map 
$\ell_k\colon N^k\otimes M_0 \to U_0$
 to $CP_k$.

\end{enumerate}
\end{proposition}

\begin{proof}
\begin{enumerate}
\item This is a consequence of the general fact that if we post-compose all maps in a 
given cocone by the same morphism,
we again have a cocone.

\item
The triangles commute because $W$ is the colimit of the initial-algebra chain and $N\otimes-$ preserves the colimit. 
So the square commutes.


\item We show by induction on $k$ that for $r, s\in \set{0,1}$, and $n_1, \ldots, n_k\in N$, 
\[
\ell_{k}(n_1\otimes\ldots\otimes n_{k}\otimes (r,s)) = 
f_k(n_1\otimes\ldots\otimes n_{k}\otimes (r,s)).
\]

For $k = 0$, $\ell_0((r,s)) = (r,s)=f_0((r,s))$, since $\ell_0= \psi\o w_0 = \psi\o S_{U_0}$ is a morphism in $\SquaMS$
and thus preserves $M_0$.

Assume our result for $k$, and fix $r$, $s$, and $n_1, \ldots, n_k, n_{k+1}\in N$.
To save on notation, write $x$ for $n_2\otimes\ldots\otimes n_{k+1}\otimes (r,s)$.
(In case $k = 1$, $x$ is $(r,s)$.)  This point $x$ belongs to $CP_k$.
Then 
\[\begin{array}{rlr}
&\ell_{k+1}(n_1\otimes x)\\
 = & \psi (w_{k+1}(n_1\otimes x))  &\mbox{by definition of $\ell_{k+1}$}\\  
 = & (\psi\o \lambda)((N\otimes w_k)(n_1 \otimes x)) &\mbox{by part (2),
 $\lambda \o (N\otimes w_k) = w_{k+1}$} \\
  = & (\alpha_N\o  (N\otimes \psi))((N\otimes w_k)(n_1 \otimes x)) &\mbox{$\psi$ is an $(N\otimes -)$-algebra morphism}\\
  = & \alpha_N(n_1 \otimes \psi(w_k(x))) & \mbox{by definition of $N\otimes \psi$ and $N\otimes w_k$}\\
  = & \alpha_N (n_1\otimes\ell_k(x)) &\mbox{by definition of $\ell_{k}$}\\    
    = & \alpha_N (n_1\otimes f_k(x)) & \mbox{by induction hypothesis}\\
  = & f_{k+1}(n\otimes x)  &\mbox{by definition of $f_{k+1}$}\\  
 \end{array}
\]
 \end{enumerate}
 This completes the proof.
 \end{proof}

In the result below and in the sequel, we use the notation  $w_{j,k}$ when $j\leq k$
for
the connecting morphism of the initial-algebra chain (\ref{eq-chain-N}):
\[w_{j,k}\colon N^{j}\otimes M_0 \to N^{k}\otimes M_0.\]
In a more general setting (using different notation) we discussed these
below (\ref{chain}).

\begin{proposition} \label{prop-CP2}
Concerning the maps $w_{j,k}$ when $j\leq k$ and the sets of corner points:
\begin{enumerate}
\item $w_{j,k}[CP_{j}] \subseteq CP_{k}$.
\item The restriction of $w_{j,k}$ to $CP_{j}$ is an isometric embedding.
\end{enumerate}
\end{proposition}

\proof
The first part is an easy induction.

For the second part, let $z$ and $z'$ belong to $CP_{j}$.


\[
\begin{array}{clr}
 & d(z,z') & \\
  = & d_{U_0}(f_j(z),f_j(z'))  & (1)\\
=& d_{U_0}(\ell_j(z),\ell_j(z')) & (2)\\
 = & d_{U_0}(\ell_k\circ w_{j,k}(z),\ell_k\circ w_{j,k}(z')) & (3)\\
 = & d_{U_0}(f_k\circ w_{j,k}(z),f_k\circ w_{j,k}(z')) & (4)\\
  = & d(w_{j,k}(z),w_{j,k}(z')) & (5)\\
 \end{array}
\]

Lines (1) and (5) hold by Proposition~\ref{prop-CP1} applied to both $f_j$ and $f_k$.  (2) and (4) 
hold because $f_j$ is the restriction of $\ell_j$, and (3) is 
proved by an easy induction
on $k\geq j$, using 
 Proposition~\ref{proposition-ell-cocone}. 
\endproof

With these preliminaries done, we now return to the topic of this section.

Recall from (\ref{GW})
that $(W,\lambda \colon N\otimes W\rightarrow W)$ is an initial $N\otimes-$ algebra and that $\lambda$ is an isomorphism.  
With $C$ the Cauchy completion functor on the category,
we have another algebra which we will call
$(V,\theta\colon N\otimes V \to V)$, where $V= CW$ is
a square metric space whose underlying metric is complete, and $\theta$ is an isomorphism.
(The map $\theta$ is $ C\lambda\o \delta^N_W$, where 
$\delta^N_W\colon N\otimes CW \to C(N\otimes W)$ is the isomorphism
which we have seen in Proposition~\ref{cauchycompletion}.)

\begin{lemma}
\label{lemma-HCP}
Let $CP = \bigcup_k w_k[CP_k]$.
\begin{enumerate}
\item For $x,y\in CP$, let $j$ be such that there are $x',y'\in N^j\otimes M_0$ with $w_j(x') =x$ and $w_j(y')=y$.  Then  $d_{N^j\otimes M_0}(x',y') = d_W(x,y)$.  
\item $CP$ is a dense subset of $W$.
\item The restriction of $\psi$ to $CP$ is an isometry.
\item $\psi$ is an isometry.
\item $\psi$ extends to an isomorphism $\overline{\psi}\colon V \to U_0$.
\end{enumerate}
\end{lemma}
\proof
\begin{enumerate}
    \item First note that such a $j$ exists, since if $x\in w_l[CP_l]$ and $y\in w_j[CP_j]$ for some $l\leq j$, then let $\widehat{x}\in CP_l$ be such that $w_l(\widehat{x}) = x$ and let $x' = w_{l,j}(\widehat{x})$.  Then $w_j(x') = x$, as required. 
    
    By Proposition~\ref{prop-CP2}(2), for any $k\geq j$, $d_{N^k\otimes M_0}(w_{j,k}(x'),w_{j,k}(y')) = d_{N^j\otimes M_0}(x',y')$.  
    
    So $d_W(x,y) = \displaystyle{\inf_{k\geq j}}\  d_{N^k\otimes M_0}(w_{j,k}(x'),w_{j,k}(y')) = d_{N^j\otimes M_0}(x',y')$.  
    
    \item 
    Let $\epsilon>0$ be given and choose $K$ such that $\frac{2}{3^K}<\epsilon$.  Let $x\in W$ and let $k\geq K$ be such that there is $x'\in N^k\otimes M_0$ with $w_k(x') =x$.  Then there are $n_1,\ldots,n_k\in N$ and $(r,s)\in M_0$ such that $x' = n_1\otimes\ldots\otimes n_k\otimes (r,s)$.  Let $c = n_1\otimes\ldots\otimes n_k\otimes (0,0)\in CP_k$, and note that $w_k(c)\in CP$.    Then $d_{N^k\otimes M_0}(x',c) \leq \frac{2}{3^k}<\epsilon$ by Corollary~\ref{distanceinonecopyN}, so since $d_W$ is the infimum of the distances in $N^k\otimes M_0$, $d_W(x,w_k(c)) \leq d_{N^k\otimes M_0}(x', c) <\epsilon$.    Hence, $CP$ is dense in $W$. 
    
    \item Let $x,y\in CP$ and let $k$ be such that there are $x',y'\in CP_k$ with $w_k(x')=x$ and $w_k(y')=y$.  Note that $\psi(x) = \psi\circ w_k(x') = \ell_k(x') = f_k(x')$ and similarly, $\psi(y) = f_k(y')$. 
    Then
    \[\begin{array}{rcll}
    d_{U_0}(\psi(x),\psi(y)) & = & d_{U_0}(f_k(x'),f_k(y')) & \\
    & = & d_{N^k\otimes M_0}(x',y') & \mbox{by Proposition~\ref{prop-CP1}}\\
    & = & d_{W}(x,y) & \mbox{by 1.} \\
    \end{array}
    \]
    
    \item This follows from parts 2 and 3.
    \item For this it will be enough to show that the image of $CP$ is dense in $U_0$.  Let $(x,y)\in U_0$ be given.  It is a standard fact that every real number has a ternary representation; see also~\cite[Example 7.3.10(2)]{AMM} for a corecursive algebra proof of the related fact that real numbers have binary representations. We can choose $(i_k,j_k)$ in $N$ such that 
    $(x,y) = \biggl(\displaystyle{\sum_{k=0}^\infty} \frac{i_k}{3^{k+1}},\displaystyle{\sum_{k=0}^\infty} \frac{j_k}{3^{k+1}}\biggr)$.    
    For $\epsilon>0$, choose $K$ such that $\frac{2}{3^K}<\epsilon$.  Let 
    \[ c=  (i_0,j_0)\otimes\ldots\otimes (i_{K-1},j_{K-1})\otimes (0,0)\in CP_k,\] and note that $f_k(c) = \ell_k(c) = \psi(w_k(c))=\overline{\psi}(w_k(c))$.  So since $w_k(c) \in CP$, this is in the image of $CP$.  Then 
    \[ d_{U_0}(f_K(c),(x,y)) = \biggl|\displaystyle{\sum_{k=K+1}^\infty} \frac{i_k}{3^{k+1}}\biggr| + \biggl|\displaystyle{\sum_{k=K+1}^\infty}\frac{j_k}{3^{k+1}}\biggr|\leq \frac{2}{3^K}<\epsilon.\]  
    
    Thus, the image of $CP$ under $\psi$ is dense in $U_0$, as required. 
\end{enumerate}
\endproof

\begin{theorem} 
\label{theorem-final-N}
$(V, \theta\colon N\otimes V\to V)$ is a corecursive $(N\otimes -)$-algebra,
and therefore
$(V,\theta^{-1}\colon V\to N\otimes V)$ is a final $(N\otimes-)$-coalgebra.
\end{theorem}

\proof
Let $(B,\beta\colon B \to N\otimes B)$ be a coalgebra.
Consider the metric space
$V^B$, and note that since $V$ is complete, $V^B$ is also complete. 
The subspace of $V^B$ of short maps which preserve the square space structure is a closed subset since 
limits of structure-preserving short maps will be short and will preserve the structure.
Crucially, 
the set of such maps is non-empty.  This is because we have $\SquaMS$ morphism $\beta^\dag\colon B\to U_0$ by Lemma~\ref{lemma-short}
(this map has nothing to do with $\beta$ in this proof)
 and an isomorphism
 $\overline{\psi}^{-1}:U_0 \to V$ by  Lemma~\ref{lemma-HCP}.
We also have a $\onethird$-contracting map $\Phi\colon V^B \to V^B$ given by 
$\Phi(f) = \theta\o (N\otimes f)\otimes 
\beta$.

Thus, $\Phi$ has a unique fixed point. 
The fixed points of $\Phi$ are exactly the 
coalgebra to algebra morphisms $B \to V$.
Thus, there is a unique 
such morphism from $B \to V$. 
This proves that $(V,\theta)$ is a corecursive 
algebra.   Since $\theta$ is invertible, $(V,\theta^{-1})$ is a
 final coalgebra; see~Proposition~\ref{prop-invertible}.
\endproof

\begin{corollary} $V$, the Cauchy completion of the initial $(N\otimes-)$-algebra, is isomorphic to $U_0$ with the taxicab metric.

\end{corollary}

\proof
Since $V$ and $U_0$ are both final $(N\otimes -)$-coalgebras (Theorem~\ref{theorem-Uzero-corec} and Theorem~\ref{theorem-final-N}), they are isomorphic. 
\endproof

\subsection{The Sierpinski carpet is a corecursive $M\otimes-$ algebra}

For our next result on this topic, recall that we have an isometry 
$\alpha_N\colon N\otimes U_0 \to U_0$
(see 
Lemma~\ref{alphaNisomorphism}).

Let $\tau$ be the restriction of $\alpha_N$ to $M\otimes \SC$.
Recall the maps $\sigma_m$ from Definition~\ref{def-sigmas}, and also $\sigma$.
Then note that for $m\otimes s\in M\otimes \SC$,
\begin{equation}\label{eq-tau}
\tau(m\otimes s) = \shrink(m)+\frac{1}{3} s = \sigma_m(s) \in \SC. 
\end{equation}
And for $s\in \SC = \sigma(\SC) = \displaystyle{\bigcup_{m\in M}}\sigma_m(\SC)$, there are $s'\in \SC$ and $m\in M$ such that $s=\sigma_m(s') = \tau(m\otimes s')$.  So $\tau:M\otimes \SC\rightarrow \SC$ is a bijection.  

This map $\tau$ is not an isometry,
so it has no inverse in $\SquaMS$, but it still is an isomorphism in $\SquaSet$.

\begin{proposition} \label{prop-tau}
The diagram below commutes in $\SquaMS$:
\[
\begin{tikzcd}[column sep=huge] 
M \otimes \SC \ar{r}{\tau} \ar{d}[swap]{M\otimes i}& \SC \ar{d}{i} \\
M\otimes U_0   \ar{r}[swap]{\alpha_M =\alpha_N\o \iota_{U_0} } &  U_0 
\end{tikzcd}
\]
Here
$i$ is the inclusion, and the natural transformation
 $\iota$ is from Proposition~\ref{prop-eta-M-N}.
\end{proposition}
\begin{proof}
 Let $m\otimes x\in M\otimes \SC$ be given.  Then $i\circ\tau(m\otimes x) = i(\alpha_N(m\otimes x)) = \alpha_N(m\otimes x)$
 and $\alpha_M\circ M\otimes i(m\otimes x) = \alpha_M(m\otimes x) = \alpha_N\circ\iota_{U_0}(m\otimes x) = \alpha_N(m\otimes x)$.
\end{proof}

Let $(B,\beta\colon B \to M\otimes B)$ be a coalgebra.
By postcomposing with the inclusion $\iota_B:M\otimes B \to N\otimes B$, we get an $N\otimes-$ coalgebra $\iota_B\circ \beta:B\rightarrow N\otimes B$.   So we have 
$(\iota_B\circ\beta)^\dag: B\to U_0$. 
We aim to show that for all $b\in B$, $(\iota_B\circ\beta)^\dag(b) \in \SC$.
Before presenting the proof, we will walk the reader through the ideas.
We will assume that $B$ is enumerated without repeats as $b_1, b_2, \ldots, b_k, \ldots$,
and also that our coalgebra $\beta$ 
is given by 
\[
\beta(b_i) = m_i\otimes b_{i+1}
\]
(Please note that we are not saying that all
coalgebras look like this; we are only making an example.  In fact,
a general coalgebra for this functor would be a family of an arbitrary set
of 
disjoint versions of this example, together with an arbitrary set of 
finite coalgebras; these would be eventually periodic.  None of this
really matters in this paper.)
The $m_i$ can be chosen in $M$, not just in $N$.
Then the solution $(\iota_B\circ \beta)^\dag\colon B \to U_0$
corresponds to elements $r_1, r_2, \ldots$ in $U_0$ such that
\[
\begin{array}{lcl}
r_1 & = & \alpha_N(m_1 \otimes r_2) \\
r_2 & = & \alpha_N(m_2 \otimes r_3) \\
r_3 & = & \alpha_N(m_3 \otimes r_4) \\
 & \vdots & \\
 \end{array}
\]
Again, we would like to show that each $r_i$ belongs to $\SC$.  It is clear that 
\[ r_1 \in \alpha_N( m_1\otimes U_0 ) = \sigma_{m_1}(U_0) \]
The notation $m_1\otimes U_0$ and similar notation below is from
Remark~\ref{remark-tensoring}.
A little more thought shows that
\[ r_1 \in \alpha_N(m_1\otimes  \alpha_N(m_2 \otimes U_0) )= \sigma_{m_1} (\sigma_{m_2}(U_0) ) \]
and then
\[ r_1 \in
\alpha_N(m_1\otimes 
 \alpha_N(m_2\otimes  \alpha_N(m_3 \otimes U_0) ))
 = \sigma_{m_1} (\sigma_{m_2}(\sigma_{m_3}(U_0)) ) \]
In the notation of Hutchinson's Theorem 
(Proposition~\ref{prop-Hutchinson}),
 $r_1 \in (U_0)_{m_1 m_2 m_3 \ldots m_p}$ for all $p$.  
 Since all of the $m$'s belong to $M$,  Proposition~\ref{prop-Hutchinson} parts (2) and (3)
 tell us that
 $r_1\in \SC_{m_1m_2\ldots}\subset \SC$.   Similarly, we can argue for each $i$, $r_i = (\iota_B\circ\beta)^\dag(b_i) \in \SC_{m_im_{i+1}\ldots}\subset \SC$. 
 
 Most of the  work in 
 the proof of our next result
 is in managing the notation (and changing it a little)
  and then filling  in the details in the sketch above.

\begin{proposition}\label{prop-factors}
For all $b\in B$, $(\iota_B\circ \beta)^\dag(b)\in \SC$.
\end{proposition}

\proof

As in the proof of Lemma~\ref{lemma-U0-N-otimes}, fix an associate $\widehat{\beta}\colon B \to M\times B$.
Define maps $u_k\colon B \to B$ for $k\geq 0$
and $v_k\colon B \to M$ for $k \geq 1$:
\[
\begin{array}{lcl}
u_0(b) & = & b\\
\widehat{\beta}(u_k(b)) & = & (v_{k+1}(b), u_{k+1}(b))\\
\end{array}
\]
We claim that for all $k\geq 0$:  $u_k(u_1(b)) = u_{k+1}(b)$.
The proof is by induction on $k$.  For $k = 0$, our result is clear.
Assume that $u_k(u_1(b)) = u_{k+1}(b)$.
Then 
\begin{equation}\label{eqclaim}
 (v_{k+1}(u_1(b)), u_{k+1}(u_1(b)))
= \widehat{\beta}(u_k(u_1(b))) =
\widehat{\beta}(u_{k+1}(b)) =(v_{k+2}(b), u_{k+2}(b)).
\end{equation}
So  $u_{k+1}(u_1(b)) = u_{k+2}(b)$.
This establishes our claim.
And from this claim we  
repeat (\ref{eqclaim}) to see that
for all $k\geq 1$,
 $v_k(u_1(b)) = v_{k+1}(b)$.

For each $b\in B$, we have an infinite sequence of elements of $M$
\begin{equation}
    \label{Mseq}
 v_1(b), v_2(b), \ldots, v_k(b), \ldots
\end{equation}
Moreover, we will show by induction on $k$  that 
\begin{equation}
 \label{vsbs}   
 (\iota_B\circ \beta)^\dag(b) \in (U_0)_{v_1(b),v_2(b), \ldots, v_k(b)}
 \end{equation}
for all $b\in B$.
For $k = 0$, $(\iota_B\circ\beta)^\dag(b)\in U_0 = (U_0)_{\eps}$.
Fix $k \geq 0$, and assume that for all $b\in B$, $(\iota_B\circ\beta)^\dag(b) \in  (U_0)_{v_1(b), \ldots, v_k(b)}$.
Now fix $b$.
So $\widehat{\beta}(b) = (v_1(b),u_1(b))$.
To save on notation,
we will write $b'$ for $u_1(b)$.
By our assumption, 
\[
(\iota_B\circ \beta)^\dag(b') \in  (U_0)_{v_1(b'), \ldots, v_k(b')} = 
(U_0)_{v_2(b), \ldots, v_{k+1}(b)}. 
\]
(Notice that we used a fact from above to write $v_i(b') = v_i(u_1(b)) = v_{i+1}(b)$.)
And then 
\[
(\iota_B\circ \beta)^\dag(b)= 
\alpha_N(v_1(b) \otimes (\iota_B\circ\beta)^\dag(b'))\in
\sigma_{v_1(b)}((U_0)_{v_2(b), \ldots, v_{k+1}(b)})
= (U_0)_{v_1(b),v_2(b), \ldots, v_{k+1}(b)}.
\]

This completes the induction.
Since the 
sequence in (\ref{Mseq})
comes from $M$, by 
(\ref{vsbs}) and Proposition~\ref{prop-Hutchinson}, we get that $(\iota_B\circ \beta)^\dag(b)\in \SC$.
\endproof

As a result of Proposition~\ref{prop-factors}, we regard $(\iota_B\o \beta)^\dag$ as a morphism with codomain $\SC$.
That is, $(\iota_B\circ \beta)^\dag\colon B \to U_0$ factors through the inclusion  $i\colon \SC \to U_0$.
So we have a map
 $\beta^*\colon B\to \SC$
 such that 
\begin{equation}\label{betastar}
(\iota_B\o \beta)^\dag = i \o \beta^*.
\end{equation}

\begin{theorem}
\label{theorem-SC}
$(\SC,\tau)$ is a corecursive algebra for $M\otimes-\colon\SquaMS\to\SquaMS$.
\end{theorem}

\proof
Let $(B,\beta)$ be a coalgebra.   Consider the following diagram in $\SquaMS$:
\[  
   \begin{tikzcd}
      B \arrow{r}{\beta} 
      \ar[bend right=40,swap]{ddd}{(\iota_B \o \beta)^\dag}
      \arrow{d}{\beta^*}
      &
M\otimes B
  \ar[bend left=50]{dd}{M\otimes(\iota_B \o \beta)^\dag}
        \arrow[d,swap, "M\otimes \beta^*"]
      \arrow{r}{\iota_B}
      &
N\otimes B \ar{ddd}{N\otimes  (\iota_B \o \beta)^\dag}
      \\
  \SC
      \arrow[>->]{dd}[swap]{i}
      &
      M\otimes \SC       \arrow{l}{\tau}
      \ar{d}[swap]{M\otimes i}
      & &
  \\
   & M\otimes U_0  \ar{dr}{\iota_{U_0}}  \ar{dl}{\alpha_N\o \iota_{U_0}}& \\
U_0 &   & \arrow{ll}{\alpha_N} N\otimes U_0  
    \end{tikzcd}
\]

We need to show that the top left corner commutes. 
We are using the natural transformation $\iota: (M\otimes -) \to (N\otimes -)$  from Proposition~\ref{prop-eta-M-N}.
We get $ (\iota_B \o \beta)^\dag$ by Theorem~\ref{theorem-Uzero-corec}, and the outside of the diagram commutes.
We have seen in (\ref{betastar}) that the small region in the center commutes.

The region on the far right commutes
by the naturality of $\iota$.

The region in the lower-left commutes by Proposition~\ref{prop-tau}.
The bottom commutes trivially.  
Thus, all of the inside parts commute.  
A diagram chase shows that $i \o \beta^* = i \o \tau \o (M\o \beta^*) \o \beta$.
Since $i$  is monic, $\beta^* =   \tau \o (M\o \beta^*) \o \beta$.    This shows that $\beta^*$ is a coalgebra-to-algebra map.

For the uniqueness of $\beta^*$, suppose that $\beta^{**}\colon B\to \SC$ satisfies
$\beta^{**} =   \tau \o (M\o \beta^{**}) \o \beta$.
Consider the diagram below:
\[  
   \begin{tikzcd}
      B \arrow{r}{\beta} 
  \ar[bend right=40,swap]{ddd}{i\o \beta^{**}}
      \arrow{d}{\beta^{**}}
      &
M\otimes B
  \ar[bend left=50]{dd}{M\otimes(i\o \beta^{**})}
        \arrow[d,swap, "M\otimes \beta^{**}"]
      \arrow{r}{\iota_B}
      &
N\otimes B \ar{ddd}{N\otimes (i\o \beta^{**})}
      \\
  \SC
      \arrow[>->]{dd}[swap]{i}
      &
      M\otimes \SC       \arrow{l}{\tau}
      \ar{d}[swap]{M\otimes i}
      & &
  \\
   & M\otimes U_0  \ar{dr}{\iota_{U_0}}  \ar{dl}{\alpha_N\o \iota_{U_0}}& \\
U_0 &   & \arrow{ll}{\alpha_N} N\otimes U_0  
    \end{tikzcd}
\]
At first glance, the maps are different from those in the previous diagram.
All of the inside parts of this diagram commute: the part on the left by definition,
the part on the right by naturality, and the remaining parts for the same reasons
as in the previous diagram.   Thus, the ouside commutes.  This implies that 
$i\o \beta^{**}$ is a coalgebra-to-algebra morphism
for $\iota_B\o \beta$.  By the uniqueness part of Theorem~\ref{theorem-Uzero-corec}, $i\o \beta^{**} =  (\iota_B \o \beta)^\dag = i \o \beta^*$.
Since $i$ is monic, $\beta^{**} = \beta^*$.
\endproof

Unfortunately $\tau^{-1}$ is not a short map, so it is not a morphism in $\SquaMS$.  However, it is an isomorphism in $\SquaSet$, so we do get the following. 

\begin{corollary}
\label{corollary-SC}
$(\SC,\tau^{-1})$ is a final coalgebra for $M\otimes-\colon\SquaSet\to\SquaSet$.
\end{corollary}

\proof
First, let us show that 
$(\SC,\tau)$ is a corecursive algebra for $M\otimes-\colon\SquaSet\to\SquaSet$.
Let $(B,\beta)$ be a coalgebra.  Endow $B$ with the following metric:  For $(r,s),(t,u)\in M_0$, let 
\[ d_B(S_B((r,s)),S_B((t,u))) = d_{U_0}((r,s),(t,u)),\]
and for  $x,y\notin S_B[M_0]$, let $d_B(x,y) = 2$ and $d_B(x,S_B((r,s))) =2$.  It is easy to verify that this is an object in $\SquaMS$. 

Then $\beta$ is automatically short.   By Theorem~\ref{theorem-SC}, there is a unique solution $\beta^\dag$.
This same morphism is a solution in $\SquaSet$, of course.   For the uniqueness, note that every morphism from 
the discrete space $B$ to $\SC$ is automatically short.

The morphism $\tau$ is a bijection, and so it is invertible in $\SquaSet$.  
So we are done by Proposition~\ref{prop-invertible}.
\endproof

\subsection{The final $(M\otimes-)$-coalgebra $(Q,\gamma:Q\rightarrow M\otimes Q)$}


Recall $(G,\eta:M\otimes G\rightarrow G)$, the initial algebra.  
By Lambek's Lemma, 
 $\eta$ is an isomorphism.  Let $Q= CG$, the Cauchy completion. 
 Consider the map below:
 \begin{equation}\label{eq-gamma}
\gamma:Q\rightarrow M\otimes Q
 \qquad\qquad  \xymatrix{Q =CG \ar[r]^-{C\eta^{-1}} & C(M\otimes G) \ar[r]^-{\rho_{G}}
 & M\otimes CG
 = M\otimes Q}.
 \end{equation}
 The morphism $\rho^M_G$ is the isomorphism from 
 Proposition~\ref{natisom}.
In this section we will show that $(Q,\gamma)$ is the final $(M\otimes-)$-coalgebra.

For the remainder of the paper, let \begin{equation}\label{GinQ} i:G\hookrightarrow Q\end{equation}
denote the inclusion map from $G$ into $Q$, and note that $G$ is dense in $Q$.

Let $(B,\beta\colon B\to M\otimes B)$ be a coalgebra.
The main task at this point is to exhibit a short map   $h\colon B\rightarrow Q$.  We will use the short map 
$(\iota_B\circ \beta)^\dag\colon B \to U_0$ in our definition, but our use will not be what
one might at first expect.
Instead, to get $h$ we will need to \emph{go via $M^n\otimes U_0$}
(in some appropriate sense that we shall discuss). 
Even if we wanted to use $(\iota_B\circ\beta)^\dag$ directly, there is an issue which arises in considering a map from $\SC$ (as a subset of $U_0$) to $Q$:
the most natural and direct map
will not be short.  For example, consider points $(\frac{1}{2},\frac{1}{3})$ and $(\frac{1}{2},\frac{2}{3})$ in $U_0$.  These have distance $\frac{1}{3}$ in the taxicab metric.   However, these correspond to the top and bottom of the
``hole'' at $(1,1)$ in $Q$, that is, if we view $Q$ as $M\otimes Q$, the points 
are $(1,0)\otimes S_Q((\frac{1}{2},1))$ and $(1,2)\otimes S_Q((\frac{1}{2},0))$, so their distance under the quotient metric will be $\frac{2}{3}$ (to navigate around the hole).  So the obvious bijective correspondence between a subset of $U_0$ and $Q$ will not be a short map, and indeed, not an isometry.  
 However, we navigate around this difficulty, going a different way. 
We will consider corner points as we did for $N\otimes-$, but note that the density of corner points in the relevant subset of $U_0$ is not going to help us: again, 
the map from the appropriate subset of $U_0$ to $Q$ is not a short map.

\paragraph{Corner Points for $(M\otimes-)$}

We will start by adapting the definition of corner points for the $(N\otimes-)$ functor.

\begin{definition}
The set $CP^M_k$ of  \emph{corner points} of $M^k\otimes M_0$ is defined as follows:
\[
\begin{array}{lcl}
CP^M_0 & = & \set{(0,0), (0,1), (1,0), (1,1)} \\
CP^M_{k+1} & = & \{m\otimes x \mid m\in M, x\in CP^M_k\}\\
\end{array}
\]

\end{definition}

We can also refer to corner points in $M^k\otimes U_0$ via the inclusion $M^k\otimes S_{U_0}(CP^M_k)$, and this is a bijective correspondence.  Right away we see that the distance between corner points in $CP^M_k$ (as a subset of $M^k\otimes M_0$) is bounded below by the distance between their images in $M^k\otimes U_0$,
because $M^k\otimes S_{U_0}$ is a short map.

In the next lemma and corollary, we will prove that $M^k\otimes S_{U_0}$ restricted to $CP^M_k$ is in fact an isometry. 

\begin{lemma}\label{CPgeodesic}
Let $x$ and $y$ be corner points in $M^k\otimes U_0$. 
Then there exists a witness path from $x$ to $y$ consisting entirely of 
 corner points in $M^k\otimes U_0$.  
\label{lemma-geodesic}
\end{lemma}

\proof
The idea is to take any path  $p$ from $x$ to $y$ and to modify $p$, obtaining a path $p'$ from $x$ to $y$
with a score
 at most that of $p$ and with at least one fewer node which is not a corner point.
 (The score of a path was defined near the beginning of Section~\ref{section-L}.)
So in effect we are arguing by induction on the number of non-corner-points that the score can drop
by replacing such a point by a corner point, and perhaps making further modifications.

Our path may be written as a path in the $8^k$ copies of $U_0$.
That is, a 
witness path (see Definition~\ref{definition-witness-path})
in $M^k\otimes U_0$ is most naturally presented as a path of ``segments'',
each from $M^{k-1}\otimes U_0$.   But this is not the way we want to view it here. 
We want to say that our path is a path of length $\leq 8^k$ in copies of $U_0$ with the taxicab metric.  We know that $8^k$ is an upper bound on the number of segments in our path, since if a copy of $U_0$ is visited twice, then by Corollary~\ref{distanceinonecopyM} we could find a smaller score by removing the cycle.  

The first thing to do is to modify $p$ on behalf of all edges which connect two non-corner points.
In the picture on the left below is a suggestive example.
We are going to work with this rather than the general case.
The edges that connect two non-corner points are  the ones shown, except for the first and last.
\[
\begin{tikzpicture}[scale=1]
\fill[white!10!white,  draw=black] (0,0) rectangle (6,1);
\draw (1,0) -- (1,1);
\draw (2,0) -- (2,1);
\draw (3,0) -- (3,1);
\draw (4,0) -- (4,1);
\draw (5,0) -- (5,1);
\draw [line width=2pt]  (0,0) --  (1,.7);
\draw [line width=2pt]  (1,.7) -- (2,.3);
\draw [line width=2pt]  (2,.3) --  (3, .9);
\draw [line width=2pt]  (3,.9) --  (4,.1);
\draw [line width=2pt]  (4,.1) --  (5,.3);
\draw [line width=2pt]  (5,.3) --  (6,1);
\draw (0,1.4) node {$a$};
\draw (1,1.4) node {$b$};
\draw (2,1.4) node {$c$};
\draw (3,1.4) node {$d$};
\draw (4,1.4) node {$e$};
\draw (5,1.4) node {$f$};
\draw (6,1.4) node {$g$};
\end{tikzpicture}
\qquad
\begin{tikzpicture}[scale=1]
\fill[white!10!white,  draw=black] (0,0) rectangle (6,1);
\draw (1,0) -- (1,1);
\draw (2,0) -- (2,1);
\draw (3,0) -- (3,1);
\draw (4,0) -- (4,1);
\draw (5,0) -- (5,1);
\draw  (6,0) --  (6,1);
\draw [line width=2pt]  (0,0) --  (6,0);
\draw [line width=2pt]  (6,0) --  (6,1);
\draw [line width=.5pt]  (0,0) --  (1,.7);
\draw [line width=.5pt]  (1,.7) -- (2,.3);
\draw [line width=.5pt]  (2,.3) --  (3, .9);
\draw [line width=.5pt]  (3,.9) --  (4,.1);
\draw [line width=.5pt]  (4,.1) --  (5,.3);
\draw [line width=.5pt]  (5,.3) --  (6,1);
\end{tikzpicture}
\]
Every edge which connects two non-corner points is part of a maximal sub-path $q$ of such edges.
This is because the first and last points on $p$ are corner points, and $p$ itself is finite.
Then we replace the sub-path $q$ as on the right above.
It is important to note that making this replacement still gives us a path in $M^k\otimes U_0$. 
(That is, we do not step out of $M^k\otimes U_0$ into $N^k\otimes U_0$ by making it.  This is because we remain within the copies of $U_0$ used in the original path, so none of our new segments will fall in one of the ``holes'' determined by $M$.) And different maximal sub-paths may be replaced simultaneously.
We check that the bold path on the left represents a longer subpath than the one on the right.
Let the coordinates of $a$ be $(x_a, y_a)$, and similarly for $b$, $c$, $\ldots$, $g$.
Then the length of the path on the left is 
\[
\begin{array}{clr}
& |x_b - x_a| + |y_b - y_a| + \cdots + |x_g - x_f| + |y_g - y_f| \\
\geq & 6 + |y_b - y_a| + |y_d - y_c| +  |y_f - y_e| + |y_g - y_f|\\
\geq & 6 + 1 = 7 
\end{array}
\]

The idea is that each $|x_b-x_a|$ is at least $1$ since they are on opposite sides of a copy of $U_0$, so these will cumulatively contribute at least $6$ to the score.  Similarly, in order to transit from $y_a$ to $y_f$, we must contribute at least $1$ to the score, since they are on opposite sides (of a row of adjacent copies) of $U_0$.

The length of the bold path on the right is $7$.
The same argument would work for a sub-path which was like this but rotated $90^{\circ}$.
There is a second kind of replacement which is similar to what we just saw but where the sub-path's two endpoints have the same $y$-coordinate.   This second kind is easier to handle,
since a sequence of horizontal segments works.

After these two kinds of replacements our path $p$ might contain non-corner points, but edges which contain non-corner points also
contain a corner point.   These edges come in pairs of three possible forms:
  \[
  \begin{tikzpicture}[scale=1]
\fill[white!10!white,  draw=black] (0,0) rectangle (1,1);
\draw [line width=2pt]  (0,1) --  (1,.6);
\draw [line width=2pt]  (1,.6) --  (0,0);
\end{tikzpicture}
\qquad
\begin{tikzpicture}[scale=1]
\fill[white!10!white,  draw=black] (0,0) rectangle (2,1);
\draw (1,0) -- (1,1);
\draw [line width=2pt]  (0,1) --  (1,.3);
\draw [line width=2pt]  (1,.3) --  (2,0);
\end{tikzpicture}
\qquad
\begin{tikzpicture}[scale=1]
\fill[white!10!white,  draw=black] (0,0) rectangle (2,1);
\draw (1,0) -- (1,1);
\draw [line width=2pt]  (0,1) --  (1,.3);
\draw [line width=2pt]  (1,.3) --  (2,1);
\end{tikzpicture}
\qquad
\]
Then each of these sub-paths may be replaced by one using only corner points, with the overall score
not increasing, as shown below:
  \[
  \begin{tikzpicture}[scale=1]
\fill[white!10!white,  draw=black] (0,0) rectangle (1,1);
\draw [line width=2pt] (0,1) -- (0,0);
\draw [line width=.5pt]  (0,1) --  (1,.6);
\draw [line width=.5pt]  (1,.6) --  (0,0);
\end{tikzpicture}
\qquad
\begin{tikzpicture}[scale=1]
\fill[white!10!white,  draw=black] (0,0) rectangle (2,1);
\draw (1,0) -- (1,1);
\draw [line width=2pt] (0,1) -- (1,1);
\draw [line width=2pt]  (1,1) -- (1,0);
\draw [line width=2pt]  (1,0) -- (2,0);
\draw [line width=.5pt]  (0,1) --  (1,.3);
\draw [line width=.5pt]  (1,.3) --  (2,0);
\end{tikzpicture}
\qquad
\begin{tikzpicture}[scale=1]
\fill[white!10!white,  draw=black] (0,0) rectangle (2,1);
\draw (1,0) -- (1,1);
\draw [line width=2pt] (0,1) -- (1,1);
\draw [line width=2pt]  (1,1) -- (2,1);
\draw [line width=.5pt]  (0,1) --  (1,.3);
\draw [line width=.5pt]  (1,.3) --  (2,1);
\end{tikzpicture}
\qquad
\]
In each case, it is clear that the new sub-path has a length at most that of the old;
this is most interesting in the middle case, where we use the fact that the metric in $U_0$
is the taxicab metric.   

In this way, we have taken a path $p$ in  $M^k\otimes U_0$ between corner points
and modified it to a path between the same points in  $M^k\otimes M_0$ without increasing the length. 
\endproof

Throughout the remainder of this section, we will adopt the following notation: for $\overline{m}\in M^k$ and $x\in X$, $\overline{m}\otimes x$ is $m_1\otimes\ldots\otimes m_k\otimes x\in M^k\otimes X$, where $\overline{m} = (m_1,\ldots,m_k)$.  

\begin{corollary}\label{cornersM0toU0}Let
$r,s,t,u\in \{0,1\}$.
Then for $\overline{m},\overline{n}\in M^k$, 
\[ d_{M^k\otimes U_0}(\overline{m}\otimes S_{U_0}((r,s)), \overline{n}\otimes S_{U_0}((t,u))) = d_{M^k\otimes M_0}(\overline{m}\otimes S_{M_0}((r,s)),\overline{n}\otimes S_{M_0}((t,u))).\]  
That is, the distance between corners in $M^k\otimes U_0$ coincides with the distance in $M^k\otimes M_0$.  
\end{corollary}
\proof
By the previous lemma, there is a witness path in $M^k\otimes U_0$ such that every entry is a corner. So for each pair contributing positively to the score,
 if they are adjacent corners, they contributes $(\frac{1}{3})^k$, and if they are opposite corners, they contribute $(\frac{2}{3})^k$ to the score.

So consider the corresponding path in $M^k\otimes M_0$.  This score will be the same.  Thus, the distance in $M^k\otimes U_0$ is bounded above the distance of the corresponding points in $M^k\otimes M_0$.  However, we know that the distance in $M^k\otimes M_0$ is bounded above by its image in $M^k\otimes U_0$ under $M^k\otimes S_{U_0}$, since this is a short map.
Thus, these distances are equal.  
\endproof

\paragraph{The map \text{\boldmath{$h$}}}

Let $(B,\beta\colon B\to M\otimes B)$ be a coalgebra.
Our final task is to find a morphism $h:B\rightarrow Q$ in $\SquaMS$:
once we know that the set of morphisms from $B$ to $Q$ is non-empty, we can use a fixed-point argument like the one we saw in Theorem~\ref{theorem-final-N} to show that $(Q,\gamma:Q\rightarrow M\otimes Q)$ is the final $M\otimes-$coalgebra in $\SquaMS$.  

We will start by defining functions $h_k:B\rightarrow M^k\otimes M_0$ which are \emph{not short maps}, but are approximately short in some technical sense described below.  We will need our work on corner points and the short map $M^k\otimes (\iota_B\circ \beta)^\dag:M^k\otimes B\rightarrow M^k\otimes U_0$ to show that the $h_k$ maps satisfy our approximate shortness property.  Then for a fixed $x\in B$, this gives a sequence $[h_k(x)]_k$ in $G$, which we will show is a Cauchy sequence, and thus, has a limit in $Q$.  This limit is what $h$ will map $x$ to.  Furthermore, we will show that $h$ preserves $S_B$, and thus, is a $\SquaMS$ morphism.

For $x\in B$, define infinite sequences $m_1(x),m_2(x),\ldots\in M$ and $b_0(x),b_1(x),\ldots\in B$ as follows:  let $b_0(x) =x$, and for $k\geq 1$, given $b_0(x),\ldots,b_{k-1}(x)$ and $m_1(x),\ldots, m_{k-1}(x)$, \emph{choose} $m_k(x)\in M$ and $b_k(x)\in B$ such that 
\begin{equation}
    \label{eq-bk}
\beta(b_{k-1}(x)) = m_{k}(x) \otimes b_k(x).
\end{equation}
Note that there may be more than one choice for $m_k(x)$ and $b_k(x)$.
The point is that we are fixing a particular selection.  

Here is how our notation works:
 \[  
   \begin{tikzcd}
B \arrow{r}{\beta} & M\otimes B \arrow{r}{M\otimes\beta} 
& M^2\otimes B \arrow{r}{M^{k-1}\otimes\beta}\cdots & M^k\otimes B \arrow{r}{M^k\otimes \beta}& \cdots\\
x & m_1(x) \otimes b_1(x) &  m_1(x) \otimes m_2(x) \otimes b_2(x) 
 & \mbar(x) \otimes b_k(x)
 \end{tikzcd}
 \]

For a given $x\in B$, we have indicated notation for the images of $x$ under the maps shown.
When the context is clear, we abbreviate $m_1(x)\otimes\ldots\otimes m_k(x)$ by $\mbar(x)$.  (However, we should be careful to note that $\mbar$ is not 
the name of any function.)

Let $h_k:B\rightarrow M^k\otimes M_0$ be given by 
\[ h_k(x) = m_1(x)\otimes\ldots\otimes m_k(x)\otimes (0,0).\]
Note that $h_k$ is not a short map.  The idea is that as $k$ increases, the distances between elements of $M^k\otimes B$ (and indeed, $M^k\otimes M_0$) depend less and less on the element of $B$ (or $M_0$) and more on $m_1\otimes\ldots\otimes m_k$, so we will use these $h_k$'s to approximate $h$,
the main map in this section.
Even though each $h_k$ is not short, 
we do have an approximate notion of shortness which it satisfies.

\begin{definition} A map $f:X\rightarrow Y$ is \emph{$\epsilon$-short} if for $x,y\in X$, $$d_Y(f(x),f(y))\leq d_X(x,y) + \epsilon.$$\end{definition}

\begin{lemma}\label{hkepsilonshort} $h_k:B\rightarrow M^k\otimes M_0$ 
is $\frac{4}{3^k}$-short. \end{lemma}

\begin{proof}
    Let $x,y\in B$ be given, and for ease of notation, let $\overline{m}(x) = m_1(x)\otimes\ldots\otimes m_k(x)$, $\overline{m}(y) = m_1(y)\otimes\ldots\otimes m_k(y)$, $x' = b_k(x)$, and $y'=b_k(y)$.  
We have:
\[
    \begin{array}{clr}
 &  d_{M^k\otimes M_0}(h_k(x),h_k(y)) 
   \\ = & 
     d_{M^k\otimes M_0}(\overline{m}(x)\otimes (0,0),\overline{m}(y)\otimes(0,0)) & (1) \\
= & d_{M^k\otimes U_0}(\overline{m}(x)\otimes (0,0),\overline{m}(y)\otimes (0,0)) & (2)\\    
  \leq & d_{M^k\otimes U_0}(\overline{m}(x)\otimes (0,0),\overline{m}(x)\otimes (\iota_B\circ\beta)^\dag(x')) & (3) \\
 & +d_{M^k\otimes U_0}(\overline{m}(x)\otimes (\iota_B\circ\beta)^\dag(x'),\overline{m}(y)\otimes (\iota_B\circ\beta)^\dag(y')) & \\
  & +d_{M^k\otimes U_0}(\overline{m}(y)\otimes (\iota_B\circ\beta)^\dag(y'), \overline{m}(y)\otimes (0,0)) & \\
  \leq & d_{M^k\otimes U_0}(\overline{m}(x)\otimes (\iota_B\circ\beta)^\dag(x'),\overline{m}(y)\otimes (\iota_B\circ\beta)^\dag(y')) +\frac{4}{3^k} & (4)\\
   \leq& d_{M^k\otimes B}(\overline{m}(x)\otimes x', \overline{m}(y)\otimes y') +\frac{4}{3^k} & (5) \\
     \leq & d_B(x,y) +\frac{4}{3^k}. & (6)\\
 \end{array}
 \]
Equality (1) is by the definition of the maps $h_k$ and the values $\mbar(x)$ and $\mbar(y)$.
(2) is by Corollary~\ref{cornersM0toU0}. 
 (3) is by the triangle inequality.
 (4) is by Corollary~\ref{distanceinonecopyM}.  That is, for a fixed $m^*\in M^k$, 
$d_{M^k\otimes U_0}(m^*\otimes u, m^*\otimes v) <\frac{2}{3^k}$ for all $u, v\in U_0$.
 In particular,
 \[ d_{M^k\otimes U_0}(\overline{m}(x)\otimes (0,0),
 \overline{m}(x)\otimes (\iota_B\circ\beta)^\dag(x'))\leq \frac{2}{3^k},\] and similarly for $y$.  
(5) follows from the fact that  $(\iota_B\circ\beta)^\dag: B\rightarrow U_0$ is a short map, which 
 implies that $M^k\otimes (\iota_B\circ\beta)^\dag:M^k\otimes B\rightarrow U_0$ is also a short map.
Finally, 
 (6) is because $(M^{k-1}\otimes!)\circ\ldots\circ!$ is a short map,
 and because (as indicated in our diagram below (\ref{eq-bk})),
 $(M^{k-1}\otimes !)\circ\ldots\circ !(x) = \overline{m}(x)\otimes x'$ 
 (and similarly for $y$).
 \end{proof}

\begin{lemma} Let $x\in B$ be given.  $[h_k(x)]_k$ is a Cauchy sequence in $G$, the initial $(M\otimes -)$-algebra.  
\end{lemma}

\begin{proof}
Let $\epsilon>0$ be given and choose $K$ sufficiently large so that $\frac{2}{3^K}<\epsilon$.  Let $k,j>K$ be given, and suppose $k>j$. 
We  use $\mbar\otimes (0,0)$ as an abbreviation for  $ m_1(x)\otimes\ldots\otimes m_k(x)\otimes (0,0)$,
and $\overline{\mbar}\otimes (0,0)$ as an abbreviation for  $ m_1(x)\otimes\ldots\otimes m_j(x)\otimes (0,0)$.

With this notation,
\[
\begin{array}{lcl} h_k(x)  & = &  \mbar\otimes (0,0)\in M^k\otimes M_0,\\
h_j(x) & = &  \overline{\mbar} \otimes (0,0)\in M^j\otimes M_0.
\end{array}
\]
Since $\beta(0,0)=(0,0)\otimes (0,0)$, we see that  
\[ (M^{k-1}\otimes \beta)\circ\ldots\circ \beta(h_j(b)) = \overline{\mbar}\otimes \overbrace{(0,0)\otimes\ldots\otimes (0,0)}^{k-j+1},\]
This belongs to the equivalence class $[h_j(b)]$ in $G$.  
So since $d_G$ is the infimum of distances between representatives coming from the sets $M^k\otimes M_0$,

\[\begin{array}{cl}
& d_G([h_k(x)],[h_j(x)])\\
\leq &d_{M^k\otimes M_0}(\mbar\otimes (0,0),\overline{\mbar}\otimes (0,0)\otimes\ldots\otimes (0,0))\\
\leq &\frac{2}{3^j} \\
\leq & \frac{2}{3^K}.\\
\end{array}
\] 
We are using  Corollary~\ref{distanceinonecopyM}. 
\end{proof}

Since $Q$ is the completion of $G$, we can define
  $h:B\rightarrow Q$ by letting $h(x)$ be the limit of the Cauchy sequence $[i(h_k(x))]_k$.

\begin{proposition} $h:B\rightarrow Q$ is a short map. 
\end{proposition}

\begin{proof}
Let $x,y\in B$ be given.  For ease of notation, let $m_i=m_i(x)$, $x_i = b_i(x)$, $n_i=m_i(y)$ and $y_i=b_i(y)$.  That is, for all $k$, 
\[
\begin{array}{lcl}
h_k(x)& = & m_1\otimes\ldots\otimes m_k\otimes (0,0)
\\ 
(M^{k-1}\otimes \beta)\circ\ldots\circ\beta(x)  & = & m_1\otimes\ldots\otimes m_k\otimes x_k\\
\end{array}
\]
We have similar equations for $y$, but using 
the elements $n_i\in M$ instead of $m_i$.

Let $\epsilon>0$ be given.  Our aim is to show that $d_B(x,y)+\epsilon\geq d_Q(h(x),h(y))$.  This, for all $\epsilon>0$
will yield our result.
Choose $k$ sufficiently large so that 
\begin{equation}\label{k1}\frac{4}{3^k}<\frac{\epsilon}{2}\end{equation}
and 
\begin{equation}\label{k2}|d_G(h_k(x),h_k(y))-d_Q(h(x),h(y))|<\frac{\epsilon}{2}.\end{equation}
This is possible, since $h(x)$ and $h(y)$ are limits of the sequences $[h_k(x)]_k$ and $[h_k(y)]_k$ respectively. 
Then
\[
\begin{array}{rcll}
    d_Q(h(x),h(y)) & \leq & d_G(h_k(x),h_k(y)) + \frac{\epsilon}{2}& \mathrm{by\ } (\ref{k2}) \\
    & \leq & d_B(x,y) + \frac{4}{3^k} + \frac{\epsilon}{2} & \mathrm{by\ Lemma\ }\ref{hkepsilonshort} \\ 
    & \leq & d_B(x,y) + \epsilon & \mathrm{by\ } (\ref{k1}) \\ 
\end{array}
\]
as required.
\end{proof}

\begin{lemma}\label{lemma-morphism-to-Q} $h:B\rightarrow Q$ is a morphism in $\SquaMS$. \end{lemma}

\begin{proof}
    Since we know that $h$ is a short map, it only remains to show that it preserves $S_B$ to see that it is a $\SquaMS$ morphism.  

    Let $(r,s)\in M_0$ be given, and first note that
    
    \begin{equation}\label{SQSG} S_Q((r,s)) = i(S_G((r,s))) = i([S_{M^k\otimes M_0}((r,s))])\end{equation}  for all $k$, where $i:G\hookrightarrow Q$ is the inclusion in (\ref{GinQ}), since the morphisms $M^k\otimes !$ preserve $M_0$.
    
    Let $m_i = m_i(S_B((r,s)))\in M$ and $x_i = b_i(S_B((r,s)))\in B$. 
    (Here $b_i$ is from (\ref{eq-bk}), with $i$ for $k$.)
For all $k$, 
\[ ((M^{k-1}\otimes \beta)\circ\ldots\circ\beta)(S_B((r,s))) = m_1\otimes\ldots\otimes m_k\otimes x_k.\]  
In particular, since $(M^{k-1}\otimes \beta)\circ\ldots\circ\beta$ is a $\SquaMS$ morphism, we have that $x_k = S_B((r_k,s_k))$ for some $(r_k,s_k)\in M_0$.  

We also have $$ h_k(S_B((r,s)) = m_1\otimes\ldots\otimes m_k\otimes (0,0).$$

Next we need to show that $S_{M^k\otimes M_0}((r,s)) = m_1\otimes\ldots\otimes m_k\otimes (r_k,s_k)$.   Note that the following diagram commutes: 
    \[  
   \begin{tikzcd}
     M_0 \arrow{r}{!} \arrow{d}{S_B} &   M \otimes M_0 \arrow{d}{M\otimes S_B} \arrow{r}{M\otimes !} & M^2\otimes M_0 \arrow{r}{M^2\otimes !} \arrow{d}{M^2\otimes S_B} & \cdots \arrow{r}{M^{k-1}\otimes !}& M^k\otimes M_0 \arrow{d}{M^k\otimes S_B} \arrow{r}{M^k\otimes !} & \cdots 
\\
B \arrow{r}{\beta} & M\otimes B \arrow{r}{M\otimes\beta} & M^2\otimes B \arrow{r}{M^2\otimes \beta} &\cdots \arrow{r}{M^{k-1}\otimes\beta}& M^k\otimes B \arrow{r}{M^k\otimes \beta}& \cdots
 \end{tikzcd}
 \]
 Let
 $n_0$, $\ldots$, $n_k\in M$ and
 $ (t_k,u_k)\in M_0$ be such that    
    \begin{equation}\label{eq:newlylabeled} n_1\otimes\ldots\otimes n_k\otimes (t_k,u_k) = S_{M^k\otimes M_0}((r,s)),\end{equation}
    and note that this is equal to $(M^{k-1}\otimes !)
    \circ\ldots\circ !((r,s))$.
We would 
get the same result by starting with $(r,s)$ in $M_0$ and going across the top of the diagram and then down to
$M^k\otimes  B$ via $M^k\otimes S_B$, or by going down to $B$ via $S_B$ first and then across the bottom of the diagram.
Thus, we have 
    
   \begin{equation}\label{samemonborder}
    \begin{array}{rcl}
    M^k\otimes S_B(n_1\otimes\ldots\otimes n_k\otimes (t_k,u_k)) & = &n_1\otimes \ldots\otimes n_k \otimes S_B((t_k,u_k))\\
    &= &m_1\otimes \ldots\otimes m_k\otimes S_B((r_k,s_k)).\\
    \end{array}
    \end{equation}

    So these must be equivalent under $E$.  Since $E$ does not depend on $B$, we must also have $$n_1\otimes\ldots\otimes n_k\otimes S_{M_0}((t_k,u_k)) = m_1\otimes \ldots\otimes m_k\otimes S_{M_0}((r_k,s_k)).$$
    
   Thus, $S_{M^k\otimes M_0}((r,s)) = m_1\otimes\ldots\otimes m_k\otimes (r_k,s_k)$.

 Now we will show that for all $\epsilon>0$, 
 \[ d_Q(h(S_B((r,s))),S_Q((r,s))) < \epsilon,\] 
 and this gives our result.
    Let $\epsilon>0$ be given and choose $k$ sufficiently large so that $\frac{2}{3^k}<
    \frac{\epsilon}{2}$ and \begin{equation}\label{closetothelimit} d_Q(h(S_B((r,s))),i([h_k(S_B((r,s)))]))<\frac{\epsilon}{2}.\end{equation}
For this $k$,
    \[
    \begin{array}{cll}
 &   d_Q(h(S_B((r,s))),S_Q((r,s)))\\
   \leq & d_Q(h(S_B((r,s))),i([h_k(S_B((r,s)))]))  
     +d_Q(i([h_k(S_B((r,s)))]),S_Q((r,s))) & (1) \\
      \leq &   d_G([h_k(S_B((r,s)))],S_G((r,s)))+\frac{\epsilon}{2} & (2)\\
     \leq &   d_G([h_k(S_B((r,s)))],[S_{M^k\otimes M_0}((r,s))])+\frac{\epsilon}{2} & (3)\\
     \leq & d_{M^k\otimes M_0}(h_k(S_B((r,s))),S_{M^k\otimes M_0}((r,s))) + \frac{\epsilon}{2} & \\
      = & d_{M^k\otimes M_0}(m_1\otimes\ldots\otimes m_k\otimes (0,0),m_1\otimes\ldots\otimes m_k\otimes (r_k,s_k)) +\frac{\epsilon}{2}& (4)\\
      \leq & \frac{2}{3^k} + \frac{\epsilon}{2} & (5)\\ 
      < & \epsilon & \\     
    \end{array}
    \]
(1) is by the triangle inequality.  (2) and (3) are by (\ref{closetothelimit}),  (\ref{SQSG}), and the fact that $i:G\hookrightarrow Q$ is an isometric embedding.  (4) is by (\ref{eq:newlylabeled}) and (\ref{samemonborder}). (5) is by Corollary~\ref{distanceinonecopyM}.

Thus, $h(S_B((r,s))) = S_Q((r,s))$ for all $(r,s)\in M_0$.  So $h$ is a $\SquaMS$ morphism.  
\end{proof}

\begin{theorem} 
\label{theorem-final-M}
$(Q,\gamma\colon Q\to M\otimes  Q)$ is the final $M\otimes-:\SquaMS\rightarrow \SquaMS$ coalgebra.
\end{theorem}

\proof
The proof is the same as that of Theorem~\ref{theorem-final-N},
except that Lemma~\ref{lemma-morphism-to-Q} is used to show that every coalgebra 
has a morphism into $Q$, instead of Lemmas~\ref{lemma-short}
and~\ref{lemma-HCP}.
\endproof

By the same proof as in Corollary~\ref{corollary-SC}, we get the following.

\begin{corollary}\label{corollary-Q}
$(Q,\gamma:Q\rightarrow M\otimes Q)$ is the final $M\otimes-:\SquaSet\rightarrow\SquaSet$ coalgebra. 
\end{corollary}

\section{Bilipschitz equivalence}
\label{section-bilipschitz}

Our concluding task 
in this paper
is to show that even though the Sierpinski carpet $\mathbbm{S}$ is not isomorphic to
$(Q,\gamma)$, the
final $(M\otimes-)$-coalgebra, the two are  bilipschitz equivalent.
We begin by recalling the definitions.
A function  $f: A\to B$  between metric spaces is  \emph{bilipschitz continuous}
if there is a number
$K\geq 1$ so that 
\[
\oneoverK d_{A}(x,y)\leq
 d_B(f(x),f(y)) \leq K d_A(x,y) \]
for all $x,y\in A$.  
In addition $A$ and $B$ are 
 are \emph{bilipschitz equivalent} if there is a bilipschitz continuous 
 bijection $f: A\to B$.

We remind the reader that the metric on $\mathbbm{S}$ is the metric induced from the 
taxicab metric on $U_0$ (see just above Definition \ref{def-sigmas}).  Recall that, by Proposition \ref{canconsidertaxi}, $\SC$ with the taxicab metric is bilipschitz equivalent to $\SC$ with the Euclidean metric, so we obtain the result by considering $\SC$ with the taxicab metric.

As we have seen in Theorem~\ref{theorem-SC}, $(\SC, \tau:M\otimes \SC\rightarrow \SC)$ is a corecursive algebra for $M\otimes -$ in $\SquaMS$.
By Corollary~\ref{corollary-SC}, $(\SC,\tau^{-1}:\SC\rightarrow M\otimes \SC)$ is 
a final coalgebra in $\SquaSet$, and in particular, it is a coalgebra. 
In addition, since $(Q,\gamma:Q\rightarrow M\otimes Q)$ is a coalgebra, we have a unique 
coalgebra-to-algebra morphism
$\gamma^\dag:Q\rightarrow \SC$.  And since $(Q,\gamma)$ is a final 
$M\otimes -$ coalgebra
(see Corollary~\ref{corollary-Q}),
there is a unique $\SquaSet$ morphism $(\tau^{-1})^\dag: \mathbbm{S}\rightarrow Q$.  
By finality,
\[
\begin{array}{lcl}
(\tau^{-1})^\dag\circ \gamma^\dag & = & \id_Q, \\
\gamma^\dag\circ (\tau^{-1})^\dag & = & \id_{\mathbbm{S}}.\\
\end{array}
\]
 Hence, $\gamma^{\dag}$ is a bijection.  
However, the inverse of $\gamma^{\dag}$ is not a short map, so 
$\gamma^{\dag}$
is not a $\SquaMS$ isomorphism.
We are going to prove that $\gamma^{\dag}$ is a bilipschitz bijection.

Since $\gamma^\dag$ is a short map, we need only 
find $K\geq 1$ such that $\frac{1}{K} d_Q(x,y)\leq d_{\SC}(\gamma^\dag(x),\gamma^\dag(y))$.  
We shall show that $K =2$ works.
To accomplish this, we will first consider maps from $M^k\otimes M_0$ to $U_0$.  
The inclusion $\SC\hookrightarrow U_0$
is an isometric embedding, by our definition
of the metric on $\SC$.   We prefer
to use $U_0$ in most of this section
because it is easier to visualize $M\otimes U_0$
than $M\otimes \SC$.

Recall from (\ref{GinQ}) that 
we also have an isometric embedding
$i:G\hookrightarrow Q$.
So for each $k<\omega$, we have a morphism $\mu_k=\gamma^\dag\circ i\circ g_k:M^k\otimes M_0\rightarrow \SC$, as in the diagram below:
 \[
      \begin{tikzcd}[column sep=.35in]
& &  & M^{k+1}\otimes M_0 
  \arrow{dr}{M\otimes g_k}
 \arrow[bend right =30,swap]{dddlll}{\mu_{k+1}  =  \gamma^\dag\circ i\circ g_{k+1}}
     \arrow{dl}{g_{k+1} }
          \arrow[bend left =30]{dddrrr}{M\otimes \mu_k  = M\otimes(\gamma^\dag\circ i\circ g_k)}
          &
   \\
 && G    \arrow{dl}[swap]{i} \arrow{rr}{\eta^{-1}} 
 & & M\otimes G \arrow{dr}{M\otimes i} \ar{dl}{i_{M\otimes G}}   \\
&  Q    \arrow{dl}{\gamma^\dag}  \arrow[bend right =15]{rrrr}{\gamma}
  \arrow{rr}{C \eta^{-1}} 
&& C(M\otimes G)  \ar{rr}{\rho_G} & & M\otimes Q \arrow{dr}[swap]{M\otimes \gamma^\dag}
  \\
 \SC      
& &  & &  & &  \arrow{llllll}{\tau}   M\otimes \SC   \\
    \end{tikzcd}    
  \] 
The top triangle commutes by the definition of the maps $\eta$ and $g_k$ (see (\ref{GW}) and (\ref{ikmaps})).  
The square below it commutes since $i$ is the component of the natural transformation 
$Id\to C$ which we saw in Lemma~\ref{lemma-Cauchy}.
We set aside for a moment the commutativity of the triangle next to this square.
 The map $\gamma$ was defined in (\ref{eq-gamma})
 to be $\rho_G\o C\eta^{-1}$. 
  The bottom commutes
by definition of $\gamma^{\dag}$.

It remains to consider the triangle in the middle of the figure.
Consider $m\otimes x \in M\otimes G$.  Using our definitions, we have the desired equation
\[
\rho_G(i_{M\otimes G}(m\otimes x)) =\rho_G(m\otimes x, m\otimes x, \ldots) = 
m \otimes (x, x, \ldots) =  (M\otimes i)(m\otimes x).
\]
Thus the triangle commutes.
  The overall figure shows that for every $k$, 
 \begin{equation}
\label{goaround}
\mu_{k+1} = \tau \o (M\otimes \mu_k).
\end{equation}

We will examine the relationship between distances in $M^k\otimes M_0$ and 
between corresponding points 
in $\SC\subset U_0$, and then use this to obtain the result. We start with the following fact about points
in $ M^k\otimes M_0$
whose images under $\mu_k$ are on a horizontal or vertical segment. 

Throughout we will be using the fact that $d_{\SC}$ is the taxicab metric on $\SC$ as a subset of $U_0$ (see (\ref{eq:taxicab})).

\begin{lemma}\label{segmentlemma}
Let $k\geq 0$ and $x,y\in M^k\otimes M_0$ be such that $\mu_k(x)$ and $\mu_k(y)$
share either an $x$-coordinate or a $y$-coordinate.
Then 
\[ d_{M^k\otimes M_0}(x,y) \leq 2d_{\SC}(\mu_k(x),\mu_k(y)).\]  
\end{lemma}

\proof
We will show this for $x,y\in M^k\otimes M_0$ 
such that $\mu_k(x)$ and $\mu_k(y)$ share a $y$-coordinate;
the other case is proved similarly.
So we will show that for all $k\geq 0$, if $x,y\in M^k\otimes M_0$ and $\mu_k(x) = (r,s)$, $\mu_k(y) = (t,s)$ for some $(r,s),(t,s)\in [0,1]^2$, then $d_{M^k\otimes M_0}(x,y) \leq 2d_{\SC}(\mu_k(x),\mu_k(y))$.  
We prove this by induction on $k$.

If $k = 0$, since $\mu_0$ is a $\SquaMS$ morphism, $\mu_0(x) = \mu_0(S_{M_0}(x)) = S_{\SC}(x)=x$ (since $S_{M_0}$ is the identity on $M_0$ and $S_{\SC}$ is the inclusion of $M_0\hookrightarrow\SC$), and similarly, $\mu_0(y) =y$. 
We are going to consider the case $r = 0$ and $t =1$; the 
other cases are either similar or easier. 
So $x = (0,s)$ and $y = (1,s)$. Recall, as in Example \ref{pathmetric}, the distance in $M_0$ is the path metric.  So the distance in $M_0$ from $x$ to $y$ is  
$1 + 2s$ when $s \leq \frac{1}{2}$, and it is $1 + 2(1-s) = 3 - 2s$ when $s \geq \frac{1}{2}$.
In either case, this is $\leq 2$.  By $(\sqtwo)$ $d_{\SC}(x,y)\geq |t-r|+|s-s|=1$, so we have $d_{M_0}(x,y)\leq 2 d_{\SC}(x,y) = 2d_{\SC}(\mu_0(x),\mu_0(y))$.

Now assume the result for $k$ and suppose $x,y\in M^{k+1}\otimes M_0$.  
Let us write $x = m \otimes x'$ and $y= n\otimes y'$,
where $m$ and $n$ belong to $M$, and $x',y'\in M^k\otimes M_0$.
(We emphasize that $n$ denotes an element of $M$, not a number.)
We argue by cases on $m$ and $n$.

Our first case is when $m = n$.  
We thus assume that $\mu_{k+1}(m\otimes x') = (r,s)$ and $\mu_{k+1}(m\otimes y') = (t,s)$.
By (\ref{goaround}), $\tau(m\otimes \mu_k(x')) = (r,s)$ and $\tau(m\otimes \mu_k(y')) = (t,s)$.

Now $\tau$ works the same way as $\alpha_M$ (it is a domain-codomain restriction of $\alpha_M$, see (\ref{eq-tau}) and (\ref{alphaMdef})).
 And so we see easily that $\mu_k(x')$ and $\mu_k(y')$ have the same $y$-coordinate.
So 
\def\arraystretch{1.5}
\[
\begin{array}{lcll}
d_{M^{k+1} \otimes M_0}(x,y) & = & \onethird d_{M^k\otimes M_0}(x',y')\\
& \leq & \onethird \cdot 2d_{\SC}(d(\mu_k(x'),\mu_k(y'))) & \mbox{by induction hypothesis}\\
& = & 2 d_{M\otimes \SC}(m\otimes \mu_k(x'),m\otimes \mu_k(y')) & 
\mbox{by (\ref{eq-tau})}
\\
& = & 2 d_{\SC}(\tau(m\otimes \mu_k(x')),\tau(m\otimes \mu_k(y'))) & \mbox{see below} \\
& = & 2 d_{\SC}(\mu_{k+1}(x), \mu_{k+1}(y))\\
\end{array}
\]
\def\arraystretch{1}
For the ``see below'' line, we use the fact that within a particular copy $m\otimes \SC$,
the restriction of $\tau$ is an isometric embedding.

Indeed, for $z_1,z_2\in \SC$ and $m\in M$,

\[\begin{array}{rcl}
d_{\SC}(\tau(m\otimes z_1),\tau(m\otimes z_2)) & = & d_{\SC}(\frac{1}{3}m + \frac{1}{3}z_1 ,\frac{1}{3}m+\frac{1}{3}z_2)\\ 
&=& \frac{1}{3} d_{\SC}(z_1,z_2)\\
&=& d_{M\otimes \SC}(m\otimes z_1,m\otimes z_2).\\
\end{array}\]

Our second case is when $m$ and $n$ are adjacent squares in $M$.
(For example, we could have $m = (0,0)$, and $n = (1,0)$ or $n = (0,1)$.)
The argument in this case is a small elaboration of what we saw in the 
first case.  Our work below on a more complicated case subsumes this one,
and so we shall pass over this particular case.    The same holds for our
third case, when we have $m = (0,0)$, and $n = (0,2)$, or another pair which
is a rotation or reflection of this one.   The main case which is \emph{not} 
handled is when $m = (0,1)$ and $n=(2,1)$, or some rotation or reflection of this.
In such cases, the shortest path in $M^{k+1}\otimes M_0$ from $x$ to $y$ 
must ``navigate around the central hole.''

Without loss of generality, suppose that $s\geq \frac{1}{2}$ (see below, $s<\frac{1}{2}$ is similar).  Then there is a path in $M^{k+1}\otimes M_0$ from $x$ to $y$ 
of the following form: 
\[
\begin{tikzpicture}[scale=1]
\draw (0,0) node {\sctwoBprime{\sconeBprime}};
\draw (-1.44,.3) node (t) {$\bullet$};
\draw (-.66,.3) node (tt) {$\bullet$};
\draw (.66,.3) node (uu) {$\bullet$};
\draw (1.37,.3) node (u) {$\bullet$};
\draw (-4,.3) node (p) {$x = (0,1)\otimes x'$};
\draw (4,.3) node  (q) {$y = (2,1)\otimes y'$};
\draw (-3,-3) node  (r) {$(0,1)\otimes v_1$};
\draw (2,-3) node  (s) {$(2,1)\otimes v_2$};
\path[->](p) edge [bend left]  (t);
\path[->](q) edge [bend left]  (u);
\path[->](r) edge [bend left]  (tt);
\path[->](s) edge [bend right]  (uu);
\draw [line width=2pt]  (-1.44,.3) -- (-.66,.3);
\draw [line width=2pt]  (-.66,.3) --  (-.66,.66);
\draw [line width=2pt]  (-.66,.66) --  (.66,.66);
\draw [line width=2pt]  (.66,.66) --  (.66,.3);
\path [line width=2pt] (1.37,.3) edge (.66,.3);
\end{tikzpicture}
\]
where $v_1=S_{M^k\otimes M_0}((1,3s-1))$ and $v_2=S_{M^k\otimes M_0}((0,3s-1))$.  Then
$\mu_{k+1}((0,1)\otimes v_1) = (\frac{1}{3},s)$ and $\mu_{k+1}((2,1)\otimes v_2) = (\frac{2}{3},s)$. 

The picture suggests going around the top of the middle square:
This is because we assume $s \geq \frac{1}{2}$.
(If $s<\frac{1}{2}$, then we get an analogous shorter path going around the bottom of the middle square.)  Then since the distance in $M^{k+1}\otimes M_0$ is the score of the shortest path, we have 
\def\arraystretch{1.5}
\[
\begin{array}{rcl}
  d_{M^{k+1}\otimes M_0}(x,y) & \leq & 
 \onethird d_{M^k\otimes M_0}(x',v_1) + (\frac{2}{3} - s) + \frac{1}{3}  + (\frac{2}{3}-s) + \onethird d_{M^k\otimes M_0}(v_2,y')
\\ 
& \leq & \onethird ( d_{M^k\otimes M_0}(x',v_1) + 2 + d_{M^k\otimes M_0}(v_2,y'))\\
\end{array}
\]
\def\arraystretch{1}
since $s\geq \frac{1}{2}$.  

Let $(r',s') = \mu_k(x')$ and note that $\mu_k(v_1) = (1,s')$.  By the induction hypothesis, $d_{M^k\otimes M_0}(x',v_1)\leq 2(1-r')$.  Further note that 
\[(r,s) = (\onethird(r'),\onethird(1+s')).\]  Thus, $d_{M^k\otimes M_0}(x',v_1)\leq 2(1-3r)$.
Similarly, $d_{M^k\otimes M_0}(v_2,y')\leq 2(3t-2)$. 
So, using our calculation above, 
\[
\begin{array}{rcl}
  d_{M^{k+1}\otimes M_0}(x,y) &\leq& \onethird (2(1-3r) + 2 + 2(3t-2))\\
& =& 2(t-r)\\
& =& 2d_{\SC}((r,s),(t,s))\\
& =& 2d_{\SC}(\mu_{k+1}(x),\mu_{k+1}(y))\\
\end{array}
\]
as required. 

This covers all of the possible cases in which
 $\mu_k(x)$ and $\mu_k(y)$
share  a $y$-coordinate. 
\endproof

Next, we need to show that the distance between the images $\gamma^\dag(x)$ and
$\gamma^\dag(y)$ can be calculated as the sum of horizontal and vertical segments between endpoints \emph{in the image of} $\mu_k$.  This is what will allow us to compare the distance in $M^k\otimes M_0$ to the distance in   $\SC$.

\begin{lemma}
\label{Aug6}
For $k\geq 0$, given $x,y\in M^k\otimes M_0$, the distance between $\mu_k(x)$ and $\mu_k(y)$ in $\SC$ is the sum of the lengths of at most four horizontal or vertical segments whose endpoints are in the image of $\mu_k$. 
\end{lemma}

\proof
Let $x,y\in M^k\otimes M_0$ be given, and let $\mu_k(x) = (r,s)$, $\mu_k(y) = (t,u)$. Without loss of generality, suppose 
that $r\leq t$ and $s\leq u$ (the other cases are similar).  Consider the point $(t,s)$.

Case 1: $(t,s)$ is in the image
of $M^k\otimes M_0$ under $\mu_k: M^k\otimes M_0\rightarrow \SC$, let $z\in M^k\otimes M_0$ be such that $\mu_k(z) = (t,s)$.  Then
\[ d_{\SC}(\mu_k(x),\mu_k(y)) = d_{\SC}(\mu_k(x),\mu_k(z))+d_{\SC}(\mu_k(z),\mu_k(y)).
\]
In this case, we are done.

Case 2: $(t,s)$ is not in the image of $M^k\otimes M_0$ under $\mu_k$.  That is, $(t,s)$ occurs in a ``hole'' which we will need to navigate around. 
Again, we are restricting our attention to the case when $r\leq t$ and $s\leq u$ (the other cases are analogous).







\begin{claim}\label{case2claim} For every $k\geq 0$, if $x,y\in M^k\otimes M_0$ and $\mu_k(x)=(r,s)$ and $\mu_k(y) = (t,u)$ and $(t,s)$ is not in the image of $\mu_k$, then there exist $z,z_1,z_2\in M^k\otimes M_0$ such that 
\begin{equation}\label{eqAug6}
\begin{array}{lcll}
d_{\SC}(\mu_k(x),\mu_k(y)) & = & 
& d_{M^k\otimes M_0}(\mu_k(x),\mu_k(z_1)) + d_{M^k\otimes M_0}(\mu_k(z_1),\mu_k(z)) \\
& & + &d_{M^k\otimes M_0}(\mu_k(z),\mu_k(z_2)) + d_{M^k\otimes M_0}(\mu_k(z_2),\mu_k(y)) \\
\end{array}
\end{equation}
\end{claim}

The idea is indicated in the picture below (which may not be to scale, the ``hole'' may be much smaller and off to one side).
The points  $z,z_1,z_2\in M^k\otimes M_0$ are such that $\mu_k(z_1) = (v_1,s)$, $\mu_k(z) = (v_1,v_2)$, and $\mu_k(z_2) = (t,v_2)$.
\[
\begin{tikzpicture}[scale=1]
\draw (-1,-1) -- (1,-1);
\draw (1,-1) -- (1,1);
\draw (1,1) -- (-1,1);
\draw (-1,1) -- (-1,-1);
\draw (-3,-3) -- (3,-3);
\draw (3,-3) -- (3,3);
\draw (3,3) -- (-3,3);
\draw (-3,3) -- (-3,-3);
\draw (-2,0) node (x1y1) {$\bullet$};
\draw (-2,-0.4) node (x1y1tag) {$(r,s)$};
\draw (-1,0) node (rs) {$\bullet$};
\draw (-2.5,.7) node (rstag) {$(v_1,s)$};
\draw (-1,1) node (ru) {$\bullet$};
\draw (-1.4,1.4) node (rutag) {$(v_1,v_2)$};
\draw (0,1) node (tu) {$\bullet$};
\draw (2,1.3) node (tutag) {$(t,v_2)$};
\draw (0,2) node (x2y2) {$\bullet$};
\draw (0,2.4) node (x2y2tag) {$(t,u)$};
\draw (0,0) node (x2y1) {$\bullet$};
\draw (.5,-2.4) node (x2y1tag) {$(t,s)$};
\draw [line width=2pt] (-2,-0) --  (-1,0);
\draw [line width=2pt] (-1,0) --  (-1,1);
\draw [line width=2pt] (-1,1) --  (0,1);
\draw [line width=2pt] (0,1) --  (0,2);
\path[->](rstag) edge [bend left]  (rs);
\path[->](tutag) edge [bend right]  (tu);
\path[->](x2y1tag) edge [bend right]  (x2y1);
\end{tikzpicture}
\]

Here is the relation of the picture to (\ref{eqAug6}).
On the right of (\ref{eqAug6}), 
each term comes from a
 horizontal or vertical segment in $U_0$.  In particular, $\mu_k(x)$ and $\mu_k(z_1)$ share $y$-coordinates, $\mu_k(z_1)$ and $\mu_k(z)$ share $x$-coordinates, $\mu_k(z)$ and $\mu_k(z_2)$ share $y$-coordinates, and $\mu_k(z_2)$ and $\mu_k(y)$ share $x$-coordinates.

Now we prove the claim by induction on $k$.
When $k=0$, we must have $(r,s),(t,u) \in M_0$, so since we have $r\leq t$ and $s\leq u$, the only case in which $(t,s)\notin M_0$ is if $r=0$ and $u=1$.  But in this case, we can let $z=z_1=z_2 = (0,1)$.

Assume the claim for some fixed $k\geq 0$ and let $x,y\in M^{k+1}\otimes M_0$.  
We will consider two cases for $(t,s)$:  when it appears in the center ``hole'', that is, in $(\frac{1}{3},\frac{2}{3})\times (\frac{1}{3},\frac{2}{3})$, and when it does not. 

First suppose it does not.  We will consider the particular case when $(t,s)\in [\frac{2}{3},1]\times [0,\frac{1}{3}]$, the bottom right corner.  The rest of the cases are similar.  

If $\mu_{k+1}(x)$ is also in this corner, let $x'$ be such that $x = (2,0)\otimes x'$.  Otherwise,
let $x' = S_{M^k\otimes M_0}((0,3s))$.  Similarly, if $\mu_{k+1}(y)$ is in this bottom right corner, let $y' $ be such that $y=(2,0) \otimes y'$.  Otherwise let $y' = S_{M^k\otimes M_0}((3t-2,1))$.  
Note that $(3t-2,3s)$ is not in the image of $\mu_k$, or else we could have $z$ such that $\mu_k(z) = (3t-2,3s)$, and thus, we would have 
\[ \mu_{k+1}((2,0)\otimes z) = \tau\circ M\otimes \mu_k((2,0)\otimes z) = \tau((2,0)\otimes (3t-2,3s)) = (t,s),\]
 a contradiction to our assumption.  So by the induction hypothesis, there are $z_1,z,z_2\in M^k\otimes M_0$ such that $\mu_k(z_1)=(v_1,3s)$, $\mu_k(z) = (v_1,v_2)$, and $\mu_k(z_2) = (3t-2,v_2)$.  
Then 
\[\begin{array}{lcl}
 \mu_{k+1}((2,0)\otimes z_1) &  = & (\onethird (2+v_1),s)\\
 \mu_{k+1}((2,0)\otimes z) & =& (\onethird (2+v_1),\onethird (v_2))\\
\mu_{k+1}((2,0)\otimes z_2) & = & (t,\onethird (v_2))
\end{array}
\]  
These are as required since the successive segments they determine are horizontal or vertical; see the picture above.

Finally, suppose $(t,s)$ is in $(\frac{1}{3},\frac{2}{3})\times (\frac{1}{3},\frac{2}{3})$. Then we must have $\mu_{k+1}(x)\in [0,\frac{1}{3}]\times [\frac{1}{3},\frac{2}{3}]$ and $\mu_{k+1}(y)\in [\frac{1}{3},\frac{2}{3}]\times [\frac{2}{3},1]$.  
Let 
\[\begin{array}{lcl}
z_1   & =  & (0,1)\otimes S_{M^k\otimes M_0}((1,3t-1))\\
z & = & (0,1)\otimes S_{M^k\otimes M_0}((1,1))\\
z_2 & = &  (1,2)\otimes S_{M^k\otimes M_0}((3t-1,0))\\
\end{array}
\]
 Then $\mu_{k+1}(z_1) = (\frac{1}{3},s)$, $\mu_{k+1}(z) = (\frac{1}{3},\frac{2}{3})$, and $\mu_{k+1}(z_2) = (t,\frac{2}{3})$.
 These again are as required in our claim.

This concludes our induction proof of the claim.  
Applying it,  we can express the distance
 $d_{\SC}(\mu_k(x),\mu_k(y))$
 as  the sum of the lengths of at most $4$ horizontal  and  vertical segments with endpoints in the image of $\mu_k$. 
 \endproof

Putting the last lemmas in this section together, we get the following:

\begin{proposition} 
\label{prop-3-bilipschitz}
For $k\geq 0$ and $x,y\in M^k\otimes M_0$, 
\[
 d_{M^k\otimes M_0}(x,y)\leq 2 d_{\SC}(\mu_k(x),\mu_k(y)).
\]
\end{proposition}

\proof
Let $x,y\in M^k\otimes M_0$, and let $\mu_k(x) = (r,s)$ and $\mu_k(y) = (t,u)$.  As in Lemma~\ref{Aug6}, assume without loss of generality that $r\leq t$ and $s\leq u$ (the other cases are similar).  

If $(t,s)$ is in the image of $M^k\otimes M_0$ under $\mu_k$ (as in Case 1 in the proof of Lemma~\ref{Aug6}), let $z$ be such that $\mu_k(z) = (t,s)$, and let $z_1=z_2=z$. 

Otherwise, if $(t,s)$ is not in the image of $M^k\otimes M_0$ under $\mu_k$, let $z_1,z,z_2\in M^k\otimes M_0$
be as in Claim~\ref{case2claim} of  Case 2 in the proof of Lemma~\ref{Aug6}.
Then in either case, \[ d_{M^k\otimes M_0}(x,y) \leq d_{M^k\otimes M_0}(x,z_1) + d_{M^k\otimes M_0}(z_1,z) + d_{M^k\otimes M_0}(z,z_2) + d_{M^k\otimes M_0}(z_2,y).
\]
We use
the fact that these are each horizontal or vertical segments in $\SC$, and also Lemmas~\ref{segmentlemma}
and~\ref{Aug6} to see that
\[\begin{array}{rl}
& d_{M^k\otimes M_0}(x,z_1) + d_{M^k\otimes M_0}(z_1,z) + d_{M^k\otimes M_0}(z,z_2) + d_{M^k\otimes M_0}(z_2,y)\\
\leq & 2(d_{\SC}(\mu_k(x),\mu_k(z_1)) + d_{\SC}(\mu_k(z_1),\mu_k(z)) + d_{\SC}(\mu_k(z),\mu_k(z_2)) + d_{\SC}(\mu_k(z_2),\mu_k(y)))\\
= &2d_{\SC}(\mu_k(x),\mu_k(y))\\
\end{array}
\]
\endproof

Next, we need 
a version of 
Proposition~\ref{prop-3-bilipschitz}
for $G$.  This comes almost immediately from the fact that for any $x\in G$, there exists $k$ and $x'\in M^k\otimes M_0$ such that $g_k(x') = [x'] = x$.  

\begin{proposition}\label{GSC}
For $x,y\in G$, $$d_G(x,y)\leq 2d_{\SC}(\gamma^\dag\circ i(x),\gamma^\dag \circ i(y)).$$
\end{proposition}

\proof
Let $x,y\in G$ and 
$\epsilon>0$ be given.  Choose $k\geq 0$ sufficiently large 
and $x',y'\in M^k\otimes M_0$ with $[x']=x$ and $[y']=y$, and $|d_G(x,y)-d_{M^k\otimes M_0}(x',y')|<\frac{\epsilon}{2}$.  Then $\gamma^\dag\circ i(x) = \gamma^\dag\circ i\circ g_k(x') = \mu_k(x')$ and similarly for $y$. 
Hence
\[\begin{array}{rcll}
    d_G(x,y) & \leq & d_{M^k\otimes M_0}(x',y') + \frac{\epsilon}{2}  \\
     & \leq &  2(d_{\SC}(\mu_k(x'),\mu_k(y')) + \frac{\epsilon}{2}) 
     & \mbox{by Proposition~\ref{prop-3-bilipschitz}} \\
     & = & 2 d_{\SC}(\gamma^\dag\circ i(x),\gamma^\dag\circ i(y)) + \epsilon\\
\end{array}\]
Since $\epsilon>0$ was arbitrary, we get the required inequality. 
\endproof

Finally, we will use the following very general fact along with 
the fact that $G$ is dense in $Q$ to get our result. 

\begin{proposition} \label{prop-density}
Let $f:A\rightarrow B$ be a
 short map between metric spaces and let $D$ be dense in $A$.  If $K\geq 1$ is such that 
 \[ d_A(x,y)\leq Kd_B(f(x),f(y))\]
 for all $x,y\in D$, then
 this same inequality holds for all $x,y\in A$.
\end{proposition}
\proof
Let $\epsilon>0$ and $x,y\in A$ be given.  Choose $x',y'\in D$ such that $d_A(x,x')<\frac{\epsilon}{4K}$ and $d_A(y,y')<\frac{\epsilon}{4K}$.  
Since  $f$ is a short map,
\[\begin{array}{rcll}
d_A(x,y) & \leq & d_A(x',y') + 2(\frac{\epsilon}{4K}) \\
& \leq & Kd_B(f(x'),f(y')) + \frac{\epsilon}{2K}\\
& \leq & K ( d_B(f(x),f(y)) + 2(\frac{\epsilon}{4K})) + \frac{\epsilon}{2K} & \mbox{see below}\\
& =& Kd_B(f(x),f(y)) + \frac{\epsilon}{2} + \frac{\epsilon}{2K}\\
& \leq & Kd_B(f(x),f(y)) + \epsilon\\
\end{array}\]
In the line marked ``see below'', we use the fact that $f$ is short to see that 
  $d_B(f(x),f(x')) <\frac{\epsilon}{4K}$ and similarly for $y$. 
Since $\epsilon>0$ was arbitrary, $d_A(x,y)\leq Kd_B(f(x),f(y))$ for all $x,y\in A$, as required. 
\endproof

\begin{theorem}\label{theorem-bilipschitz}
The metric space 
$Q$ is bilipschitz equivalent to the Sierpinski carpet $\mathbbm{S}$ as a subset of the plane with the taxicab metric, and thus, the Euclidean metric.  
\end{theorem}

\proof
We use $\gamma^\dag\colon Q \to \SC$.   This is a short bijection, it has the additional property
that $d_Q(x,y) \leq 2 d_{\SC}(\gamma^\dag(x), \gamma^\dag(y))$.  In this last estimate, we use 
Proposition~\ref{prop-density}, taking $A$ to be $Q$ and the dense set $D$ to be the image of $G$ under the 
isometric embedding $i$.  We also use Proposition~\ref{GSC}.
\endproof

\section{Conclusion}
Stepping back,
the main point of this paper has been to further the interaction between the subject of 
coalgebra broadly considered (including corecursive algebras)
and continuous mathematics.  The questions that we asked in this paper
concerned the relationship between very natural 
and very concrete fractal sets on the one hand,
and more abstract ideas like initial algebras  and final coalgebras
on the other. We came to this work in order to explore these general issues.
What we found in the exploration was a set of ideas connecting category-theoretic
and analytic concepts such as colimits in metric spaces, short maps 
approximated by non-short maps, and corecursive algebras as an alternative
to infinite sums.    We hope that the results in this
paper further these connections.

Here are two general next steps in this line of research.
First, it would be desirable to merge the ideas here with the general
categorical framework for self-similarity developed in Leinster~\cite{lein}.
This would mean taking assumptions on our category $\SquaMS$
(such as ($\sqone$) and ($\sqtwo$)) and also assumptions on the functor
(see Theorem~\ref{quotientmetric}) and incorporating them as additional
assumptions in Leinster's framework, in addition to the
requirements needed there, such as the
non-degeneracy requirements.  In a different direction, one would want to know which aspects
of the classical theory of fractals may be derived from the universal properties
which we have established.

\bibliographystyle{plain}

\bibliography{references}	

\end{document}